\documentclass[11pt,reqno]{amsart}

\usepackage{amsmath, amsthm, amssymb}
\usepackage{amsfonts}

\usepackage{hyperref}

\usepackage[ansinew]{inputenc}
\usepackage[dvips]{epsfig}
\usepackage{graphicx}
\usepackage[english]{babel}

\usepackage{thmtools}
\theoremstyle{plain}
\declaretheorem[title=Theorem, parent=section]{theorem}
\declaretheorem[title=Lemma,sibling=theorem]{lemma}
\declaretheorem[title=Proposition,sibling=theorem]{proposition}
\declaretheorem[title=Corollary,sibling=theorem]{corollary}

\theoremstyle{definition}
\declaretheorem[title=Definition,sibling=theorem]{definition}
\declaretheorem[title=Remark,sibling=theorem]{remark}
\declaretheorem[title=Remark, numbered=no]{remark*}

\declaretheorem[title=Assumption, numbered=no]{assumption*}

\numberwithin{equation}{section}

\usepackage[backgroundcolor=white, bordercolor=blue,
linecolor=blue]{todonotes}

\usepackage{dsfont}
\usepackage{bbm}

\newcommand{\N}{\mathbb{N}}
\newcommand{\R}{\mathbb{R}}

\newcommand{\cP}{\mathcal{P}}

\newcommand{\eps}{\varepsilon}

\newcommand{\1}{\mathbbm{1}}

\DeclareMathOperator{\dist}{dist}

\DeclareMathOperator{\diam}{diam}

\DeclareMathOperator{\dvg}{div}

\DeclareMathOperator*{\osc}{osc}

\renewcommand{\d}{\textnormal{\,d}}

\newcommand{\average}{{\mathchoice {\kern1ex\vcenter{\hrule height.4pt
width 6pt depth0pt} \kern-9.7pt} {\kern1ex\vcenter{\hrule
height.4pt width 4.3pt depth0pt} \kern-7pt} {} {} }}
\newcommand{\dashint}{\average\int}

\begin{document}
\allowdisplaybreaks
 \title{Regularity for nonlocal equations with local Neumann boundary conditions}

\author{Xavier Ros-Oton}
\author{Marvin Weidner}

\address{ICREA, Pg. Llu\'is Companys 23, 08010 Barcelona, Spain \& Universitat de Barcelona, Departament de Matem\`atiques i Inform\`atica, Gran Via de les Corts Catalanes 585, 08007 Barcelona, Spain \& Centre de Recerca Matem\`atica, Barcelona, Spain}
\email{xros@icrea.cat}

\address{Departament de Matem\`atiques i Inform\`atica, Universitat de Barcelona, Gran Via de les Corts Catalanes 585, 08007 Barcelona, Spain}
\email{mweidner@ub.edu}

\keywords{nonlocal, regularity, boundary, large solution, Neumann condition}

\subjclass[2020]{47G20, 35B65, 31B05, 35R53, 35B44}

\allowdisplaybreaks

\begin{abstract}
In this article we establish fine results on the boundary behavior of solutions to nonlocal equations in $C^{k,\gamma}$ domains which satisfy local Neumann conditions on the boundary. Such solutions typically blow up at the boundary like $v \asymp \dist^{s-1}$ and are sometimes called large solutions. In this setup we prove optimal regularity results for the quotients $v/\dist^{s-1}$, depending on the regularity of the domain and on the data of the problem. The results of this article will be important in a forthcoming work on nonlocal free boundary problems.
\end{abstract}

\allowdisplaybreaks

\maketitle
\section{Introduction}  

The study of nonlocal operators of the form
\begin{align}
\label{eq:L}
L v(x) = \text{p.v.} \int_{\R^n} \big(v(x) - v(x+h) \big) K(h) \d h,
\end{align}
where $K : \R^n \to [0,\infty]$ is a kernel satisfying for some $s \in (0,1)$
\begin{align}
\label{eq:K}
K(h) = \frac{K(h/|h|)}{|h|^{n+2s}}, \qquad 0 < \lambda \le K(\theta) \le \Lambda ~~ \forall \theta \in \mathbb{S}^{n-1}, \qquad K(h) = K(-h)
\end{align}
has been an important area of research in analysis and probability for the past 30 years. Operators $L$ of the type \eqref{eq:L}-\eqref{eq:K} arise naturally as generators of $2s$-stable L\'evy processes, and are used to model different kinds of real-world phenomena involving long range interactions, e.g. in mathematical finance and in physics. From a PDE perspective, it is of particular interest to study the effect of the nonlocality of $L$ on the regularity of solutions to nonlocal equations. By now, the question of \emph{interior} regularity of solutions is fairly well-understood, and there are several important works in this context, such as  \cite{CaSi09,CaSi11a,CaSi11b}, \cite{Sil06}, \cite{BaLe02}, \cite{Kas09}, \cite{DKP14,DKP16}, \cite{BFV14}, \cite{RoSe16b}.

A much more delicate question is the one of \emph{boundary} regularity of solutions to nonlocal problems. Previous works have mostly focused on nonlocal Poisson problems, given as follows
\begin{align}
\label{eq:Dirichlet-problem-weak}
\left\{\begin{array}{rcl}
L v &=& f ~~ \text{ in } \Omega,\\
v &=& 0 ~~ \text{ in } \R^n \setminus \Omega.
\end{array}\right.
\end{align}
The nonlocal Poisson problem \eqref{eq:Dirichlet-problem-weak} arises naturally as the Euler-Lagrange equation of a nonlocal energy minimization problem and can therefore be studied via variational methods, but also via non-variational methods. For \eqref{eq:Dirichlet-problem-weak} it was proved (see \cite{RoSe14}, \cite{Gru15}) that weak solutions satisfy $v \in C^s(\overline{\Omega})$, once $\partial \Omega \in C^{1,\gamma}$ and $f \in L^{\infty}(\Omega)$. The $C^s$ regularity of solutions is optimal, as one can see from the following explicit example (see \cite{Get61}, \cite{Lan72}, \cite{Dyd12}):
\begin{align}
\label{eq:usual-example}
(-\Delta)^s (1-|x|^2)_+^s = c_{n,s} > 0 ~~ \text{ in } B_1,
\end{align}
which also remains valid for $L$ satisfying \eqref{eq:L}-\eqref{eq:K} (see \cite{Ros16}).
However, it turns out that once the domain, the kernel, and the data are regular enough, also the quotient $v/d^s$ will be regular, yielding a fine description of the behavior of the solution $v$ at the boundary. The best known result in the literature, establishing optimal boundary regularity of weak solutions of \eqref{eq:Dirichlet-problem-weak} in terms of the regularity of the domain and the data was shown in \cite{RoSe17,AbRo20,Gru15} (see also \cite{RoSe16,RoSe16b,AbGr23}) and reads as follows:
\begin{align}
\label{eq:AbRo-result}
\partial \Omega \in C^{k+1,\gamma}, ~~ f \in C^{k+\gamma-s}(\overline{\Omega}) \quad \Rightarrow \quad \frac{v}{d^s} \in C^{k,\gamma}(\overline{\Omega}) \quad \forall k \in \N \cup \{ 0 \}, ~~ \gamma \in (0,1).
\end{align}

All the previously mentioned results on the nonlocal Poisson problem \eqref{eq:Dirichlet-problem-weak} address weak solutions for which one can  prove that they must remain bounded in $\overline{\Omega}$ (see \cite{SeVa14,KKP16}). However, explicit computations reveal that there also exist pointwise solutions of \eqref{eq:Dirichlet-problem-weak}, which explode at the boundary of the domain behaving asymptotically like $d^{s-1}$. The following most prominent example goes back to a work by Hmissi \cite{Hmi94} (see also \cite[Example 1, p.239]{Bog99}, \cite[Example 3.3]{BBKRSV09}, \cite{Dyd12}):
\begin{align}
\label{eq:Hmissi}
(-\Delta)^s (1 - |x|^2)_+^{s-1} = 0 ~~ \text{ in } B_1.
\end{align}

The example \eqref{eq:Hmissi} has initiated the conceptual study of boundary blow-up for solutions to nonlocal equations (see \cite{Gru14,Gru15,Gru18,Gru23}, \cite{Aba15,Aba17,AGV23}, \cite{CGV21}). In this theory, solutions such as \eqref{eq:Hmissi} are sometimes called ``large solutions''. Note that due to the explosion at the boundary, the above function cannot be a weak solution, and clearly violates \eqref{eq:AbRo-result}. 

In order to have a unified framework which also allows for singular behavior at the boundary, it is necessary to keep track of the boundary behavior of the solution, or more precisely to prescribe somehow the behavior of the quotient $v/d^{s-1}$. In this spirit, the following Neumann problem, which was introduced in \cite{Gru14} (see also \cite{Gru18}, \cite{Gru23}), can be seen as a generalization of \eqref{eq:Dirichlet-problem-weak}
\begin{align}
\label{eq:Neumann-problem-blowup}
\left\{\begin{array}{rcl}
L v &=& f ~~ \text{ in } \Omega,\\
v &=& 0 ~~ \text{ in } \R^n \setminus \Omega,\\ \displaystyle
\partial_{\nu}\left(\frac{v}{d^{s-1}}\right) &=& g ~~ \text{ on } \partial \Omega,
\end{array}\right.
\end{align}
where $\nu(x_0) \in \mathbb{S}^{n-1}$ denotes the inner unit normal at $x_0 \in \partial \Omega$. The problem \eqref{eq:Neumann-problem-blowup} is a natural \emph{nonlocal Neumann problem} with inhomogeneous Neumann data $g$, and one can show that the problem is well-posed in suitable function spaces, at least if the domain is $C^\infty$ (see \cite{Gru14}). Moreover, the solutions blow up at every boundary point where $v/d^{s-1}$ does not vanish. 

\begin{remark}
The functions in \eqref{eq:usual-example} and \eqref{eq:Hmissi}, are both solutions to \eqref{eq:Neumann-problem-blowup}, with $g \equiv 1$ and $f = c_{n,s}$ and with $g \equiv (s-1)2^{s-2}$ and $f = 0$, respectively, in case $\Omega = B_1$.
\end{remark}

Note that the Neumann condition in \eqref{eq:Neumann-problem-blowup} is purely local in nature in the sense that it is imposed only on the topological boundary $\partial \Omega$. Therefore, \eqref{eq:Neumann-problem-blowup} is conceptually completely different from the nonlocal Neumann problem introduced in \cite{DGLZ12}, \cite{DRV17} (see also \cite{AlTo20}, \cite{Von21}, \cite{AFR23}, \cite{FoKa24}, \cite{GrHe24}). It is also of entirely different nature than \cite{BCGJ14,BGJ14}, and \cite{BBC03,ChKi02}, where local boundary conditions are imposed, but instead the operator is changed, depending on the domain.

\subsection{Main result}

The  aforementioned regularity results \eqref{eq:AbRo-result} from \cite{RoSe17,AbRo20} do not apply to \eqref{eq:Neumann-problem-blowup} since solutions are in general not continuous and might even explode at the boundary. However, it is natural to expect fine regularity results for the quotients $v/d^{s-1}$ depending on the regularity of the domain and the data. 

When $\Omega$ is $C^{\infty}$ and $K|_{\mathbb S^{n-1}}$ is $C^{\infty}$, the regularity theory for \eqref{eq:Neumann-problem-blowup} was developed by Grubb in \cite{Gru14} using an approach via pseudodifferential operators. 

Our goal in this work is twofold: to establish sharp boundary regularity estimates for \eqref{eq:Neumann-problem-blowup} in $C^{k,\gamma}$ domains, and at the same time to prove them for the first time as localized estimates in $\Omega\cap B_2$.
This is new even for $C^\infty$ domains, and it is crucial for our application to free boundary problems.

Our main result is the following:

\begin{theorem}
\label{thm:higher-reg}
Let $L$, $K$, $s$, $\lambda$, $\Lambda$ be as in \eqref{eq:L}-\eqref{eq:K}. Let $k \in \N$, $\gamma \in (0,1)$ with $\gamma \not= s$, and $\Omega \subset \R^n$ be a $C^{k+1,\gamma}$ domain, and $K \in C^{2k+2\gamma+3}(\mathbb{S}^{n-1})$. \\
Let $v \in L^1_{2s}(\R^n)$ with $v/d^{s-1} \in C(\overline{\Omega})$ be a viscosity solution to
\begin{align*}
\left\{\begin{array}{rcl}
Lv &=& f ~~ \text{ in } \Omega \cap B_2,\\
v &=& 0 ~~ \text{ in } B_2 \setminus \Omega,\\ \displaystyle
\partial_{\nu} \left(\frac{v}{d^{s-1}}\right) &=& g ~~ \text{ on } \partial \Omega \cap B_2,
\end{array}\right.
\end{align*}
where $\nu : \partial \Omega \to \mathbb{S}^{n-1}$ is the normal vector of $\Omega$, and $f \in C(\Omega) \cap \mathcal{X}(\Omega \cap B_2)$, $g \in C^{k-1+\gamma}(\partial\Omega \cap B_2)$,
\begin{align}
\label{eq:f-fct-space}
\mathcal{X}(\Omega \cap B_2) = 
\begin{cases}
d^{s-\gamma} L^{\infty}(\Omega \cap B_2), ~~ \text{ if } k +\gamma \le 2s,\\
C^{k-2s+\gamma}(\Omega \cap B_2), ~~ \text{ if } k + \gamma > 2s.
\end{cases}
\end{align}
Then, it holds $v/d^{s-1} \in C_{loc}^{k+\gamma}(\overline{\Omega} \cap B_2)$ and
\begin{align*}
\left\Vert \frac{v}{d^{s-1}} \right\Vert_{C^{k,\gamma}(\overline{\Omega} \cap B_1)} \le c\left( \left\Vert \frac{v}{d^{s-1}} \right\Vert_{L^{\infty}(\Omega \cap B_2)} + \left\Vert v \right\Vert_{L^1_{2s}(\R^n \setminus B_2)} + \Vert f \Vert_{\mathcal{X}(\Omega \cap B_2)} + \Vert g \Vert_{C^{k-1+\gamma}(\partial\Omega \cap B_2)} \right),
\end{align*}
for some $c > 0$, which only depends on $n,s,\lambda,\Lambda,k,\gamma,\Omega$, and $\Vert K \Vert_{C^{2k+2\gamma+3}(\mathbb{S}^{n-1})}$.
\end{theorem}

For the definition of $L^1_{2s}(\R^n)$ and the notion of viscosity solutions, we refer to Section \ref{sec:prelim}. 

Note that the regularity we obtain for $v/d^{s-1}$ depending on the regularity of the domain $\Omega$ and the data $f,g$ is expected to be optimal. For $f$ and $g$, this is an immediate consequence of interior Schauder theory (see \cite{RoSe16b}), and the order of the equation. For the regularity of the domain, we observe that our results are in align with the ones in \cite{AbRo20} once $v \in C(\overline{\Omega} \cap B_2)$.
We obtain results with regularity assumptions on $K$ that are expected to be optimal in case $\Omega$ is a half-space (see \autoref{thm:higher-reg-half-space}). As in \cite{Gru14}, we rule out the case $\gamma = s$. Note that the result is expected to be false in this case. It corresponds to proving Schauder-type regularity estimates of integer order.

Another key advantage of our approach is that it allows for localized results in $\Omega \cap B_2$. Nonetheless, if $\Omega \subset B_2$, and $v$ is a solution to \eqref{eq:Neumann-problem-blowup}, by application of the maximum principle (see \autoref{lemma:Neumann-max-princ}) to the estimate in \autoref{thm:higher-reg} we can obtain the following bound which is purely in terms of $f$ and $g$
\begin{align*}
\left\Vert \frac{v}{d^{s-1}} \right\Vert_{C^{k,\gamma}(\overline{\Omega})} \le c\left( \Vert f \Vert_{\mathcal{X}(\Omega)} + \Vert g \Vert_{C^{k-1+\gamma}(\partial\Omega)} \right).
\end{align*}
Thus, we have the following generalization of \eqref{eq:AbRo-result} to solutions of \eqref{eq:Neumann-problem-blowup}:
\begin{align}
\label{eq:Neumann-result}
\partial \Omega \in C^{k+1,\gamma}, ~ f \in C^{k-2s+\gamma}(\overline{\Omega}), ~ g \in C^{k-1+\gamma}(\partial \Omega) \quad \Rightarrow \quad \frac{v}{d^{s-1}} \in C^{k,\gamma}(\overline{\Omega}) \quad \forall k \in \N, ~ \gamma \in (0,1).
\end{align}

\subsection{A weak maximum principle and nonlocal problems with local Dirichlet conditions}

The example \eqref{eq:Hmissi} of a non-trivial $s$-harmonic function that vanishes outside $B_1$ implies that the Poisson problem \eqref{eq:Dirichlet-problem-weak} for the fractional Laplacian is {ill-posed} even in the homogeneous case.  Therefore, maximum principles are usually established under an additional assumption on the boundary behavior of the solution, ruling out ``large'' solutions such as \eqref{eq:Hmissi} (see \cite{Sil07}, \cite{SeVa14}, \cite{FKV15}, \cite{JaWe19}, \cite{FeJa23}, \cite{FeRo23}). 
Note that a similar phenomenon occurs for local equations, where any constant function is a pointwise solution inside the solution domain.

In this paper, we prove the following nonlocal weak maximum principle, which allows for solutions that blow up at the boundary.

\begin{proposition}
\label{lemma:weak-max-princ-large}
Let $L$, $K$, $s$, $\lambda$, $\Lambda$ be as in \eqref{eq:L}-\eqref{eq:K}. Let $\gamma > 0$ and $\Omega \subset \R^n$ be a $C^{1,\gamma}$ domain. Let $v \in L^1_{2s}(\R^n)$ with $v/d^{s-1} \in C(\overline{\Omega})$ be a viscosity solution to
\begin{align*}
\left\{\begin{array}{rcl}
L v &\ge& 0 ~~ \text{ in } \Omega,\\
v &\ge& 0 ~~ \text{ in } \R^n \setminus \Omega,\\ \displaystyle
\frac{v}{d^{s-1}} &\ge& 0 ~~ \text{ on } \partial \Omega.
\end{array}\right.
\end{align*}
Then, $v \ge 0$.
\end{proposition}

The condition $v/d^{s-1} \ge 0$ in \autoref{lemma:weak-max-princ-large} includes solutions that blow up at the boundary, such as \eqref{eq:Hmissi}. Previously, maximum principles including large solutions have been established in \cite{Aba15}, and \cite{GrHe23}, \cite{LiZh22, LiLi23}. \autoref{lemma:weak-max-princ-large} extends these results to general $2s$-stable integro-differential operators, and to $C^{1,\gamma}$ domains, respectively.

Recall that a natural way to make the nonlocal Poisson problem \eqref{eq:Dirichlet-problem-weak} well-posed is to impose Neumann boundary conditions as in \eqref{eq:Neumann-problem-blowup}. Another way would be to prescribe the limit of the quotient $v/d^{s-1}$ directly, which leads to the following nonlocal problem with local Dirichlet data, which was introduced independently in \cite{Gru14}, \cite{Aba15}:
\begin{align}
\label{eq:Dirichlet-problem-blowup}
\left\{\begin{array}{rcl}
L v &=& f ~~ \text{ in } \Omega,\\
v &=& 0 ~~ \text{ in } \R^n \setminus \Omega,\\ \displaystyle
\frac{v}{d^{s-1}} &=& h ~~ \text{ on } \partial \Omega.
\end{array}\right.
\end{align}

The weak maximum principle in \autoref{lemma:weak-max-princ-large} implies that the problems \eqref{eq:Dirichlet-problem-blowup} and \eqref{eq:Dirichlet-problem-weak} are equivalent, when $h \equiv 0$. Thus, \eqref{eq:Dirichlet-problem-blowup} can be seen as an inhomogeneous nonlocal Dirichlet problem.

Another contribution of this article is the following Schauder-type boundary regularity estimate for solutions to nonlocal equations with local Dirichlet data:

\begin{theorem}
\label{thm:dirichlet}
Let $L$, $K$, $s$, $\lambda$, $\Lambda$ be as in \eqref{eq:L}-\eqref{eq:K}. Let $k \in \N$, $\gamma \in (0,1)$ with $\gamma \not= s$, and $\Omega \subset \R^n$ be a $C^{k+1,\gamma}$ domain, and $K \in C^{2k+2\gamma+3}(\mathbb{S}^{n-1})$. \\
Let $v \in L^1_{2s}(\R^n)$ with $v/d^{s-1} \in C(\overline{\Omega})$ be a viscosity solution to \eqref{eq:Dirichlet-problem-blowup} with $f \in C(\Omega) \cap \mathcal{X}(\Omega)$ and $h \in C^{k+\gamma}(\partial \Omega)$, where $\mathcal{X}$ is as in \eqref{eq:f-fct-space}.
Then, it holds $v/d^{s-1} \in C_{loc}^{1+\gamma}(\overline{\Omega})$, and
\begin{align*}
\left\Vert \frac{v}{d^{s-1}} \right\Vert_{C^{k,\gamma}(\overline{\Omega})} \le c\left( \Vert f \Vert_{\mathcal{X}(\Omega)} + \Vert h \Vert_{C^{k+\gamma}(\partial\Omega)} \right)
\end{align*}
for some $c > 0$, which only depends on $n,s,\lambda,\Lambda,k,\gamma,\Omega$, and $\Vert K \Vert_{C^{2k+2\gamma+3}(\mathbb{S}^{n-1})}$.
\end{theorem}

We refer to \cite{Gru15,Gru23} for similar results in the framework of pseudodifferential operators.

Note that \eqref{eq:Dirichlet-problem-blowup} can always be reduced to the homogeneous problem \eqref{eq:Dirichlet-problem-weak}. In fact, if $\Omega$ and $h$ are regular enough, one can extend $h$ to a smooth function in $\overline{\Omega}$ and consider $w := v - d^{s-1} h$. Then, $w$ solves the homogeneous problem \eqref{eq:Dirichlet-problem-weak} with a new right hand side $\tilde{f} = f - L(d^{s-1}h)$. Since $L(d^{s-1}h)$ has good regularity properties (see \autoref{cor:dist-s-1-smooth-domains}), we can prove \autoref{thm:dirichlet}, by application of the results in \cite{RoSe17, AbRo20}.

\subsection{Strategy of the proof: regularity for nonlocal problems with local Neumann data}

Since the nonlocal problem with inhomogeneous local Dirichlet data \eqref{eq:Dirichlet-problem-blowup} can always be reduced to the homogeneous problem \eqref{eq:Dirichlet-problem-weak} for which the boundary regularity theory was already established (see \cite{RoSe17}, \cite{AbRo20}), the proof of \autoref{thm:dirichlet} is rather simple. 

In sharp contrast to that, for the Neumann problem \eqref{eq:Neumann-problem-blowup} there is no cheap way to obtain the boundary regularity results in \autoref{thm:higher-reg} from the existing theory. In fact, it is already highly non-trivial to establish H\"older continuity of the quotient $v/d^{s-1}$ up to the boundary (see \autoref{thm:bdry-Holder} below). 

Our proof of \autoref{thm:bdry-Holder} goes in \emph{three main steps}.\\
\emph{First}, we establish a weak maximum principle for solutions to the Neumann problem \eqref{eq:Neumann-problem-blowup}.

\begin{proposition}
\label{lemma:Neumann-max-princ-intro}
Let $L$, $K$, $s$, $\lambda$, $\Lambda$ be as in \eqref{eq:L}-\eqref{eq:K}. Let $\gamma > 0$, $\Omega \subset \R^n$ be a $C^{2,\gamma}$ domain, and $K \in C^{5+2\gamma}(\mathbb{S}^{n-1})$. Let $v \in L^1_{2s}(\R^n)$ with $v/d^{s-1} \in C(\overline{\Omega})$ be a viscosity solution to
\begin{align*}
\left\{\begin{array}{rcl}
L v &\ge& 0  ~~ \text{ in } \Omega,\\
v &\ge& 0 ~~ \text{ in } \R^n \setminus \Omega,\\ \displaystyle
\partial_{\nu} \left(\frac{v}{b_{\Omega}}\right) &\le& 0 ~~ \text{ on } \partial \Omega,
\end{array}\right.
\end{align*}
where $b_{\Omega}$ is defined in \eqref{eq:b},
then, $v \ge 0$.
\end{proposition}

Note that this result seems to be the first maximum principle for nonlocal problems with local Neumann boundary conditions in the literature. We believe it to be of independent interest and refer to \autoref{lemma:Neumann-max-princ} for a corresponding $L^{\infty}$ bound in the case of inhomogeneous data. The function $b$ can be thought of as a special regularized distance function taken to the power $s-1$. We stress that the result is no longer true if the function $b$ is replaced by $\tilde{d}^{s-1}$, where $\tilde{d}$ is another regularized distance function. In fact, \autoref{lemma:Neumann-max-princ-intro} holds true for the function in \eqref{eq:Hmissi} if $b = (1-|\cdot|)_+^{s-1}$, but fails if we replace $b$ by the regularized distance $\tilde{d} = (1-|\cdot|^4)$.

The proof of \autoref{lemma:Neumann-max-princ-intro} follows from a nonlocal Hopf-type lemma for solutions to the inhomogeneous Dirichlet problem \eqref{eq:Dirichlet-problem-blowup} (see \autoref{lemma:Hopf-large}), which in turn follows from the weak maximum principle in \autoref{lemma:weak-max-princ-large}. All of these results rely heavily on explicit barriers for \eqref{eq:Dirichlet-problem-blowup} in $C^{1,\gamma}$ domains that are adapted to the geometry of the domain and blow up at the boundary like $d^{s-1}$. These barriers can be seen as perturbations of \eqref{eq:Hmissi}, or rather of 1D solutions such as
\begin{align}
\label{eq:1D-barrier}
L (x_n)_+^{s-1} = 0 ~~ \text{ in } \{ x_n > 0 \}.
\end{align}
Note that \eqref{eq:1D-barrier} follows simply by differentiating the equation
\begin{align*}
L (x_n)_+^s = 0 ~~ \text{ in } \{ x_n > 0 \}.
\end{align*}
The previous identity is a classical fact for nonlocal operators \eqref{eq:L}-\eqref{eq:K} (see \cite[Lemma 2.6.2]{FeRo23}).

The \emph{second} main step in the proof of \autoref{thm:bdry-Holder} is to establish H\"older continuity of order $\alpha$, for $\alpha \in (0,1)$ small enough, up to the boundary of $v/d^{s-1}$ for solutions to \eqref{eq:Neumann-problem-blowup} in $C^{1,\gamma}$ domains. 

\begin{theorem}
\label{thm:bdry-Holder}
Let $L$, $K$, $s$, $\lambda$, $\Lambda$ be as in \eqref{eq:L}-\eqref{eq:K}. Let $\gamma \in (0,1)$, $\Omega \subset \R^n$ be a $C^{2,\gamma}$ domain, and $K \in C^{5+2\gamma}(\mathbb{S}^{n-1})$. Let $v \in L^1_{2s}(\R^n)$ with $v/d^{s-1} \in C(\overline{\Omega})$ be a viscosity solution to
\begin{align*}
\left\{\begin{array}{rcl}
L v &=& f ~~ \text{ in } \Omega \cap B_2,\\
v &=& 0 ~~ \text{ in } B_2 \setminus \Omega,\\ \displaystyle
\partial_{\nu}\left( \frac{v}{d^{s-1}}\right) &=& g ~~ \text{ on }  \partial \Omega \cap B_2, 
\end{array}\right.
\end{align*}
with $f \in C(\Omega \cap B_2)$ and $g \in C(\partial \Omega \cap B_2)$.
Then, there exists $\alpha_0 > 0$, such that when $d^{s-\alpha} f \in L^{\infty}(\Omega \cap B_2)$ for some $\alpha \in (0,\alpha_0]$, then it holds $v/d^{s-1} \in C^{\alpha}_{loc}(\overline{\Omega} \cap B_2)$, and
\begin{align*}
\left\Vert \frac{v}{d^{s-1}} \right\Vert_{C^{\alpha}(\overline{\Omega} \cap B_1)} \le c \left( \left\Vert \frac{v}{d^{s-1}} \right\Vert_{L^{\infty}(\Omega \cap B_2)} + \Vert v \Vert_{L^1_{2s}(\R^n \setminus B_2)} + \Vert d^{s-\alpha} f \Vert_{L^{\infty}(\Omega \cap B_2 )} + \Vert g \Vert_{L^{\infty}( \partial \Omega \cap B_2 )} \right),
\end{align*}
where $c > 0$ and $\alpha_0$ depend only on $n,s,\lambda,\Lambda,\gamma$, and the $C^{2,\gamma}$ radius of $\Omega$.
\end{theorem}

The proof of \autoref{thm:bdry-Holder} uses the weak maximum principle in \autoref{lemma:Neumann-max-princ-intro} and the interior weak Harnack inequality, to establish a weak Harnack inequality for $v/d^{s-1}$ at the boundary (see \autoref{lemma:bdry-weak-Harnack}). This allows us to deduce a so called ``growth lemma'' for $v/d^{s-1}$, stating that $v/d^{s-1}$ must be large pointwise in a ball centered at the boundary, if $v/d^{s-1}$ was large in a measure-theoretic sense in a ball away from the boundary. Such growth lemma allows to establish oscillation decay for $v/d^{s-1}$ at the boundary, and to deduce the H\"older estimate in \autoref{thm:bdry-Holder}. A similar proof for the classical Laplacian can be found in \cite{LiZh23}.

Once the boundary H\"older estimate is shown, we can establish the higher order boundary regularity in \autoref{thm:higher-reg} via a blow-up argument. This is the \emph{third}, and last step of the proof. \autoref{thm:bdry-Holder} is crucial in order to deduce uniform convergence of the blow-up sequence.\\
The blow-up argument follows the scheme in \cite{AbRo20} and relies on a Liouville theorem in the half-space with local Neumann data (see \autoref{thm:Liouville}). However, major modifications have to be made in most of the steps due to the boundary blow-up of solutions. For instance, we need to show the following new result (see \autoref{cor:dist-s-1-smooth-domains}):
\begin{align*}
\partial \Omega \in C^{k+1,\gamma} \quad \Rightarrow \quad L(d^{s-1}) \in C^{k-1+\gamma-s}(\overline{\Omega}) \qquad \text{ if } k+\gamma > 1 + s.
\end{align*}
Moreover, the presence of a Neumann boundary condition complicates some of the arguments, such as the proof of a stability result for viscosity solutions (see \autoref{lemma:stability}). Finally, as in \cite{AbRo20} we need to make use of a suitable notion of nonlocal equations up to a polynomial (see \cite{DSV19}, \cite{DDV22}) in order to account for solutions that grow too fast at infinity (see \autoref{def:L-up-to-a-polynomial}).

\subsection{Applications to free-boundary problems}

We end the discussion of the main results of this article by shedding some light on a, perhaps unexpected, connection between nonlocal problems with local Neumann boundary data and free boundary problems. This connection is a main motivation for us to study \eqref{eq:Neumann-problem-blowup}. Let us explain this phenomenon in the particular case of the fractional Laplacian.

The nonlocal one-phase free boundary problem, which was introduced in \cite{CRS10} (see also \cite{RoWe24}), deals with the minimization of the following functional
\begin{align}
\label{eq:one-phase}
\mathcal{I}(w) := \iint_{(B_1^c \times B_1^c)^c} \big( w(x) - w(y) \big)^2 \frac{\d y \d x}{|x-y|^{n+2s}} + M \big| \{ w > 0 \} \cap B_1 \big|
\end{align}
for some $M > 0$ and with prescribed values of $w$ in $\R^n \setminus B_1$. One can show (see \cite{CRS10,FeRo24}) that local minimizers of \eqref{eq:one-phase} are $C^s(B_1)$ and that they are viscosity solutions to
\begin{align}
\label{eq:one-phase-viscosity}
\left\{\begin{array}{rcl}
(-\Delta)^s w &=& 0 \qquad ~~ \text{ in } \Omega \cap B_1,\\
w &=& 0 ~~ \qquad \text{ in } B_1 \setminus \Omega,\\ \displaystyle
 \frac{w}{d^{s}} &=& c_{n,s} M ~~ \text{ on } \partial \Omega \cap B_1,
\end{array}\right.
\end{align}
where $c_{n,s} > 0$ is a constant and $\Omega := \{ w > 0 \}$. An important question in the theory is to determine the regularity of the free boundary $\partial \Omega$ near so called ``regular points''. These are the points $x_0 \in \partial \Omega \cap B_1$ for which blow-ups of $w$ are half-space solutions, i.e., (up to rotations and multiplicative constants)
\begin{align*}
\frac{w(x_0 + rx)}{r^s} \to w_0(x) := (x_n)_+^s ~~ \text{ locally uniformly}.
\end{align*}
One can show using the extension for $(-\Delta)^s$ (see \cite{DeRo12, DeSa12, DSS14}) that once a sequence $(w_{\eps})$ of viscosity solutions \eqref{eq:one-phase-viscosity} is ``$\eps$-close'' to the half-space solution $w_0$ in the sense that
\begin{align*}
(x_n - \eps)_+^s \le w_{\eps}(x) \le (x_n + \eps)_+^s,
\end{align*}
then it holds, as $\eps \searrow 0$:
\begin{align*}
\frac{w_{\eps}(x) - (x_n)_+^{s}}{\eps} \to (x_n)_+^{s-1} u(x),
\end{align*}
where $u$ solves the so called ``linearized problem''
\begin{align}
\label{eq:linearized-problem}
\left\{\begin{array}{rcl}
(-\Delta)^s ((x_n)_+^{s-1} u) &=& 0 ~~ \text{ in } \{ x_n > 0 \} \cap B_1,\\
\partial_n u &=& 0 ~~ \text{ on } \{ x_n = 0 \} \cap B_1.
\end{array}\right.
\end{align}
Hence, $(x_n)_+^{s-1} u$ is a solution to a nonlocal problem with local Neumann data \eqref{eq:Neumann-problem-blowup} in the half-space, and it explodes at the boundary $\{ x_n = 0\} \cap B_1$.

In order to establish regularity results for the free boundary $\Omega = \{ w > 0 \}$ near regular points, it is an important step to establish boundary regularity results for the solution to the linearized problem. For \eqref{eq:linearized-problem} this was done in \cite{DeRo12}, \cite{DeSa12}, \cite{DSS14}, using the Caffarelli-Silvestre extension.\\
In the light of this connection, our main result \autoref{thm:higher-reg} also makes a contribution to the theory of the nonlocal one-phase problem \eqref{eq:one-phase}, and provides a completely new proof of the regularity for \eqref{eq:linearized-problem}, even in the case of the fractional Laplacian.

We end this discussion by stating a variant of \autoref{thm:higher-reg} in the special case $\Omega = \{ x_n > 0 \}$. This result holds true under assumptions on the regularity of $K$ which are expected to be optimal, and it will be helpful in the study of the nonlocal one-phase free boundary problem \eqref{eq:one-phase-viscosity} with respect to general nonlocal operators \eqref{eq:L}-\eqref{eq:K}, which we plan to investigate in a future work (see \cite{RoWe24b}). 

\begin{theorem}
\label{thm:higher-reg-half-space}
Let $L$, $K$, $s$, $\lambda$, $\Lambda$ be as in \eqref{eq:L}-\eqref{eq:K}. Let $k \in \N$, $\gamma \in (0,1)$ with $\gamma \not= s$.\\
Let $u \in C(\{ x_n \ge 0 \} \cap B_2)$ with $(x_n)_+^{s-1} u \in L^1_{2s}(\R^n)$ be a viscosity solution to
\begin{align*}
\left\{\begin{array}{rcl}
L((x_n)_+^{s-1} u ) &=& f ~~ \text{ in } \{ x_n > 0 \} \cap B_2,\\
\partial_{n} u &=& g ~~ \text{ on } \partial \{ x_n = 0 \} \cap B_2.
\end{array}\right.
\end{align*}
with $f \in C(\{ x_n > 0 \} \cap B_2) \cap \mathcal{X}(\{ x_n > 0 \} \cap B_2)$, $g \in C^{k-1+\gamma}(\{ x_n = 0\} \cap B_2)$, and $K \in C^{k-2s+\gamma}(\mathbb{S}^{n-1})$ if $k+\gamma > 2s$, where $\mathcal{X}$ is as in \eqref{eq:f-fct-space}. Then, it holds
\begin{align*}
\left\Vert u \right\Vert_{C^{k,\gamma}(\{ x_n \ge 0 \} \cap B_1)} \le c\Big( \left\Vert u \right\Vert_{L^{\infty}(\{ x_n > 0 \} \cap B_2 )} &+ \Vert (x_n)_+^{s-1} u \Vert_{L^1_{2s}(\R^n \setminus B_2)} \\
&+ \Vert f \Vert_{\mathcal{X}(\{ x_n > 0 \} \cap B_2)} + \Vert g \Vert_{C^{k-1+\gamma}(\{ x_n = 0 \} \cap B_2)} \Big)
\end{align*}
for some $c > 0$, which only depends on $n,s,\lambda,\Lambda,k,\gamma$, and (if $k + \gamma > 2s$) also on $\Vert K \Vert_{C^{k-2s+\gamma}(\mathbb{S}^{n-1})}$.
\end{theorem}

Finally, we make the following remark.

\begin{remark}
Note that the following two problems are equivalent if $v \in C(\overline{\Omega} \cap B_2)$, i.e., if solutions do not blow up on $\partial \Omega \cap B_2$:
\begin{align*}
\left\{\begin{array}{rcl}
Lv &=& f ~~ \text{ in } \Omega \cap B_2,\\
v &=& 0 ~~ \text{ in } B_2 \setminus \Omega,\\ \displaystyle
\partial_{\nu} \left( \frac{v}{d^{s-1}}\right) &=& g ~~ \text{ on } \partial \Omega \cap B_2,
\end{array}\right.
\qquad \leftrightarrow \qquad 
\left\{\begin{array}{rcl}
Lv &=& f ~~ \text{ in } \Omega \cap B_2,\\
v &=& 0 ~~ \text{ in } B_2 \setminus \Omega,\\ \displaystyle
\frac{v}{d^{s}} &=& g ~~ \text{ on } \partial \Omega \cap B_2.
\end{array}\right.
\end{align*}
Indeed, since $v \equiv 0$ in $B_2 \setminus \Omega$, it holds for any $x_0 \in \partial \Omega \cap B_2$:
\begin{align*}
\partial_{\nu} \left(\frac{v}{d^{s-1}}\right) = \lim_{x \to x_0} \frac{\frac{v}{d^{s-1}}(x) - \lim_{z \to x_0} \frac{v}{d^{s-1}}(z)}{d(x)} = \lim_{x \to x_0} \frac{v}{d^s}(x).
\end{align*}
Recall that the second problem is satisfied by minimizers to the nonlocal one-phase problem \eqref{eq:one-phase-viscosity}. Moreover, the above problem is the nonlocal counterpart of the over-determined Serrin's problem whenever $\Omega \subset B_2$ (see for instance \cite{FaJa15,SoVa19,BiJa20,DPTV23}).
\end{remark}

\subsection{Acknowledgments}

The authors were supported by the European Research Council under the Grant Agreements No. 801867 (EllipticPDE) and No. 101123223 (SSNSD), and by AEI project PID2021-125021NA-I00 (Spain).
Moreover, X.R was supported by the grant RED2022-134784-T funded by AEI/10.13039/501100011033, by AGAUR Grant 2021 SGR 00087 (Catalunya), and by the Spanish State Research Agency through the Mar\'ia de Maeztu Program for Centers and Units of Excellence in R{\&}D (CEX2020-001084-M).

\subsection{Organization of the paper}

This paper is organized as follows.
In Section \ref{sec:prelim} we introduce the notion of viscosity solutions to \eqref{eq:Neumann-problem-blowup} and give some preliminary lemmas. Among them are already several new results of independent interest, such as the construction of explicit barriers exploding at the boundary (see Subsection \ref{subsec:barriers}), an analysis of the regularity of $L(d^{s-1})$ in terms of the regularity of the domain (see \autoref{cor:dist-s-1-smooth-domains}), and a stability result for viscosity solutions (see \autoref{lemma:stability}). In Section \ref{sec:max-principles} we prove maximum principles for solutions to nonlocal problems with local Dirichlet- and Neumann data (see \autoref{lemma:weak-max-princ-large} and \autoref{lemma:Neumann-max-princ-intro}). Section \ref{sec:Holder-estimate} is devoted to the proof of the H\"older estimate up to the boundary (see \autoref{thm:bdry-Holder}). In Section \ref{sec:Liouville} we prove a Liouville theorem in the half-space (see \autoref{thm:Liouville}), and in Section \ref{sec:higher} we carry out a blow-up argument to prove our main result, \autoref{thm:higher-reg}. Finally, Section \ref{sec:dirichlet} contains the proof of the regularity for the inhomogeneous Dirichlet problem (see \autoref{thm:dirichlet}).

\section{Preliminaries}
\label{sec:prelim}

In this section, we give several important definitions, such as the definitions of viscosity solutions to \eqref{eq:Neumann-problem-blowup}. In Subsection \ref{subsec:prelim-dist} we establish the regularity of $L(d^{s-1})$ depending on the regularity of the domain and in Subsection \ref{subsec:barriers} we use these results to construct barrier functions. In Subsection \ref{subsec:poly}, we introduce the notion of nonlocal equations satisfied up to a polynomial, and in Subsection \ref{subsec:two-visc} we establish stability of viscosity solutions and prove that the sum of two viscosity solutions is again a viscosity solution.

From now on, we denote by $\mathcal{L}_s^{\text{hom}}(\lambda,\Lambda)$ the class of operators \eqref{eq:L} with kernels satisfying \eqref{eq:K}. Moreover, whenever we say $K \in C^{\alpha}(\mathbb{S}^{n-1})$ for some $\alpha > 0$, we mean that $\Vert K \Vert_{C^{\alpha}(\mathbb{S}^{n-1})} \le \Lambda$.
Sometimes, we denote the class of operators \eqref{eq:L} satisfying \eqref{eq:K} and $K \in C^{\alpha}(\mathbb{S}^{n-1})$ by $\mathcal{L}^{\text{hom}}_s(\lambda,\Lambda,\alpha)$.

Moreover, given an open, bounded domain $\Omega \subset \R^n$ with $\partial \Omega \in C^{\beta}$ for some $\beta > 1$, $d := d_{\Omega} : \R^n \to [0,\infty)$ will denote the regularized distance which satisfies $d \in C^{\infty}(\Omega) \cap C^{\beta}(\overline{\Omega})$ and $d \equiv 0$ in $\R^n \setminus \Omega$. Crucially, we have $\dist(\cdot,\Omega) \le d \le C \dist(\cdot,\Omega)$ in $\R^n$, i.e., the topological distance and the regularized distance are pointwise comparable. We will often use the fact that $|D^k d| \le c d^{\beta-k}$ (see \cite[Definition 2.7.5]{FeRo23}). Throughout this article, we will define $d^{s-1} \equiv 0$ in $\R^n \setminus \Omega$.

In the following, whenever $x_0 \in \partial \Omega$, we write $v/d^{s-1}(x_0) := \lim_{\Omega \ni x \to x_0} v/d_{\Omega}^{s-1}(x)$.

\subsection{Function spaces and solution concepts}

Let us introduce the following function space
\begin{align*}
L^1_{\alpha}(\R^n) &:= \left\{ u : \Vert u \Vert_{L^1_{\alpha}(\R^n)} := \int_{\R^n} |u(y)| (1 + |y|)^{-n-\alpha} \d y < \infty \right\}, ~~ \alpha > 0.
\end{align*}

Typically, we will use the previous definition with $\alpha = 2s$. We are now in a position to give the notion of viscosity solution to \eqref{eq:Neumann-problem-blowup}. 

\begin{definition}[Viscosity solution]
Let $\Omega \subset \R^n$ be an open, bounded domain with $\partial \Omega \in C^{1,\gamma}$. By $\nu \in \mathbb{S}^{n-1}$, we denote the inner normal vector to $\partial \Omega$.
\begin{itemize}
\item[(i)] We say that $v \in C(\Omega) \cap L^1_{2s}(\R^n)$ is a viscosity subsolution to 
\begin{align}
\label{eq:def-visc}
L v = f ~~ \text{ in } \Omega \cap B_1,
\end{align}
where $f \in C(\Omega \cap B_1)$, if for any $x \in \Omega \cap B_1$ and any neighborhood $N_x \subset \Omega$ of $x$ it holds 
\begin{align}
\label{eq:viscosity-sol}
L \phi(x) \le f(x) ~~ \forall \phi \in C^2(N_x) \cap L^1_{2s}(\R^n) ~~ \text{ s.t. } v(x) = \phi(x), ~~ \phi \ge v.
\end{align}
We say that $v$ is a viscosity supersolution to \eqref{eq:viscosity-sol} if \eqref{eq:viscosity-sol} holds true for $-v$ and $-f$ instead of $v$ and $f$. Moreover, $v$ is a viscosity solution to \eqref{eq:viscosity-sol}, if it is a viscosity subsolution and a viscosity supersolution.

\item[(ii)] For any function $b \in L^1_{2s}(\R^n)$ with $b/d^{s-1} \in C^{1}(\overline{\Omega})$ we say that $v \in L^1_{2s}(\R^n)$ with $v/d^{s-1} \in C( \overline{\Omega})$ is a viscosity subsolution to
\begin{align*}
\partial_{\nu}(v/b) = g ~~ \text{ on } \partial \Omega \cap B_1,
\end{align*}
where $g \in C(\partial \Omega \cap B_1)$, if for any $x \in \partial\Omega \cap B_1$ and any neighborhood $N_x \subset \overline{\Omega} \cap B_1$ of $x$ it holds 
\begin{align}
\label{eq:viscosity-sol-Neumann}
\partial_{\nu} \phi(x) \le g(x) ~~ \forall \phi \in C^2(N_x) \cap L^{\infty}(\overline{\Omega}) ~~ \text{ s.t. } v/b(x) = \phi(x), ~~ \phi \le v/b.
\end{align}
We say that $v$ is a viscosity supersolution to \eqref{eq:viscosity-sol-Neumann} if \eqref{eq:viscosity-sol-Neumann} holds true for $-v$ and $-g$ instead of $v$ and $g$. Moreover, $v$ is a viscosity solution to \eqref{eq:viscosity-sol-Neumann}, if it is a viscosity subsolution and a viscosity supersolution.
\end{itemize}
\end{definition}

Note that clearly, if in (i) $Lv(x)$, or if in (ii) $\partial_{\nu}(v/d^{s-1})(x) = \lim_{\Omega \ni y \to x} (v/d^{s-1})(y)$ exists in the strong sense, then the notions of viscosity solutions coincide with the ones for strong solutions (see \cite[Lemma 3.4.13]{FeRo23}).

\subsection{Nonlocal operators and the distance function}
\label{subsec:prelim-dist}

The goal of this subsection is to establish several lemmas on the regularity of $L(d^{s-1})$ depending on the regularity of $\Omega$. \autoref{lemma:dist-s-1+eps} will help us to establish barriers in $C^{1,\gamma}$ domains and \autoref{cor:dist-s-1-smooth-domains} is crucial for domains that are more regular.

The following lemma is a slight modification of \cite[Lemma B.2.4]{FeRo23}.

\begin{lemma}
\label{lemma:dist-int-est}
Let $L \in \mathcal{L}_s^{\text{hom}}(\lambda,\Lambda)$. Let $\Omega \subset \R^n$ be a bounded Lipschitz domain with Lipschitz constant $L$ and $C^{0,1}$ radius $\rho_0 > 0$. Let $x_0 \in \Omega$ with $\rho := d_{\Omega}(x_0)$, $\gamma > -1$ and $\gamma < \beta$. Then,
\begin{align*}
\int_{\Omega \setminus B_{\rho/2}} d^{\gamma}_{\Omega}(x_0 + y) |y|^{-n-\beta} \d y \le C (1 + \rho^{\gamma - \beta})
\end{align*}
for some constant $C > 0$, depending only on $n,\gamma,\beta,\rho_0,L$, and,  if $\gamma > 0$ or $\beta \le 0$ also on $\diam(\Omega)$.
\end{lemma}

\begin{proof}
We assume that $x_0 = 0$. By \cite[Lemma B.2.4]{FeRo23}, there exists $\kappa > 0$ such that for any $t \in (0,\kappa)$:
\begin{align}
\label{eq:Hausdorff-estimate}
\mathcal{H}^{n-1}\left(\{d = t\} \cap (B_{2^{j+1}\rho} \setminus B_{2^j \rho}) \right) \le C (2^j \rho)^{n-1}.
\end{align}
Note that 
\begin{align*}
\int_{ (\Omega \setminus B_{\rho/2}) \cap \{d \ge \kappa \} } d^{\gamma}(y) |y|^{-n-\beta} \d y \le (\diam(\Omega)^{\gamma}\1_{\{\gamma > 0\}} + \kappa^{\gamma} \1_{\{ \gamma \le 0\}}) \int_{ (\Omega \setminus B_{\rho/2}) \cap \{d \ge \kappa \} }  |y|^{-n-\beta} \d y \le c
\end{align*}
for some constant $c > 0$ depending on $\kappa$ and, if $\gamma > 0$ or $\beta \le 0$ also on $\diam(\Omega)$, but independent of $\rho$. The independence of $\rho$ is trivial if $\kappa \le 2\rho$ since then $\Omega \setminus B_{\rho/2} \subset \Omega \setminus B_{\kappa/4}$, and otherwise, it follows from the fact that $B_r \cap \{ d \ge \kappa \} = \emptyset$ once $r \le \kappa/2 \le \kappa - \rho$ (recall that $d(0) = \rho$), so also in this case, we can replace the domain of integration by $\Omega \setminus B_{\kappa/2}$. Moreover, using \eqref{eq:Hausdorff-estimate} and the co-area formula:
\begin{align*}
\int_{(\Omega \setminus B_{\rho/2}) \cap \{d \le \kappa \}} d^{\gamma}(y) |y|^{-n-\beta} \d y &\le c\sum_{j \ge 1} \left( (2^{j}\rho)^{-n-\beta} \int_{(B_{2^{j+1} \rho} \setminus B_{3^j \rho}) \cap \{d \le \kappa \} } d^{\gamma}(y) |\nabla d(y)| \d y \right) \\
&\le c\sum_{j \ge 1} \left( (2^{j} \rho)^{-n-\beta} \int_0^{ \min\{ 2^j \rho , \kappa \} } \hspace{-0.2cm} t^{\gamma} \left[ \int_{ (B_{2^{j+1} \rho} \setminus B_{2^j \rho}) \cap \{ d = t \} } \hspace{-0.6cm} \d \mathcal{H}^{n-1}(y) \right] \d t  \right) \\
&\le c\sum_{j \ge 1}  \left( (2^j \rho)^{-\beta + \gamma} \right) \le c \rho^{\gamma-\beta}
\end{align*}
for some $c > 0$, where we used that $\gamma -\beta < 0$.
\end{proof}

The following lemma will be of central importance for the proof of \autoref{lemma:subsol} and \autoref{lemma:supersol}

\begin{lemma}
\label{lemma:dist-s-1+eps}
Let $L \in \mathcal{L}_s^{\text{hom}}(\lambda,\Lambda)$. Let $\Omega \subset \R^n$ be an open, bounded domain with $\partial \Omega \in C^{1,\gamma}$ for some $\gamma > 0$. Then, for any $\delta \in (0,s)$, there exists $c_1 > 0$, depending only on $n,s,\lambda,\Lambda,\Omega,\gamma,\delta$, and the $C^{1,\gamma}$ radius of $\Omega$, such that
\begin{align*}
|L(d^{s-1})| \le c_1 d^{\delta \gamma - s - 1} ~~ \text{ in } \Omega.
\end{align*}
Moreover, for any $\eps \in (0,s)$, there exist $c_2, c_3 > 0$  depending only on $n,s,\lambda,\Lambda,\gamma,\eps$, and the $C^{1,\gamma}$ radius of $\Omega$, such that
\begin{align*}
- L(d^{s-1+\eps}) \le -c_2 d^{\eps - s - 1} + c_3 ~~ \text{ in } \Omega.
\end{align*}
\end{lemma}

The first claim follows in a similar way as \cite[Proposition B.2.1]{FeRo23}.

\begin{proof}
We let $x_0 \in \Omega$ and denote $\rho = d(x_0)$. Then, we denote
\begin{align*}
l(x) = (d(x_0) + \nabla d(x_0) \cdot (x-x_0))_+
\end{align*}
and observe that 
\begin{align*}
L(l^{s-1}) = 0 ~~ \text{ in } \{ l > 0 \},	
\end{align*}
as a consequence of $L(l^s) = 0$ and $\nabla l^s = s l^{s-1} \nabla l = s \nabla d(x_0) l^{s-1} $.
Next, we claim that
\begin{align}
\label{eq:claim-d-l-est}
|d^{s-1} - l^{s-1}|(x_0 + y) \le 
\begin{cases}
C \rho^{s+\gamma -3} |y|^2 ~~ \text{ in } B_{\rho/2},\\
C |y|^{(1+\gamma)\delta}|d^{s-1-\delta}(x_0 + y) + l^{s-1-\delta}(x_0 + y)| ~~ \text{ in } \R^n \setminus B_{\rho/2}.
\end{cases}
\end{align}
Note that from here, we can compute
\begin{align*}
|L(d^{s-1})(x_0)| &= |L(d^{s-1} - l^{s-1})(x_0)| \\
&\le C \rho^{s+\gamma -3} \int_{B_{\rho/2}} |y|^{2-n-2s}\d y\\
&\quad + C \int_{(x_0 + \Omega) \setminus B_{\rho/2}} |y|^{-n-2s + (1+\gamma)\delta}|d^{s-1-\delta}(x_0 + y) + l^{s-1-\delta}(x_0 + y)| \d y \\
&\le C(1 + \rho^{\gamma-s-1} + \rho^{ \gamma \delta -s-1}) ,
\end{align*}
where we applied \autoref{lemma:dist-int-est} to $d$ and to $l$ with $s-1-\delta =: \gamma < \beta:= 2s - (1+\gamma)\delta$ (choosing $\gamma \in (0,s)$ so small that $\beta > 0$), in order to estimate the third integral. Since this estimate implies the first result, it remain to verify the claim \eqref{eq:claim-d-l-est}.
In case $x \in B_{\rho/2}(x_0)$, we estimate
\begin{align*}
|d^{s-1} - l^{s-1}|(x) &\le |d-l|(x) \Vert d^{s-2} + l^{s-2} \Vert_{L^{\infty}(B_{\rho/2}(x_0))} \le c \Vert D^2 d \Vert_{L^{\infty}(B_{\rho/2}(x_0))} |x_0-x|^2 \rho^{s-2} \\
&\le C \rho^{s+\gamma-3} |y|^2.
\end{align*}
Here, we used that $|D^2 d| \le C d^{-1+\gamma}$ by \cite[Lemma B.0.1]{FeRo23} and that $l \ge c \rho$ in $B_{\rho/2}(x_0)$. The latter statement follows since by the $C^{1,\gamma}$ regularity of $d$, it must be
\begin{align*}
|d(x) - d(x_0) - \nabla d(x_0) \cdot (x-x_0)| \le C \rho^{1+\gamma} ~~ \forall x \in B_{\rho/2}(x_0),
\end{align*}
due to Taylor's formula, and therefore $d(x)$ and $\rho$ are comparable in $B_{\rho/2}(x_0)$, which yields for small enough $\rho$ for some $c > 0$:
\begin{align*}
l(x) \ge d(x_0) + \nabla d(x_0) \cdot (x-x_0) \ge d(x) - C \rho^{1+\gamma} \ge c \rho > 0 ~~ \forall x \in B_{\rho/2}(x_0).
\end{align*}
Note that we can always assume that $\rho > 0$ is small, since otherwise, the result follows by the regularity of $d^{s-1}$ away from the boundary of $\Omega$.

Next, for $x \in \R^n \setminus B_{\rho/2}(x_0)$, we make use of the following algebraic inequality, which follows from the $C^{\delta}$ regularity of the function $t \mapsto t^{s-1-\delta}$ in $[\min\{a,b\} , \max\{a,b\}]$
\begin{align*}
|a^{s-1} - b^{s-1}| \le c|a-b|^{\delta}|a^{s-1-\delta} + b^{s-1-\delta}| ~~ \forall a ,b > 0,
\end{align*}
for any $\delta \in (0,s)$ and some $c > 0$, depending only on $s,\delta$, which allows us to estimate
\begin{align*}
|d^{s-1}(x) - l^{s-1}(x)| &\le c|d(x) - l(x)|^{\delta} |d^{s-1-\delta}(x) + l^{s-1-\delta}(x)| \\ 
&\le c |x_0-x|^{(1+\gamma)\delta}||d^{s-1-\delta}(x) + l^{s-1-\delta}(x)|,
\end{align*}
where we used that by \cite[Lemma B.2.2]{FeRo23} it holds
\begin{align*}
|d(x) - l(x)| \le C |x_0 - x|^{1+\gamma}.
\end{align*}
This proves the first claim.

Now, we turn to to proof of the second result. First, we observe that by similar arguments as in the first part of the proof, we obtain
\begin{align*}
|d^{\eps+ s-1} - l^{\eps + s-1}|(x_0 + y) \le 
\begin{cases}
C \rho^{\eps + s+\gamma -3} |y|^2 ~~ \text{ in } B_{\rho/2},\\
C |y|^{(1+\gamma)\delta}||d^{\eps + s-1-\delta}(x_0 + y) + l^{\eps + s-1-\delta}(x_0 + y)| ~~ \text{ in } \R^n \setminus B_{\rho/2},
\end{cases}
\end{align*}
and therefore
\begin{align*}
|L(d^{\eps + s-1} - l^{\eps + s-1})(x_0)| \le C(1 + \rho^{\eps + \gamma - s - 1} + \rho^{\eps + \gamma \delta - s - 1}).
\end{align*}
We claim that for any $e \in \mathbb{S}^{n-1}$ it holds
\begin{align}
\label{eq:barrier-eps-claim}
\begin{cases}
L((x \cdot e)_+^{\eps + s-1}) &= c_e(x \cdot e)_+^{\eps -s - 1} ~~ \text{ in } \{ x \cdot e  >  0 \},\\
(x \cdot e)_+^{s-1+\eps} &= 0 ~~ \qquad\qquad\qquad \text{ in } \{ x \cdot e \le 0 \}
\end{cases}
\end{align}
for some $c_e \in [c_-,c_+]$, where $c_+ > c_- > 0$ depend only on $n,s,\lambda,\Lambda$. Note that once we have shown the claim \eqref{eq:barrier-eps-claim}, we can conclude the proof, since it implies
\begin{align*}
- L(d^{\eps + s-1})(x_0) &\le - L(l^{\eps + s-1})(x_0) + |L(d^{\eps + s-1} - l^{\eps + s-1})(x_0)| \\
&\le -c \rho^{\eps -s - 1} + C(1 + \rho^{\eps + \gamma - s - 1} + \rho^{\eps + \gamma \delta - s - 1}) \le -c \rho^{\eps -s - 1} + C.
\end{align*}
Hence, it remains to prove \eqref{eq:barrier-eps-claim}. By the $2s$-homogeneity of $L$ we can apply \cite[Lemma B.1.5]{FeRo23} and \cite[Lemma 1.10.3(iii)]{FeRo23} and deduce
\begin{align*}
L((x \cdot e)_+^{\eps + s-1}) = c_1 (-\Delta)^s_{\R} (x \cdot e)_+^{\eps + s - 1} = c_1 c_2 (x \cdot e)_+^{\eps - s - 1}
\end{align*}
for some constant $c_1 > 0 $ and where $c_2$ is given by see \cite[Lemma 2.4]{FaRo22}
\begin{align*}
c_2 = (-\Delta)^s_{\R} (t_+^{\eps + s - 1})(1) = \frac{\Gamma(s+\eps)}{\Gamma(-s+\eps)} \frac{\sin(\pi(-1+\eps))}{\sin(\pi(-1-s+\eps))} > 0.
\end{align*}
This concludes the proof.
\end{proof}

The following lemma is crucial in the proofs of \autoref{lemma:stability}, and  in Section \ref{sec:higher}. It follows by differentiating the corresponding results in \cite{AbRo20}.

\begin{lemma}
\label{lemma:dist-s-1-smooth-domains}
Let $L \in \mathcal{L}_s^{\text{hom}}(\lambda,\Lambda)$. Let $k \in \N$, and $\Omega \subset \R^n$ be an open, bounded domain with $\partial \Omega \in C^{k+1,\gamma}$ for some $\gamma \in (0,1)$ with $\gamma \not= s$, and $0 \in \partial \Omega$. Assume that $K \in C^{2k + 2 \gamma + 3}(\mathbb{S}^{n-1})$. Let $\eta \in C^{k,\gamma}(\overline{\Omega \cap B_1}) \cap C^{\infty}(\Omega \cap B_1)$. Then, there exists $c > 0$, depending only on $n,s,\lambda,\Lambda,\Omega,\gamma,k$, such that the following holds true:
\begin{itemize}
\item[(i)] If $k = 1$ and $\gamma < s$, then
\begin{align*}
|L(d^{s-1} (\nabla d) \eta )| \le c(| \cdot | + |\eta(0)| +  |\nabla \eta(0)|) d^{\gamma - s} ~~ \text{ in } \overline{\Omega} \cap B_{1/2}.
\end{align*}
\item[(ii)] If $k \ge 2$ or $\gamma > s$, then
\begin{align*}
[L(d^{s-1}(\nabla d) \eta)]_{C^{k-1-s+\gamma}(\overline{\Omega} \cap B_{1/2})} \le c \left(| \cdot | + |\eta(0)| + |\nabla \eta(0)|\right).
\end{align*}
\item[(iii)] If $k+\gamma > 2s$, then we have for any $x_0 \in \Omega \cap B_{1/2}$
\begin{align*}
[L(d^{s-1} (\nabla d) \eta )]_{C^{k+\gamma -2 s}(B_{d(x_0)/2}(x_0))} \le c \left(| \cdot | + |\eta(0)| + |\nabla \eta(0)|\right) d^{s-1}(x_0).
\end{align*}
\end{itemize}
\end{lemma}

\begin{proof}
By \cite[Corollary 2.3]{AbRo20} (see also \cite[Corollary 3.9]{Kuk21} for $i > k+1$), we deduce that
\begin{align}
\label{eq:AboRo-Cor23}
|D^i L(d^s \eta)| \le c(|\cdot| + |\eta(0)|) d^{k + \gamma - s- i} ~~ \text{ in } \overline{\Omega} \cap B_{1/2} ~~ \forall i \in \N.
\end{align}
Note that by \cite[Theorem 2.2]{AbRo20} and the choice of $\psi$ in the proof of \cite[Corollary 2.3]{AbRo20}, it follows that the assumption $\eta \in C^{k,\gamma}(\overline{\Omega \cap B_1}) \cap C^{\infty}(\Omega \cap B_1)$ is sufficient for \eqref{eq:AboRo-Cor23} to hold true.
Let us now prove (i) and assume that $k = 1$ and $\gamma < s$. Then, since $D^i \eta \in C^{\infty}(\R^n)$, another application of \cite[Corollary 2.3]{AbRo20} yields
\begin{align*}
\Vert L(d^s D^i \eta) \Vert_{C^{1+\gamma -s}(\overline{\Omega} \cap B_{1/2})} \le C ~~ \text{ for } i \in \{1,2\}.
\end{align*}
Since $\nabla (d^s \eta) = s d^{s-1}(\nabla d) \eta + d^s \nabla \eta$, 
a combination of the previous two estimates with $i = 1$ implies
\begin{align*}
|L(d^{s-1} (\nabla d) \eta)| \le s^{-1}|\nabla L(d^s \eta)| + s^{-1}|L(d^s \nabla \eta)| \le c(| \cdot | + |\eta(0)| +  |\nabla \eta(0)|) d^{\gamma - s} ~~ \text{ in } \overline{\Omega} \cap B_{1/2},
\end{align*}
which yields the result in (i).\\
To see (ii) and (iii), we observe first that by application of \eqref{eq:AboRo-Cor23}, we have for any $i \in \N$
\begin{align*}
|D^i L(d^{s} (\nabla \eta))| \le c(| \cdot | + |\nabla \eta(0)|) d^{k + \gamma - s - i} ~~ \text{ in } \overline{\Omega} \cap B_{1/2}.
\end{align*}
Next, by differentiation, we obtain 
\begin{align*}
D^{i+1}(d^s \eta) &= s D^i(d^{s-1} (\nabla d) \eta) + D^i(d^s \nabla \eta).
\end{align*}
Thus, altogether for every $i \in \N$
\begin{align}
\label{eq:dist-s-1-smooth-key}
\begin{split}
|D^{i}(L(d^{s-1}(\nabla d) \eta))| &\le s^{-1} |D^{i+1} L(d^s \eta)| + s^{-1} |D^i L(d^{s} (\nabla \eta))| \\
&\le c(| \cdot | + |\eta(0)|) d^{k + \gamma -s - (i+1)} + c (| \cdot | + |\nabla \eta(0)|)d^{k + \gamma- s-i} \\
&\le c(| \cdot | + |\eta(0)| + |\nabla \eta(0)|) d^{k + \gamma -s - i - 1} ~~ \text{ in } \overline{\Omega} \cap B_{1/2}.
\end{split}
\end{align}
To conclude the proof of (ii) let $x_0 \in \Omega \cap B_{1/2}$, and note that if $k \ge 2$ or $\gamma > s$, then \eqref{eq:dist-s-1-smooth-key} applied with $i = k-1 + \lceil \gamma - s \rceil$ implies
\begin{align*}
[L(d^{s-1} (\nabla d) \eta )]_{C^{k-1-s+\gamma}(B_{d(x_0)/2}(x_0))} &\le \sup_{x,y \in B_{d_{x_0}/2}(x_0)} \frac{\Vert D^{k-1 + \lceil \gamma - s \rceil }(L(d^{s-1}(\nabla d) \eta)) \Vert_{L^{\infty}(B_{d(x_0)/2}(x_0))} }{|x-y|^{\gamma - s - \lceil \gamma - s \rceil }} \\
&\le c(| \cdot | + |\eta(0)| + |\nabla \eta(0)|) d^{\gamma -s - \lceil \gamma - s \rceil }(x_0) d^{s-\gamma + \lceil \gamma - s \rceil}(x_0) \\
&\le c(| \cdot | + |\eta(0)| + |\nabla \eta(0)|),
\end{align*}
where we used that $\gamma < 1$. From here, a covering argument (see \cite[Lemma A.1.4]{FeRo23}) yields the desired regularity estimate in $\overline{\Omega} \cap B_{1/2}$.\\
To prove (iii), note that if $k + \gamma > 2s$, then \eqref{eq:dist-s-1-smooth-key} applied with $i = k + \lceil \gamma - s \rceil$ implies
\begin{align*}
[L(d^{s-1} (\nabla d) \eta )]_{C^{k+\gamma -2 s}(B_{d(x_0)/2}(x_0))} &\le \sup_{x,y \in B_{d(x_0)/2}(x_0))} \frac{\Vert D^{k+\lceil \gamma - s \rceil}(L(d^{s-1}(\nabla d) \eta)) \Vert_{L^{\infty}(B_{d(x_0)/2}(x_0))} }{|x-y|^{\gamma-2s-\lceil \gamma - s \rceil}} \\
&\le c(| \cdot | + |\eta(0)| + |\nabla \eta(0)|) d^{\gamma -s -1-\lceil \gamma - s \rceil}(x_0) d^{2s-\gamma+\lceil \gamma - s \rceil}(x_0) \\
&\le c(| \cdot | + |\eta(0)| + |\nabla \eta(0)|) d^{s-1}(x_0), 
\end{align*}
where we used that $\gamma < 1$. This implies (iii), and we conclude the proof. 
\end{proof}

As a corollary, we obtain the following result:

\begin{corollary}
\label{cor:dist-s-1-smooth-domains}
Let $L \in \mathcal{L}_s^{\text{hom}}(\lambda,\Lambda)$. Let $k \in \N$, and $\Omega \subset \R^n$ be an open, bounded domain with $\partial \Omega \in C^{k+1,\gamma}$ for some $\gamma \in (0,1)$ with $\gamma \not= s$, and $0 \in \partial \Omega$. Assume that $K \in C^{2k + 2 \gamma + 3}(\mathbb{S}^{n-1})$. Let $\eta \in C^{k,\gamma}(\overline{\Omega \cap B_1}) \cap C^{\infty}(\Omega \cap B_1)$. Then, there exists $c > 0$, depending only on $n,s,\lambda,\Lambda,\Omega,\gamma,k$, such that the following holds true:
\begin{itemize}
\item[(i)] If $k = 1$ and $\gamma < s$, then
\begin{align*}
|L(d^{s-1} \eta )| \le c \Vert \eta \Vert_{C^1(\overline{\Omega \cap B_{1/2}})} d^{\gamma - s} ~~ \text{ in } \overline{\Omega} \cap B_{1/2}.
\end{align*}
\item[(ii)] If $k \ge 2$ or $\gamma > s$, then
\begin{align*}
[L(d^{s-1} \eta)]_{C^{k-1-s+\gamma}(\overline{\Omega} \cap B_{1/2})} \le c \left(\Vert \eta \Vert_{C^1(\overline{\Omega \cap B_{1/2}})} + \Vert \eta \Vert_{C^{k-1+s+\gamma}(\Omega \cap B_1)} \right).
\end{align*}
\item[(iii)] If $k+\gamma > 2s$, then we have for any $x_0 \in \Omega \cap B_{1/2}$
\begin{align*}
[L(d^{s-1} \eta )]_{C^{k+\gamma -2 s}(B_{d(x_0)/2}(x_0))} \le c \left( \Vert \eta \Vert_{C^1(\overline{\Omega \cap B_{1/2}})} +  \Vert \eta \Vert_{C^{k+\gamma}(\Omega \cap B_1)} \right) d^{s-1}(x_0).
\end{align*}
\end{itemize}
\end{corollary}

\begin{proof}
Note that there exist $N \in \N$ and $\delta > 0$, $\nu_i \in \mathbb{S}^{n-1}$, $x_i \in \partial \Omega \cap B_1$, depending only on $\Omega$,  such that $\partial_{\nu_i} d \ge 1/2$ in $\overline{\Omega \cap B_{\delta}(x_i)}$, for $i \in \{1,\dots,N\}$, and such that 
\begin{align*}
\{x \in \overline{\Omega} \cap B_{1/2} : d(x) \le \delta/2 \} \subset \bigcup_{i=1}^N B_{\delta}(x_i).
\end{align*}
Then, by application of \autoref{lemma:dist-s-1-smooth-domains}(i) to $\eta := (\partial_{\nu_i} d)^{-1} \eta \in C^{k,\gamma}(\overline{\Omega \cap B_{\delta}(x_i)}) \cap C^{\infty}(\Omega \cap B_{\delta}(x_i))$, we deduce that for any $i \in \{1,\dots,N\}$
\begin{align*}
|L(d^{s-1} \eta )| \le c(| \cdot | + |\eta(x_i)| + |\nabla \eta(x_i)|) d^{\gamma - s} ~~ \text{ in } \overline{\Omega} \cap B_{\delta/2}(x_i).
\end{align*}
Thus, we have proved (i) in $\overline{\Omega} \cap B_{1/2} \cap \{ d(x) \le \delta /2 \}$. The result in $\overline{\Omega} \cap B_{1/2} \cap \{ d(x) > \delta /2 \}$ is immediate from the regularity of $K$ (see \cite[Lemma 2.2.6]{FeRo23}).\\
The proofs of (ii) and (iii) follow from \autoref{lemma:dist-s-1-smooth-domains} in an analogous way.
\end{proof}

\subsection{Barriers with boundary blow-up} 
\label{subsec:barriers}

Let us construct barrier functions that are suitable for establishing maximum principles for solutions that blow up at the boundary. We establish a subsolution and a supersolution in the following two lemmas.

\begin{lemma}
\label{lemma:subsol}
Let $L \in \mathcal{L}_s^{\text{hom}}(\lambda,\Lambda)$. Let $\Omega \subset \R^n$ be an open, bounded domain with $\partial \Omega \in C^{1,\gamma}$ for some $\gamma > 0$. Then, for any $l \in \R$, $\eps \in (0, \min \{s , 1-s \} )$, and $M > 0$ there exists $\phi_l \in C^{\infty}(\Omega)$ such that
\begin{align*}
\left\{\begin{array}{rcl}
L \phi_l &\le& - d^{\eps -s -1} - M ~~ \text{ in } \Omega,\\
\phi_l  &=& 0 ~~ \qquad\qquad\quad ~ \text{ in } \R^n \setminus \Omega,\\
\phi_l /d^{s-1} &=& l ~~ \qquad\qquad\quad ~ \text{ on } \partial \Omega.
\end{array}\right.
\end{align*}
Moreover, if $l \ge 0$, then there exists $\delta \in (0,1)$, depending only on $n,s,\lambda,\Lambda,\diam(\Omega),\eps,M$, such that $\phi_l \ge 0$ in $\Omega \cap \{d \le \delta \}$. And if $l < 0$, then for $M$ large enough, depending on $n,s,\lambda,\Lambda,\diam(\Omega),\eps$, it holds $\phi_l \le 0$ in $\Omega$.
\end{lemma}

\begin{proof}
Let $\eps \in (0,s)$ and $N > 1$ to be chosen small and large, respectively, later. We set 
\begin{align*}
\phi_l(x) := l d^{s-1}(x) - d^{s-1+ \eps}(x) - N \1_{\Omega}(x).
\end{align*}
Then, by \autoref{lemma:dist-s-1+eps},
\begin{align*}
L \phi_l \le c_1 l d^{\delta \gamma - s - 1} - c_2 d^{\eps -s -1} + c_3 - N L \1_{\Omega}.
\end{align*}
Note that since $L \1_{\Omega} \ge 0$, by taking any $\delta \in (0,s)$ and then $\eps < \delta \gamma$, we see that there exists $\eta > 0$, depending on $s,l,\eps,M,\delta,\gamma$, such that 
\begin{align*}
L \phi_l \le - d^{\eps-s-1} - M ~~ \text{ in } \Omega \cap \{ d < \eta \}.
\end{align*}
Moreover, note that there exists $c_4 > 0$, depending on $\diam(\Omega)$, such that $L\1_{\Omega} \ge c_4$ in $\Omega \cap \{ d \ge \eta\}$. Thus, choosing $N = M c_4^{-1}$, we deduce that
\begin{align*}
L \phi_l \le - d^{\eps-s-1} - M ~~ \text{ in } \Omega,
\end{align*}
as desired.
The remaining properties of $\phi_l$ follow immediately from its construction.
\end{proof}

\begin{lemma}
\label{lemma:supersol}
Let $L \in \mathcal{L}_s^{\text{hom}}(\lambda,\Lambda)$. Let $\Omega \subset \R^n$ be an open, bounded domain with $\partial \Omega \in C^{1,\gamma}$ for some $\gamma > 0$. Then, there is $c_1 > 0$, depending only on $n,s,\lambda,\Lambda$, such that for any $l \in \R$, $\eps \in (0 , \min\{s\gamma , 1-s \} )$, and $M > 0$ there exists $\psi_l \in C^{\infty}(\Omega)$ such that
\begin{align*}
\left\{\begin{array}{rcl}
L \psi_l &\ge& c_1 d^{\eps -s -1} + M ~~ \text{ in } \Omega,\\
\psi_l  &=& 0 ~~ \qquad\qquad\quad ~ \text{ in } \R^n \setminus \Omega,\\
\psi_l /d^{s-1} &=& l ~~ \qquad\qquad\quad ~ \text{ on } \partial \Omega.
\end{array}\right.
\end{align*}
Moreover, for any $M > 0$, if $l > 0$ is large enough, depending only on $n,s,\lambda,\Lambda,\diam(\Omega),\eps$, it holds $\psi_l \ge 0$ in $\Omega$.

Moreover, for any $\eps \in (0,s)$, there is $\tilde{\psi} \in C^s(\overline{\Omega})$ such that for some $c_2 > 0$
\begin{align*}
\begin{cases}
L \tilde{\psi} &\ge d^{\eps-s} ~~ \text{ in } \Omega,\\
\tilde{\psi} &= 0 ~~~~~ \text{ in } \R^n \setminus \Omega,\\
\tilde{\psi}/d^{s-1} &= 0 ~~~~~ \text{ on } \partial \Omega,\\
\partial_{\nu}(\tilde{\psi}/d^{s-1}) &\le c_2 ~~~~ \text{ on } \partial \Omega.
\end{cases}
\end{align*}
\end{lemma}

\begin{proof}
Let $l \in \R$ and $\eps \in (0 , \min\{s\gamma , 1-s \} )$. The proof is similar to the one of \autoref{lemma:subsol}. We set
\begin{align*}
\psi_l(x) = l d^{s-1}(x) + d^{s-1+\eps}(x) + C_2 \1_{\overline{\Omega}}(x)
\end{align*}
and since $\eps < s\gamma$, we can choose $\delta \in (0,s)$ such that $\eps < \delta \gamma$ and take $C_2 > 0$ to be chosen later. Note that by \autoref{lemma:dist-s-1+eps}:
\begin{align*}
L \psi_l \ge -c_1 l d^{\delta \gamma -s -1} + c_2 d^{\eps - s - 1} - c_3 + c_4 C_2 ~~ \text{ in } \Omega,
\end{align*}
for some constants $c_1,c_2,c_3,c_4 > 0$, depending only on $n,s,\lambda,\Lambda,\delta$, and the $C^{1,\gamma}$ radius of $\Omega$. 
Thus, if we choose $C_2 > 0$ large enough, depending on $M,l,c_1,c_2,c_3,c_4,\eps, \diam(\Omega)$, then we deduce
\begin{align*}
L \psi_l \ge c d^{\eps - s - 1} + M ~~ \text{ in } \Omega.
\end{align*}

Finally, we observe that upon choosing $l > 0$ large enough, depending only on $\eps, \diam(\Omega)$, we have
\begin{align*}
\psi_l \ge l d^{s-1} + d^{s-1+\eps} \ge 0 ~~ \text{ in } \Omega.
\end{align*}

For the second claim, we recall from \cite[Lemma B.2.6]{FeRo23} that for any $\eps \in (0,s)$, there exist $N > 0$ and $c_1 > 0$ such that
\begin{align*}
L (-Nd^{s+\eps}) \ge d^{\eps-s} - c_1 ~~ \text{ in } \Omega.
\end{align*}
Let $\tilde{\psi}_2 \in L^{\infty}(\Omega)$ be the solution to the Dirichlet problem
\begin{align*}
\begin{cases}
L \tilde{\psi}_2 &= c_1 ~~ \text{ in } \Omega,\\
\tilde{\psi}_2 &= 0 ~~ \text{ in } \R^n \setminus \Omega,
\end{cases}
\end{align*}
and observe that by the boundary regularity theory from \cite{FeRo23}, it holds $\tilde{\psi}_2 \in C^{s}(\overline{\Omega})$, and hence for $\tilde{\psi} := -N d^{s+\eps} - \tilde{\psi}_2$, we obtain
\begin{align*}
\frac{\tilde{\psi}}{d^{s-1}} = 0 ~~ \text{ on } \partial \Omega.
\end{align*}
Therefore, for some $c_2 > 0$,
\begin{align*}
|\partial_{\nu}(\tilde{\psi} / d^{s-1})| = (1-s)|\tilde{\psi} /d^s| \le c_2 ~~ \text{ on } \partial \Omega,
\end{align*}
as desired.
\end{proof}

\subsection{Nonlocal equations up to a polynomial}
\label{subsec:poly}

We will need the following definition of nonlocal equations that hold true up to a polynomial. It was introduced in \cite{DSV19} for the fractional Laplacian and the theory was extended in \cite{DDV22} to general nonlocal operators (see also \cite{AbRo20}).

\begin{definition}
\label{def:L-up-to-a-polynomial}
For $k \in \N$, a bounded domain $\Omega \subset \R^n$, $f \in C(\Omega)$, and $K \in C^{k-1+\delta}(\mathbb{S}^{n-1})$ for some $\delta > 0$, we say that a function $u \in C(\Omega) \cap L^1_{2s+k}(\R^n)$ solves in the viscosity sense
\begin{align*}
L u \overset{k}{=} f ~~ \text{ in } \Omega,
\end{align*}
if there exist polynomials $(p_R)_{R > 1} \in \mathcal{P}_{k-1}$ of degree $k-1$, and functions $(f_R)_{R > 1}$ such that
\begin{align*}
L(u \1_{B_R}) = f_R + p_R ~~ \text{ in } \Omega, ~~ \forall R > \diam(\Omega),\\
\Vert f_R - f \Vert_{L^{\infty}(\Omega)} \to 0 ~~ \text{ as } R \to \infty.
\end{align*}
\end{definition}

\begin{remark}$ $\\
\vspace{-0.6cm}
\begin{itemize}
\item In case $k = 0$, we set $\cP_{0-1} = \cP_{-1} = \{ 0 \}$. Then, $L u \overset{0}{=} f$ is equivalent to $L u = f$ (see \cite[Corollary 2.13]{DDV22}).
\item Instead of $K \in C^{k}(\mathbb{S}^{n-1})$, here we only assume $K \in C^{k-1+\delta}(\mathbb{S}^{n-1})$ for some $\delta > 0$. It is easy to see that all the arguments in \cite{DDV22} remain valid under this weaker assumption. We decided to make this change in order to have optimal assumptions on $K$ in \autoref{thm:higher-reg-half-space}.
\item As in \cite{AbRo20}, we assume uniform convergence $f_R \to f$. This is slightly different from \cite{DSV19}, where pointwise convergence was assumed. 
\end{itemize}
\end{remark}

The following lemma is a slight improvement of \cite[Lemma 3.6]{AbRo20} (see also \cite[Lemma 8.1]{RTW25}) in the sense that the estimate involves a weighted $L^1$ norm instead of a weighted $L^{\infty}$ norm.

\begin{lemma}
\label{lemma:interior-regularity}
Let $L \in \mathcal{L}_s^{\text{hom}}(\lambda,\Lambda)$. Let $u \in C(B_1)$ be a viscosity solution to 
\begin{align*}
L u = f ~~ \text{ in } B_1.
\end{align*}
Then, the following holds true:
\begin{itemize}
\item[(i)] Let $\beta \in (0,2s]$ if $s \not= 1/2$ and $\beta \in (0,1)$ if $s = 1/2$. If $f \in C(B_1)$ and $u \in L^1_{2s}(\R^n)$, then it holds $u \in C^{\beta}_{loc}(B_1)$ and
\begin{align*}
\Vert u \Vert_{C^{\beta}(B_{1/2})} \le c \left( \Vert u \Vert_{L^{\infty}(B_1)} + \int_{\R^n \setminus B_{1}} \frac{|u(y)|}{|y|^{n+2s}}  \d y + \Vert f \Vert_{L^{\infty}(B_1)} \right)
\end{align*}
for some $c > 0$, depending only on $n,s,\lambda,\Lambda,\beta$.
\item[(ii)] If $f \in C^{\alpha}(B_1)$ for some $\alpha > 0$ such that $2s + \alpha \not\in \N$, $K \in C^{\alpha}(\mathbb{S}^{n-1})$, and $u \in L^{1}_{2s+\alpha}(\R^n)$ then $u \in C^{2s+\alpha}_{loc}(B_1)$ and
\begin{align*}
\Vert u \Vert_{C^{2s+\alpha}(B_{1/2})} \le c \left( \Vert u \Vert_{L^{\infty}(B_1)} + \int_{\R^n \setminus B_{1}} \frac{|u(y)|}{|y|^{n+2s+\alpha}}  \d y + [ f ]_{C^{\alpha}(B_1)} \right)
\end{align*}
for some $c > 0$, depending only on $n,s,\lambda,\Lambda,\alpha$.
\end{itemize}
\end{lemma}

\begin{remark}
\label{remark:int-up-to-poly}
From the proof it is apparent, that \autoref{lemma:interior-regularity}(ii) remains true if $L u \overset{k}{=} f$ for $k < \alpha$.
\end{remark}

\begin{proof}
Let us first show (ii) in case $\alpha < 1$. Let us define $v = u \1_{B_1}$. We claim that $v$ solves $L v = \tilde{f}$ in $B_{3/4}$ for some $\tilde{f} \in C^{2s+\alpha}(B_{3/4})$ with
\begin{align}
\label{eq:tilde-f-claim}
\Vert \tilde{f} \Vert_{C^{\alpha}(B_{3/4})} \le C \left(\Vert u \Vert_{L^{\infty}(B_1)} + \int_{\R^n \setminus B_{1}} \frac{|u(y)|}{|y|^{n+2s+\alpha}} \d y + [f]_{C^{\alpha}(B_{3/4})} \right).
\end{align}
To prove it, first, we observe that for any $h \in B_{3/4}$ and $x \in B_{\frac{3}{4} - |h|}$, using that $K \in C^{\alpha}(\mathbb{S}^{n-1})$:
\begin{align*}
|L(u \1_{\R^n \setminus B_1})(x) - L(u \1_{\R^n \setminus B_1})(x+h)| &\le |h|^{\alpha} \int_{\R^n \setminus B_1} |u(y)| \frac{|K(x-y)-K(x+h-y)|}{|h|^{\alpha}} \d y \\
&\le c |h|^{\alpha} \int_{\R^n \setminus B_1} \frac{|u(y)|}{|y|^{n+2s+\alpha}} \d y.
\end{align*}

Therefore, since we can add and subtract constants to $\tilde{f}$ without affecting the left hand side of the next estimate, for any $h \in B_{3/4}$ it holds:
\begin{align}
\label{eq:osc-tilde-f}
\Vert \tilde{f} - \tilde{f}(\cdot + h) \Vert_{L^{\infty}(B_{\frac{3}{4} -|h|})} \le C \left( \osc_{B_{3/4}} f + |h|^{\alpha} \int_{\R^n \setminus B_{1}} \frac{|u(y)|}{|y|^{n+2s+\alpha}} \d y \right).
\end{align}
From here, we deduce that there exist $g \in L^{\infty}(B_{3/4})$ and a constant $p$ such that $\tilde{f} = g + p$. By construction we have
\begin{align*}
\Vert \tilde{g} \Vert_{L^{\infty}(B_{3/4})} \le C \left(\Vert u \Vert_{L^{\infty}(B_1)} + \int_{\R^n \setminus B_{1}} \frac{|u(y)|}{|y|^{n+2s+\alpha}} \d y + \osc_{B_{3/4}} f \right).
\end{align*}
We split $v = v_1 + v_2$, where $v_1$ and $v_2$ are solutions to
\begin{align*}
\begin{cases}
L v_1 & = \tilde{g} ~~ \text{ in } B_{3/4},\\
v_1 & = v ~~ \text{ in } \R^n \setminus B_{3/4}, 
\end{cases}
\qquad 
\begin{cases}
L v_2 & = p ~~ \text{ in } B_{3/4},\\
v_2 & = 0 ~~ \text{ in } \R^n \setminus B_{3/4}, 
\end{cases}
\end{align*}
and note that the existence of $v_1,v_2$ follows from \cite[Theorem 3.2.27]{FeRo23}). Then, by the maximum principle (see \cite[Corollary 3.2.22]{FeRo23}) we deduce that
\begin{align*}
\Vert v_1 \Vert_{L^{\infty}(B_{3/4})} &\le C \left(\Vert v \Vert_{L^{\infty}(\R^n \setminus B_{3/4})} + \Vert \tilde{g} \Vert_{L^{\infty}(B_{3/4})} \right) \\
&\le C \left(\Vert u \Vert_{L^{\infty}(B_1)} + \int_{\R^n \setminus B_{1}} \frac{|u(y)|}{|y|^{n+2s+\alpha}} \d y + \osc_{B_{3/4}} f \right).
\end{align*}
Hence,
\begin{align*}
\Vert v_2 \Vert_{L^{\infty}(B_{3/4})} &\le \Vert u \Vert_{L^{\infty}(B_{3/4})} + \Vert v_1 \Vert_{L^{\infty}(B_{3/4})} \\
&\le C \left(\Vert u \Vert_{L^{\infty}(B_1)} + \int_{\R^n \setminus B_{1}} \frac{|u(y)|}{|y|^{n+2s+\alpha}} \d y + \osc_{B_{3/4}} f \right).
\end{align*}
Then, by \cite[Lemma 3.7]{AbRo20}, we deduce 
\begin{align*}
\Vert p \Vert_{L^{\infty}(B_{3/4})} \le C \Vert v_2 \Vert_{L^{\infty}(B_{3/4})} \le C \left(\Vert u \Vert_{L^{\infty}(B_1)} + \int_{\R^n \setminus B_{1}} \frac{|u(y)|}{|y|^{n+2s+\alpha}} \d y + \osc_{B_{3/4}} f \right).
\end{align*}
Altogether, we have shown
\begin{align*}
\Vert \tilde{f} \Vert_{L^{\infty}(B_{3/4})} \le \Vert \tilde{g} \Vert_{L^{\infty}(B_{3/4})} + \Vert p \Vert_{L^{\infty}(B_{3/4})}  \le C \left(\Vert u \Vert_{L^{\infty}(B_1)} + \int_{\R^n \setminus B_{1}} \frac{|u(y)|}{|y|^{n+2s+\alpha}} \d y + \osc_{B_{3/4}} f \right).
\end{align*}
Finally, as a direct consequence of \eqref{eq:osc-tilde-f}, we deduce
\begin{align*}
[\tilde{f}]_{C^{\alpha}(B_{3/4})} \le  C \left(\Vert u \Vert_{L^{\infty}(B_1)} + \int_{\R^n \setminus B_{1}} \frac{|u(y)|}{|y|^{n+2s+\alpha}} \d y + [f]_{C^{\alpha}(B_{3/4})} \right),
\end{align*}
which yields the claim \eqref{eq:tilde-f-claim}.\\
Thus, by application of the interior regularity estimate \cite[Theorem 2.4.1]{FeRo23} to $v$, we obtain
\begin{align*}
\Vert u \Vert_{C^{2s+\alpha}(B_{1/2})} &= \Vert v \Vert_{C^{2s+\alpha}(B_{1/2})} \le c \left( \Vert v \Vert_{L^{\infty}(\R^n)} + \Vert \tilde{f} \Vert_{C^{\alpha}(B_{3/4})} \right) \\
&\le c \left( \Vert u \Vert_{L^{\infty}(B_1)} + \int_{\R^n \setminus B_{1}} \frac{|u(y)|}{|y|^{n+2s+\alpha}} \d y + [f]_{C^{\alpha}(B_{3/4})} \right),
\end{align*}
as desired. This proves (ii) in case $\alpha < 1$. The case $\alpha \ge 1$ goes in the same way by considering higher order incremental quotients in the arguments above. Statement (i) was proved in \cite[Theorem 2.4.3]{FeRo23}. The $L^{\infty}$ norm can be replaced by the $L^1_{2s}(\R^n)$ norm by the same truncation argument we employed above.
\end{proof}

Next, we provide a lemma stating that equations up to a polynomial can be differentiated in the same way as classical nonlocal equations. This lemma will be used in the proof of \autoref{lemma:large-solution-linear-fct}.

\begin{lemma}
\label{lemma:derivative-up-to-a-polynomial}
Let $L \in \mathcal{L}_s^{\text{hom}}(\lambda,\Lambda)$. Let $k \in \N$, $f \in C^{1}(B_1)$, and $K \in C^{k+\delta}(\mathbb{S}^{n-1})$ for some $\delta > 0$.  Let $u \in C(B_1) \cap L^1_{2s+k+\delta}(\R^n)$ with $\partial_i u \in L^1_{2s+k-1+\delta}(\R^n)$. Then, it holds
\begin{align*}
L u \overset{k+1}{=} f ~~ \text{ in } B_1,
\end{align*}
and $\partial_i f_R \to \partial_i f$, if and only if
\begin{align*}
L (\partial_i u) \overset{k}{=} \partial_i f ~~ \text{ in } B_1.
\end{align*}
\end{lemma}

\begin{proof}
Let us assume first that $\partial_i u \in L^1_{2s+k-1+\delta}(\R^n)$, and assume that $L u \overset{k+1}{=} f$ in $B_1$. Then, there exist polynomials $p_R \in \mathcal{P}_k$ and functions $f_R \in L^{\infty}(B_1)$ with $f_R \to f$ such that
\begin{align*}
L(u \1_{B_R}) = f_R + p_R ~~ \text{ in } B_1.
\end{align*}
Let us now consider difference quotients $D^h_i u(x) = \frac{u(x+ e_i h) - u(x)}{|h|}$ and compute
\begin{align*}
L(D^h_i u \1_{B_R}) = L(D^h_i(u \1_{B_R})) - L(u D^h_i \1_{B_R}) = D^h_i f_R + D^h_i p_R - L(u D^h_i \1_{B_R}),
\end{align*}
where, by following the proof of \cite[Theorem 2.1]{DDV22}, we can decompose
\begin{align*}
- L(u D^h_i \1_{B_R})(x) &= \int_{\R^n \setminus B_3} u D^h_i\1_{B_R}(y) K(x-y) \d y = d_{R,h}(x) + g_{R,h}(x)
\end{align*}
for functions $g_{R,h}$ such that
\begin{align*}
g_{R,h}(x) = \int_{\R^n} D^h_i(\1_{B_R})(y)u(y) \psi(x,y) \d y = - \int_{B_R} (D^i_{-h}u(y) \psi(x,y) + u(y) D^i_{-h} \psi(x,y)) \d y
\end{align*}
for some function $\psi : B_1 \times (\R^n \setminus B_3) \to \R$ such that 
\begin{align*}
\sup_{x \in B_1} \psi(x,y) \le C \sup_{x \in B_1}(1 + |x-y|)^{-(n+2s+k-1+\delta)}, \quad \sup_{x \in B_1} |\nabla_y \psi(x,y)| \le C \sup_{x \in B_1}(1 + |x-y|)^{-(n+2s+k+\delta)},
\end{align*}
and polynomials $d_{R,h} \in \mathcal{P}_{k-1}$ with
\begin{align*}
d_{R,h}(x) = \sum_{|\alpha| \le k-1} \kappa_{\alpha,h} x^{\alpha}, \qquad \kappa_{\alpha,h} = c_{\alpha}\int_{B_R} D_{-h}^i[u(y) \partial_x^{\alpha}K(x-y)] \d y, \qquad c_{\alpha} \in \R.
\end{align*}
Clearly, it holds $g_{R,h} \to g_R$, and $d_{R,h} \to d_{R}$, as $R \to \infty$, where
\begin{align*}
g_{R}(x) = \int_{B_R} \partial_i[u(y) \psi(x,y)] \d y, \qquad  d_{R}(x) = \sum_{|\alpha| \le k-1} \kappa_{\alpha} x^{\alpha}, \qquad \kappa_{\alpha} = c_{\alpha}\int_{B_R} \partial_i[u(y) \partial_x^{\alpha}K(x-y)] \d y.
\end{align*}
Note that for the convergence $g_{R,h} \to g_R$ we are using that for any $x \in B_1$
\begin{align*}
|\1_{B_R}(y) & (D^i_{-h}u(y) \psi(x,y) + \1_{B_R}(y) u(y) D^i_{-h}\psi(x,y))|\\
& \le C |\partial_i u(y)| \sup_{x \in B_1}(1 + |x-y|)^{-(n+2s+k-1+\delta)} + C |u(y)| \sup_{x \in B_1}(1 + |x-y|)^{-(n+2s+k+\delta)} \in L^1(\R^n)
\end{align*}
and dominated convergence. Moreover, from integrating by parts, we see that it holds for any $x \in B_1$
\begin{align*}
\int_2^{\infty} R^{-1} |g_{R}^{(2)}(x)|  \d R &\le \int_2^{\infty} R^{-1} \int_{\partial B_{R}} |u(y)||x-y|^{-(n+2s+k-1+\delta)}\d y \d R \\
&\le c \int_{B_2^{c}} |u(y)||y|^{-(n+2s+k+\delta)} \d y < \infty,
\end{align*}
which implies that $g_R(x) \to 0$, as $R \to \infty$, uniformly in $x$.\\
Altogether, we have shown
\begin{align*}
L(\partial_i u \1_{B_R}) = \lim_{h \to 0} D_i^h f_R + \lim_{h \to 0} D^h_i p_R + \lim_{h \to 0} d_{R,h} + \lim_{h \to 0} g_{R,h} = \partial_i f_R + \partial_i p_R + d_R + g_R,
\end{align*}
which implies that
\begin{align*}
L(\partial_i u \1_{B_R}) \overset{k}{=} f ~~ \text{ in } B_1,
\end{align*}
as desired. \\
Let us now show the other implication, i.e., assume that $L(\partial_i u) \overset{k}{=} \partial_i f$ in $B_1$. Then, by \cite{DDV22} we observe that there are $F_R : \R^n \to \R$, and $P_R \in \cP_{k}$ such that 
\begin{align*}
L(u \1_{B_R}) = F_R + P_R ~~ \text{ in } B_1.
\end{align*}
Clearly, by the same arguments as above, we have
\begin{align*}
L((D_i^h u)\1_{B_R}) = D_i^h L(u\1_{B_R}) - L(u D_i^h \1_{B_R}) = D_i^h F_R + D_i^h P_R + d_{R,h} + g_{R,h}
\end{align*}
with $D_i^h P_R + d_{R,h} \in \cP_{k-1}$ and $g_{R,h} \to g_R$, as $h \to 0$ with $g_R \to 0$, as $R \to \infty$. Thus, by the stability for viscosity solutions up to a polynomial (see \cite[Lemma 3.5]{AbRo20}), we have that
\begin{align*}
f_R + p_R = L(\partial_i u \1_{B_R}) = \partial_i F_R + \partial_i P_R + d_R + g_R,
\end{align*}
where $d_R,p_R,\partial_i P_R \in \cP_{k-1}$. Hence, after integrating the previous identity in $x_i$ and denoting $\tilde{F}_R(x) = \int_{-\infty}^{x_i} (f_R - g_R)(x',y_i) \d y$, we can deduce:
\begin{align*}
F_R = \tilde{F}_R + \tilde{P}_R ~~ \text{ in } B_1,
\end{align*}
where $\tilde{P}_R \in \cP_{k}$ is such that $\partial_i \tilde{P}_R = p_R - d_R - \partial_i P_R$. Then, since $f_R \to f$ and $g_R \to 0$, as $R \to \infty$, we deduce that $\tilde{F}_R \to F$, where $\partial_i F = f$, and the proof is complete.
\end{proof}

\subsection{Two lemmas on viscosity solutions}
\label{subsec:two-visc}

In this section, we prove two auxiliary lemmas for viscosity solutions to nonlocal equations with local Neumann boundary data, namely a stability result, and that sums of viscosity subsolutions are again viscosity subsolutions. Both results are standard for nonlocal equations in the interior of the solution domain (see \cite{FeRo23}). However, since we consider equations at the boundary, where solutions satisfy a Neumann condition in the viscosity sense, both results require a proof. Both proofs heavily rely on the interaction of nonlocal operators with the distance function and the results in Subsection \ref{subsec:prelim-dist}.

First, we prove a stability result, which will be crucial in the blow-up argument of our proof of the higher boundary regularity. 

\begin{lemma}
\label{lemma:stability}
Let $k \in \N \cup \{ 0 \}$, $\gamma \in (0,1)$ with $\gamma \not= s$, and $\Omega_j \subset \R^n$ be open, bounded domains with $\partial \Omega_j \in C^{2,\gamma}$ such that $0 \in \partial \Omega_j$, $\nu_0 = e_n$ for any $j \in \N$, and such that the $C^{2,\gamma}$ radii of $\Omega_j$ and $\diam(\Omega_j)$ are uniformly bounded. Given a sequence $r_j \searrow 0$, we set $\tilde{\Omega}_j = r_j^{-1}\Omega_j$ and  $\tilde{d}_j := d_{\tilde{\Omega}_j}$. Let $v_j \in L^{1}_{2s+k}(\R^n)$ with $v_j/\tilde{d}_j^{s-1} \in C(\overline{\tilde{\Omega}_j})$ be viscosity solutions to
\begin{align*}
\left\{\begin{array}{rcl}
L_j v_j &\overset{k}{=}& f_j ~~~ \text{ in } \tilde{\Omega}_j \cap B_{1},\\
v_j &=& 0 ~~~~ \text{ in } \R^n \setminus \tilde{\Omega}_j,\\
\partial_{\nu}(v_j/\tilde{d}_j^{s-1}) &=& g_j ~~~ \text{ on } \partial \tilde{\Omega}_j \cap B_1,
\end{array}\right.
\end{align*}
where $f_j \in C(\tilde{\Omega}_j \cap B_1)$, $g_j \in C(\partial \tilde{\Omega}_j \cap B_1)$, and $(L_j)_j \subset \mathcal{L}^{\text{hom}}_s(\lambda,\Lambda,k-1+\alpha)$ for some $\alpha > 0$. Moreover, assume that there are $v \in L^{1}_{2s+k}(\R^n)$ with $v/(x_n)_+^{s-1} \in C(\{ x_n \ge 0 \})$, $f \in C(\{ x_n > 0 \} \cap B_1)$, $g \in C(\{ x_n = 0\} \cap B_1)$, $L \in \mathcal{L}^{\text{hom}}_s(\lambda,\Lambda,k-1+\alpha)$, and $\eps_j \searrow 0$, $q_j \in \cP_{k}$ such that 
\begin{align*}
v_j/\tilde{d}_j^{s-1} &\to v/(x_n)_+^{s-1} ~~~ \text{ in } L^{\infty}_{loc}(B_1),\\
v_j &\to v ~~ \qquad \qquad  \text{ in } L^1_{2s+k}(\R^n),\\
|f_j - p_j - f| &\to 0 ~~ \qquad \qquad  \text{ in } L^{\infty}_{loc}(B_1 \cap \{ x_n > 0 \} ),\\
|g_j - q_j - g|(x) \le c \eps_j &\to 0 ~~ \qquad \qquad ~  \forall x \in \partial \tilde{\Omega}_j \cap B_1,\\
K_j &\to K ~~ \qquad \quad ~~  \text{ in } C^{k-1+\alpha}(\mathbb{S}^{n-1}).
\end{align*} 
Then, there exists $q \in \cP_k$ such that $v$ is a viscosity solution to
\begin{align*}
\left\{\begin{array}{rcl}
Lv &\overset{k}{=}& f ~~~~ \quad \text{ in } B_1 \cap \{ x_n > 0\},\\
v &=& 0 ~~ \quad ~~ \text{ in } \R^n \setminus \{ x_n > 0 \},\\
\partial_n (v/(x_n)_+^{s-1}) &=& g + q ~~ \text{ on } B_1 \cap \{ x_n = 0 \}.
\end{array}\right.
\end{align*}
If $k = 0$, the same result holds with $\tilde{d}_j^{s-\gamma} f_j \in L^{\infty}(\tilde{\Omega}_j \cap B_1)$ and $(x_n)_+^{s-\gamma}f \in L^{\infty}(\{ x_n > 0\} \cap B_1)$.
\end{lemma}

\begin{proof}
Let us define $u_j := v_j/\tilde{d}_j^{s-1}$ and $u := v/(x_n)_+^{s-1}$.
Note that since $u_j \to u$ in $L^{\infty}_{loc}(B_1)$ it follows that $v_j = \tilde{d}_j^{s-1} u_j \to (x_n)_+^{s-1} u = v$ in $L^{\infty}_{loc}(B_1 \cap \{ x_n > 0 \})$. This property is enough to use the stability of viscosity solutions from \cite[Proposition 3.2.12]{FeRo23} to $v_j$ and $v$. The higher order version which we require here follows from \cite[Lemma 3.5]{AbRo20}. Since $v_j \to v$ in $L^1_{2s}(\R^n)$, we also have that $v = 0$ in $B_1 \setminus \{ x_n > 0\}$. Consequently, it only remains to prove the convergence of the Neumann boundary condition.\\
To do so, let $x_0 \in B_1 \cap \{ x_n = 0 \}$. In case $k \ge 1$, we  first truncate $v$ and $v_j$ in $B_2(x_0)$ and apply \cite[Lemma 3.6]{AbRo20} to obtain the equations satisfied by $v\1_{B_2}(x_0)$ and $v_j\1_{B_2}(x_0)$. In order not to over-complicate the notation, let us denote the truncations still by $v$ and $v_j$ and the corresponding source terms by $f$ and $f_j$. Then, let $\phi \in C^2(B_{r}(x_0))$ for some $r \in (0,1)$ with $\phi \le u$ in $B_{r}(x_0)$, $\phi(x_0) = u(x_0)$, and $\phi \equiv u$ in $\R^n \setminus \overline{B_{r}(x_0)}$ be a test function. 
Given $\delta \in (0,1)$, $\eta \in (0,\gamma)$, we define now
\begin{align*}
\psi^{(\delta)}(x) &=  - \delta \1_{B_{r}(x_0)}(x) \left[(x_n)_+ -  (x_n)_+^{1 + \eta} \right], \qquad \psi^{(\delta)}_j(x) = - \delta\1_{B_{r}(x_0)}(x) \left[ \tilde{d}_j(x) - \tilde{d}_j^{1+\eta}(x) \right].
\end{align*}
Note that there exist $C > 0$ and $\eps \in (0,r/2)$, independent of $\delta,j$, such that
\begin{align}
\label{eq:stability-operator-estimate}
L_j(\tilde{d}_j^{s-1} \psi^{(\delta)}_j) \le -C \delta \tilde{d}_j^{\eta - s} ~~ \text{ in } \tilde{\Omega}_j \cap B_{\eps}(x_0).
\end{align}
This is due to \cite[Proposition B.2.1, Lemma B.2.6, Corollary B.2.8]{FeRo23}, and since $\tilde{\Omega}_j \cap B_R(x_0)$ and the respective $C^{2,\gamma}$-radii of $\tilde{\Omega}_j$ are uniformly bounded. Indeed, the aforementioned results yield the existence of $\eps_0 > 0$ such that
\begin{align*}
- \delta L_j(\tilde{d}_j^{s-1} [ \tilde{d}_j - \tilde{d}_j^{1+\eta} ] ) \le -c_1 \delta \tilde{d}_j^{\eta - s } ~~ \text{ in } \tilde{\Omega}_j \cap B_{\eps_0}(x_0).
\end{align*}
Moreover, one computes by scaling from $\tilde{\Omega}_j$ to $\Omega_j$, denoting $d_j = d_{\Omega_j}$, and applying \autoref{lemma:dist-int-est}
\begin{align}
\label{eq:stability-operator-estimate-truncation}
\begin{split}
- \delta L_j(\tilde{d}_j^{s-1} [ \tilde{d}_j - \tilde{d}_j^{1+\eta} ] \1_{\R^n \setminus B_r(x_0)} ) &\le c \delta \int_{\tilde{\Omega}_j \setminus B_r(x_0)} (\tilde{d}_j^{s}(y) + \tilde{d}_j^{s+\eta}(y)) |y|^{-n-2s} \d y \\
&\le c \delta \int_{\Omega_j \setminus B_{r r_j}(x_0)} (r_j^{s} d_{j}^s(y) + r_j^{s-\eta} d_{j}^{s+\eta}(y)) |y|^{-n-2s} \d y \\
&\le c_2 \delta (1 + r^{-s} + r^{\eta - s}) \qquad \text{ in } \tilde{\Omega}_j \cap B_{r/2}(x_0),
\end{split}
\end{align}
where $c_2 > 0$ might depend on $\diam(\Omega_j)$, which we assumed to be bounded, but not on $j$.
Thus, by combination of the previous two computations, we deduce \eqref{eq:stability-operator-estimate} upon choosing $\eps < \eps_0$ if necessary.\\
Moreover, it is immediate by construction that 
\begin{align}
\label{eq:psi-nonnegative}
\psi^{(\delta)} &\le 0 ~~\quad \qquad \text{ in } \R^n.
\end{align}
Next, we set $\phi^{(\delta)} := \phi + \psi^{(\delta)}$.
Note that for any $\delta > 0$, it still holds $\phi^{(\delta)} \le u$ by \eqref{eq:psi-nonnegative}, and $\phi^{(\delta)}(x_0) = u(x_0)$, however $u - \phi^{(\delta)}$ has a strict minimum at $x_0$ in $\overline{B_r(x_0)}$.\\
It suffices to prove for any $\delta > 0$ small enough
\begin{align}
\label{eq:stability-final}
\partial_n \phi^{(\delta)}(x_0) \le g(x_0) + q(x_0),
\end{align}
since then it follows that $\partial_n \phi(x_0) = \partial_n \phi^{(\delta)}(x_0) + \delta \le g(x_0) + q(x_0) + \delta$, and we obtain the desired result upon taking the limit $\delta \searrow 0$.\\
Let us now construct test functions $\phi_j^{(\delta)}$ for any $j \in \N$ as follows
\begin{align*}
\left\{\begin{array}{rcl}
\phi_j^{(\delta)} &=& u_j = u_j + \psi_j^{(\delta)} ~~ \text{ in } \R^n \setminus \overline{B_r(x_0)},\\
\phi_j^{(\delta)} &=&  \phi + c_j + \psi_j^{(\delta)}  ~~ ~~ \text{ in } \overline{B_r(x_0)},
\end{array}\right.
\end{align*}
where 
\begin{align*}
c_j = \min \left\{ c \in \R : \phi + c + \psi_j^{(\delta)} \le u_j ~~ \text{ in } \overline{B_r(x_0)} \right\}.
\end{align*}
Since $\psi^{(\delta)}_j \to \psi^{(\delta)}$ (lower half-relaxed limits) in $\overline{B_{r}(x_0)}$, we obtain that $c_j \to 0$, and that there exist $x_j \in \overline{B_{r}(x_0)}$ with $x_j \to x_0$ such that $\phi_j^{(\delta)}(x_j) = u_j(x_j)$ and $\phi^{(\delta)}_j \le u_j$ by \cite[Lemma 3.2.10 and Proof of Proposition 3.2.12]{FeRo23}.

Next, we argue that $x_j \in \partial \tilde{\Omega}_j \cap B_1$. Without loss of generality, we can assume that $x_j \in B_{\eps}(x_0)$ upon taking $j \in \N$ large enough. In fact, if $x_j \in \tilde{\Omega}_j \cap B_{\eps}(x_0)$, then we can compute using \autoref{cor:dist-s-1-smooth-domains}(i), and \eqref{eq:stability-operator-estimate} 
\begin{align}
\label{eq:L-phi-estimate-2}
\begin{split}
L_j (\tilde{d}_j^{s-1} \phi_j^{(\delta)} )(x_j) &= L_j (\tilde{d}_j^{s-1} \phi \1_{B_r(x_0)})(x_j) +  c_j L_j(\tilde{d}_j^{s-1} \1_{B_r(x_0)})(x_j)\\
&\quad +  L_j(v_j \1_{\R^n \setminus B_r(x_0)})(x_j) + L_j(\tilde{d}_j^{s-1} \psi_j^{(\delta)})(x_j) \\
&\le L_j(\tilde{d}_j^{s-1} \phi)(x_j) + c_j L_j(\tilde{d}_j^{s-1})(x_j) \\
&\quad + L_j(\tilde{d}_j^{s-1} u_j \1_{\R^n \setminus B_r(x_0)})(x_j) - c_j L_j(\tilde{d}_j^{s-1} \1_{\R^n \setminus B_r(x_0)})(x_j) \\
&\quad + c \Vert v_j \Vert_{L^1_{2s}(\R^n)} - C \delta \tilde{d}_j^{\eta-s}(x_j)\\
&\le C_r \tilde{d}_j^{\gamma - s}(x_j) - C \delta \tilde{d}_j^{\eta - s}(x_j)
\end{split}
\end{align}
for some constant $C_r > 0$, depending also on $\Vert v_j \Vert_{L^1_{2s}(\R^n)}$ and $\Vert v \Vert_{L^1_{2s}(\R^n)}$. Note that to estimate the fourth term in the last estimate, we used an argument similar to \eqref{eq:stability-operator-estimate-truncation}, namely
\begin{align*}
-L_j(\tilde{d}_j^{s-1} \1_{\R^n \setminus B_r(x_0)})(x_j) &\le c \int_{\tilde{\Omega}_j \setminus B_r(x_0)} \tilde{d}_j^{s-1}(y) |y|^{-n-2s} \d y \\
&\le c \int_{\Omega_j \setminus B_{rr_j}(x_0)} r_j^{s+1} d_j^{s-1}(y) |y|^{-n-2s} \d y \le c r^{-s-1} =: c_r 
\end{align*}
for some $c_r > 0$, where we applied \autoref{lemma:dist-int-est}.
Let us now recall that $\eta < \gamma$. 

Hence, upon making $\eps > 0$ even smaller, we can have in $\tilde{\Omega}_j \cap B_{\eps}(x_0)$:
\begin{align*}
C \delta \tilde{d}_j^{\eta-\gamma} >  (C_r + \1_{\{k = 0\}}\Vert \tilde{d}_j^{s-\gamma} f_j \Vert_{L^{\infty}(\tilde{\Omega}_j \cap B_1)} + \1_{\{ k \ge 1\}} \Vert f_j \Vert_{L^{\infty}(\tilde{\Omega}_j \cap B_1)}).
\end{align*}
Then it holds
\begin{align}
\label{eq:stability-contradiction}
L_j (\tilde{d}_j^{s-1} \phi_j^{(\delta)} )(x_j) < - \1_{\{ k = 0\}} \tilde{d}_j^{\gamma -s}(x_j) \Vert \tilde{d}_j^{s-\gamma} f_j \Vert_{L^{\infty}(\tilde{\Omega}_j \cap B_1)} - \1_{\{ k \ge 1\}} \Vert f_j \Vert_{L^{\infty}(\tilde{\Omega}_j \cap B_1)} < f_j(x_j).
\end{align}
However, note that by construction, $\tilde{d}_j^{s-1}\phi_j^{(\delta)}$ is a valid test function for the equation that is satisfied for $\tilde{d}_j^{s-1}u_j = v_j$ at $x_j$. Since we assumed that $x_j \in \tilde{\Omega}_j \cap B_1$, it must hold $L_j(\tilde{d}_j^{s-1} \phi_j^{(\delta)})(x_j) \ge f_j(x_j)$, which contradicts \eqref{eq:stability-contradiction}. \\
Therefore, it must be $x_j \in \partial \tilde{\Omega}_j \cap B_1$, as we claimed before. Thus, by the boundary condition
\begin{align*}
\partial_{\nu_{x_j}} \phi_j^{(\delta)}(x_j) \le g_j(x_j).
\end{align*}
Passing this inequality to the limit, and using the uniform convergence $|g_j -q_j - g| \to 0$, $\nu_{x_j} \to \nu_0 = e_n$, and $\tilde{\Omega}_j \to \{ x_n > 0 \}$, we obtain 
\begin{align*}
\partial_n (\phi^{(\delta)}(x_0) - q(x_0) ) \le g(x_0),
\end{align*}
where $q \in \cP_k$ is the limit of the sequence of polynomials $(q_j)_j$. Thus, we have $\partial_n \phi^{(\delta)}(x_0) \le g(x_0) + q(x_0)$, i.e., \eqref{eq:stability-final}, as desired. This concludes the proof.
\end{proof}

Second, we prove that the difference of two viscosity solutions is again a viscosity subsolution.

\begin{lemma}
\label{lemma:sum-viscosity}
Let $k \in \N \cup \{ 0 \}$, $\Omega \subset \R^n$ be an open, bounded domain with $\partial \Omega \in C^{2,\gamma}$ for some $\gamma > 0$. Let $L \in \mathcal{L}_s^{\text{hom}}(\lambda,\Lambda,k-1+\alpha)$ for some $\alpha > 0$. Let $v,w \in L^1_{2s}(\R^n)$ with $v/d^{s-1}, w/d^{s-1} \in C(\overline{\Omega})$ be viscosity solutions to
\begin{align*}
\left\{\begin{array}{rcl}
L v &\overset{k}{=}& f_1 ~~ \text{ in } \Omega,\\
v &=& 0 ~~~ \text{ in } \R^n \setminus \Omega,\\
v/d^{s-1} &=& g_1 ~~ \text{ on } \partial \Omega,
\end{array}\right.
\qquad
\left\{\begin{array}{rcl}
L w &\overset{k}{=}& f_2 ~~ \text{ in } \Omega,\\
w &=& 0 ~~~ \text{ in } \R^n \setminus \Omega,\\
w/d^{s-1} &=& g_2 ~~ \text{ on } \partial \Omega,
\end{array}\right.
\end{align*}
for some $f_1,f_2 \in C(\Omega)$ and $g_1,g_2 \in C(\partial \Omega)$.
Then, $v-w$ is a viscosity solution to 
\begin{align*}
\left\{\begin{array}{rcl}
L (v-w) &\overset{k}{=}& f_1 - f_2 ~~ \text{ in } \Omega,\\
v-w &=& 0 ~~ \qquad ~ \text{ in } \R^n \setminus \Omega,\\
(v-w)/d^{s-1} &=& g_1 - g_2 ~~ \text{ on } \partial \Omega.
\end{array}\right.
\end{align*}
\end{lemma}

\begin{proof}
We will only demonstrate the proof in case $k=0$. The general case follows immediately by combining the arguments with \autoref{def:L-up-to-a-polynomial}.
For the nonlocal equation, the result follows for instance from \cite[Lemma 3.4.14]{FeRo23}. For the boundary condition, one can proceed as follows. First, we define the $\sup$- and $\inf$-convolutions (see \cite[Lemma 3.2.16]{FeRo23})
\begin{align*}
(v/d^{s-1})_{\eps}(x) &:= \inf_{\overline{D}} \left( \frac{v}{d^{s-1}}(z) + \frac{|x-z|^2}{\eps} \right) ~~ \forall x \in \overline{D}, \qquad  (v/d^{s-1})_{\eps}(x) = \frac{v}{d^{s-1}}(x) ~~ \forall x \in \R^n \setminus D,\\
(w/d^{s-1})^{\eps} &= \sup_{\overline{D}} \left( \frac{w}{d^{s-1}}(z) - \frac{|x-z|^2}{\eps} \right) ~~ \forall x \in \overline{D}, \qquad  (w/d^{s-1})^{\eps}(x) = \frac{w}{d^{s-1}}(x) ~~ \forall x \in \R^n \setminus D,
\end{align*}
 with $D \subset \Omega$ open, bounded such that $\overline{D} \cap \partial \Omega \not= \emptyset$. In analogy to \cite[Proposition 3.2.17]{FeRo23}, we claim that for any $x \in \partial \Omega \cap \overline{D}$ it holds in the viscosity sense
\begin{align}
\label{eq:sum-visc-claim}
\partial_{\nu} (v/d^{s-1})_{\eps}(x) \le g_1(x) + \delta_{\eps}, \qquad \partial_{\nu} (w/d^{s-1})^{\eps} \ge g_2(x) + \delta^{\eps},
\end{align}
where $\delta_{\eps} , \delta^{\eps} \to 0$, as $\eps \to 0$. Note that once \eqref{eq:sum-visc-claim} is proven, since $(v/d^{s-1})_{\eps}$ and $-(w/d^{s-1})^{\eps}$ are both semi-concave, we have that at any point $x \in \partial \Omega \cap \overline{D}$, where $(v/d^{s-1})_{\eps} - (w/d^{s-1})^{\eps}$ can be touched by a paraboloid from below, the functions $(v/d^{s-1})_{\eps}$ and $-(w/d^{s-1})^{\eps}$ must be in $C^{1,1}$. Hence, by the linearity of $\partial_{\nu}$, and due to \eqref{eq:sum-visc-claim} it must hold 
\begin{align*}
\partial_{\nu} ((v/d^{s-1})_{\eps} - (w/d^{s-1})^{\eps})(x) \le g_1(x) - g_2(x) + \delta_{\eps} - \delta^{\eps} \to g_1(x) - g_2(x) ~~ \text{ as } \eps \to 0.
\end{align*}
Thus, by the stability for viscosity solutions (which was provided in a significantly more general framework in \autoref{lemma:stability}), we deduce that $\partial_{\nu}((v-w)/d^{s-1}) \le (g_1-g_2)$ in the viscosity sense. In a similar way, one can prove $\partial_{\nu}((v-w)/d^{s-1}) \ge (g_1-g_2)$, and thus, we obtain the desired result.\\
Thus, it remains to give a proof of \eqref{eq:sum-visc-claim}.
To see it, for any test function $\phi \in C^2(B_r(x_0))$ touching $(v/d^{s-1})_{\eps}$ from below at $x_0 \in \partial \Omega \cap \overline{D}$, we define 
\begin{align*}
\psi^{(\delta)}(x) = -\delta \1_{B_1(x_0)}(x) \left[d(x) - d^{1+\eta}(x) \right]
\end{align*} 
for some $\eta \in (0,\gamma)$ and observe that $\phi^{(\delta)} = \phi + \psi^{(\delta)}$ is still a valid test function, touching $(v/d^{s-1})_{\eps}$ (strictly) from below, at $x_0$. Then, there exists $x_{\eps} \in \overline{D}$ with $x_{\eps} \in B_{c\eps}(x_0)$ for some $c > 0$, depending only on the oscillation of $v/d^{s-1}$, such that $\phi^{(\delta)}(\cdot + x_0 - x_{\eps}) - \eps^{-1}|x_0 - x_{\eps}|^2$ touches $v/d^{s-1}$ from below at $x_{\eps}$. Indeed, from the definition of $(v/d^{s-1})_{\eps}$ we deduce that there exist $x_{\eps} \in \overline{D}$ with $x_{\eps} \to x_0$ such that 
\begin{align*}
\frac{v}{d^{s-1}}(x_0) \ge (v/d^{s-1})_{\eps}(x_0) = \frac{v}{d^{s-1}}(x_{\eps}) + \frac{|x_0 - x_{\eps}|^2}{\eps}.
\end{align*}
Hence, the rate of convergence $x_{\eps} \to x_0$ only depends on the oscillation of $v/d^{s-1}$. Then, since $\phi^{(\delta)}$ is a valid test function, we deduce that for any $x \in D$
\begin{align*}
\phi^{(\delta)}(x + x_0 - x_{\eps}) \le (v/d^{s-1})_{\eps}(x + x_0 - x_{\eps}) \le \frac{v}{d^{s-1}}(x) + \frac{|x_0 - x_{\eps}|^2}{\eps}
\end{align*}
if $\eps > 0$ is so small that $x + x_0 - x_{\eps} \in D$. Since the aforementioned inequality becomes an equality in case $x = x_{\eps}$, we deduce that indeed, $\phi^{(\delta)}(\cdot + x_0 - x_{\eps}) - \eps^{-1}|x_0 - x_{\eps}|^2$ touches $v/d^{s-1}$ from below at $x_{\eps}$, as claimed.

We observe that $x_{\eps} \not \in \Omega$ since otherwise one would get a contradiction with the nonlocal equation satisfied by $v$, in the exact same way as in the proof of \eqref{eq:stability-contradiction}, if $\eps > 0$ is small enough. Thus, $x_{\eps} \in \partial \Omega \cap \overline{D}$, and from the boundary condition satisfied by $v$, it follows $\partial_{\nu} \phi^{(\delta)}(x_0) \le g_1(x_{\eps})$. Thus, by the definition of $\phi^{(\delta)}$, we have $\partial_{\nu} \phi(x_0) = \partial_{\nu} \phi^{(\delta)}(x_0) + \delta \le g_1(x_{\eps}) + \delta$ for any $\delta > 0$. Thus, sending $\delta \to 0$ and recalling that $x_{\eps} \to x_0$, as $\eps \to 0$, this proves the first statement in \eqref{eq:sum-visc-claim} with $\delta_{\eps} = g_1(x_{\eps})-g_1(x_0)$. Analogously, one proves the second claim in \eqref{eq:sum-visc-claim}.
\end{proof}

\section{Nonlocal maximum principles with local Dirichlet and Neumann conditions}
\label{sec:max-principles}

In this section, we establish weak maximum principles for nonlocal equations with local Dirichlet- and Neumann data (see \autoref{lemma:weak-max-princ-large} and \autoref{lemma:Neumann-max-princ-intro}).

First, we establish a weak maximum principle for solutions to the inhomogeneous Dirichlet problem in \eqref{eq:Dirichlet-problem-blowup} (see \autoref{lemma:weak-max-princ-large}). Its proof goes by sliding the barrier subsolution $\phi$ from \autoref{lemma:subsol} underneath $v$ from below.

\begin{proof}[Proof of \autoref{lemma:weak-max-princ-large}]
By assumption on $v$, we have that $v/d^{s-1} \in C(\overline{\Omega})$ with $v/d^{s-1} \ge 0$ on $\partial \Omega$. Let $z \in \partial \Omega$ be such that $\min_{\partial \Omega} v/d^{s-1} = v/d^{s-1}(z) =: l \ge 0$. Let $\eps \in (0,s)$ and $M > 1$ to be chosen later, and recall the subsolution $\phi_l \in C(\Omega)$ from \autoref{lemma:subsol}. We define 
\begin{align*}
c_0 := \inf \{ c \in \R : \phi_l/d^{s-1} - c \le v/d^{s-1} \text{ in } \overline{\Omega} \}.
\end{align*}
Since also $\phi_l/d^{s-1} \in C(\overline{\Omega})$, the above set is nonempty and $c_0 < \infty$. In fact, recalling the definition of $\phi_l$, it must be
\begin{align}
\label{eq:c_0-M}
c_0 \le \left\Vert v/d^{s-1} \right\Vert_{L^{\infty}(\overline{\Omega})} + \left\Vert (\phi_l)_+ /d^{s-1} \right\Vert_{L^{\infty}(\overline{\Omega})} \le \left\Vert v/d^{s-1} \right\Vert_{L^{\infty}(\overline{\Omega})} + l + c|\diam(\Omega)|^{\eps},
\end{align}
which is independent of $M$.
Moreover, since $\phi_l/d^{s-1}(z) = l = v/d^{s-1}(z)$, we have that $c_0 \ge 0$. Then, in particular, we have
\begin{align*}
\phi_l/d^{s-1} - c_0 \le v/d^{s-1} ~~ \text{ in } \R^n, \qquad \text{ and } \qquad \phi_l/d^{s-1}(x_0) - c_0 = v/d^{s-1}(x_0) ~~ \text{ for some } x_0 \in \overline{\Omega}.
\end{align*}
In case $x_0 \in \Omega$, we have
\begin{align*}
\phi_l - c_0 d^{s-1} - v \le 0 ~~ \text{ in } \R^n \qquad \text{ and }  \qquad (\phi_l - c_0 d^{s-1} - v)(x_0) = 0,
\end{align*}
so it must be
\begin{align*}
0 \le L(\phi_l - c_0 d^{s-1} - v)(x_0) \le L\phi_l(x_0) - c_0 L(d^{s-1})(x_0) \le -d^{\eps -s  -1}(x_0) - M + (l+c_0) cd^{\delta \gamma -s-1}(x_0),
\end{align*}
where we used \autoref{lemma:subsol} and that $|L(d^{s-1})| \le c d^{\delta \gamma -s-1}$ for any $\delta \in (0,s)$ by \autoref{lemma:dist-s-1+eps}. Next, we fix any $\delta \in (0,s)$, and take $\eps < \delta \gamma$ and $M$ so large, depending only on $c_0,l,\diam(\Omega)$ (but not on $x_0$), such that
\begin{align*}
-d^{\eps -s  -1}(x_0) - M + (l+c_0) cd^{\delta \gamma -s-1}(x_0) < 0.
\end{align*}
Since $c_0$ is independent of $M$ (see \eqref{eq:c_0-M}), we obtain a contradiction. Thus, it must be $x_0 \in \partial \Omega$, which by construction yields that $c_0 = 0$, and therefore $\phi_l \le v$ in $\Omega$. Since $l \ge 0$, by \autoref{lemma:subsol}, there exists $\delta > 0$ such that $\phi_l \ge 0$ in $\Omega \cap \{ d \le \delta \}$. Therefore, $v$ is a viscosity solution to
\begin{align*}
\begin{cases}
L v &\ge 0 ~~ \text{ in } \Omega \cap \{ d > \delta \},\\
v &\ge 0 ~~ \text{ in } \R^n \setminus (\Omega \cap \{ d > \delta \}).
\end{cases}
\end{align*} 
Since $v \in C(\overline{\Omega \cap \{ d > \delta \}})$, we can apply the maximum principle for viscosity solutions to $v$ (see \cite[Lemma 3.2.19]{FeRo23}) and deduce that $v \ge 0$ in $\R^n$, as desired.
\end{proof}

In particular, we have the following comparison principle:
\begin{lemma}
\label{lemma:comp-princ-large}
Let $L \in \mathcal{L}_s^{\text{hom}}(\lambda,\Lambda)$. Let $\Omega \subset \R^n$ be an open, bounded domain with $\partial \Omega \in C^{1,\gamma}$ for some $\gamma > 0$. Let $v,b \in L^1_{2s}(\R^n)$ with $v/d^{s-1}, b/d^{s-1} \in C(\overline{\Omega})$ be viscosity solutions to
\begin{align*}
\begin{cases}
L v &\ge f ~~ \text{ in } \Omega,\\
v/d^{s-1} &\ge 0 ~~ \text{ on } \partial \Omega,
\end{cases}
\qquad
\begin{cases}
L b &\le f ~~ \text{ in } \Omega,\\
b &\le v ~~ \text{ in } \R^n \setminus \Omega,\\
b/d^{s-1} &\le 0 ~~ \text{ on } \partial \Omega
\end{cases}
\end{align*}
for some $f \in C(\Omega)$. Then, $v \ge b$ in $\R^n$.
\end{lemma}

\begin{proof}
Since by \cite[Lemma 3.4.13]{FeRo23} $w= v-b$ is a viscosity solution to $Lw \ge 0$ in $\Omega$ such that $w/d^{s-1} \ge 0$ on $\partial \Omega$, and $w \ge 0$ in $\R^n \setminus \Omega$, it satisfies the assumptions of \autoref{lemma:weak-max-princ-large}. An application of this result concludes the proof.
\end{proof}

As an application, we have the following version of a nonlocal Hopf lemma for viscosity solutions. The proof follows in the same way as \cite[Proposition 2.6.6]{FeRo23}, where the Hopf lemma was proved for bounded solutions.

\begin{lemma}
\label{lemma:Hopf-viscosity}
Let $L \in \mathcal{L}_s^{\text{hom}}(\lambda,\Lambda)$. Let $\Omega \subset \R^n$ be an open, bounded domain with $\partial \Omega \in C^{1,\gamma}$ for some $\gamma > 0$. Let $v \in L^1_{2s}(\R^n)$ with $v/d^{s-1} \in C(\overline{\Omega})$ satisfy in the viscosity sense:
\begin{align*}
\begin{cases}
L v = f &\ge 0 ~~ \text{ in } \Omega,\\
v &\ge 0 ~~ \text{ in } \R^n \setminus \Omega,\\
v/d^{s-1} &\ge 0 ~~ \text{ on } \partial \Omega
\end{cases}
\end{align*}
for some $f \in C(\Omega)$. Then, either $v \equiv 0$ in $\Omega$, or 
\begin{align*}
v(x) \ge C \left(\inf_{\{x \in \Omega : \dist(x,\partial\Omega) \ge \delta \}} v \right) d^s(x) ~~ \text{ in } \Omega
\end{align*}
for some $C, \delta > 0$, which depend only on $n,s,\lambda,\Lambda,\gamma,\diam(\Omega)$, and the $C^{1,\gamma}$ radius of $\Omega$.
\end{lemma}

\begin{proof}
First, by the weak maximum principle for viscosity solutions with boundary blow-up (see \autoref{lemma:weak-max-princ-large}), we have $v \ge 0$ in $\R^n$. In order to deduce $v > 0$ in case $v \not\equiv 0$, one uses the nonlocal weak Harnack inequality (see \cite[Theorem 3.3.1]{FeRo23}). Then, we use the subsolution $\phi$ from \cite[Corollary B.2.8]{FeRo23} which satisfies
\begin{align*}
\begin{cases}
L \phi &\le - 1 ~~ \text{ in } N_{\delta},\\
\max\{ d^s , \delta^{-1}\} \ge \phi &\ge \delta d^s ~~ \text{ in } \R^n
\end{cases}
\end{align*}
for some $\delta > 0$ and where $N_{\delta} = \{ 0 < d < \delta \}$. Let us define
\begin{align*}
c_{\ast} = \min \{ v(x) : x \in \Omega \setminus N_{\delta} \} > 0.
\end{align*}
Then, we have
\begin{align*}
c_{\ast} \delta L \phi \le L v ~~ \text{ in } N_{\delta} \qquad \text{ and } c_{\ast} \delta \phi \le v ~~ \text{ in } \R^n \setminus N_{\delta}
\end{align*}
Hence, by the comparison principle in \autoref{lemma:comp-princ-large}, we deduce that $c_{\ast} \delta \phi \le u$ in $\R^n$, which implies the desired result.
\end{proof}

Given a $C^{1,\gamma}$ domain $\Omega \subset \R^n$, let us now consider functions $b : \R^n \to \R$, which arise as the solution to the following Dirichlet problem
\begin{align}
\label{eq:b}
\begin{cases}
L b &= f_b ~~~~ \text{ in } \Omega,\\
b_{\Omega} &= e_b ~~~~ \text{ in } \R^n \setminus \Omega,\\
b_{\Omega}/d^{s-1} &= g_b ~~~~ \text{ on } \partial \Omega
\end{cases}
\end{align}
for some $f_b \ge 0$ with $f_b \not= 0$, $e_b \ge 0$, and $g_b \ge 0$.
Note that with the maximum principle (see \autoref{lemma:weak-max-princ-large}) at hand, the existence of $b$ can be established using standard techniques. For well-posedness results in case $L = (-\Delta)^s$, we refer to \cite{Aba15}. Moreover, note that by \autoref{lemma:weak-max-princ-large}, we have $b \ge 0$ in $\Omega$, and by the same argument as in the proof of \eqref{eq:application-Dirichlet-barrier}, we have $b/d^{s-1} \in L^{\infty}(\Omega)$. Moreover, if $\partial \Omega \in C^{2,\gamma}$ and $f_b,e_b,g_b$ are smooth, then by \autoref{thm:dirichlet}, we have $b_{\Omega}/d^{s-1} \in C^{1,\gamma}(\overline{\Omega})$, and $\partial_{\nu}(b/d^{s-1})$ exists in the classical sense.\\
In the following, we will denote by $b_{\Omega}$ the solution to \eqref{eq:b} with $f_b=g_b=1$ and $e_b=0$.

As a corollary of the previous results, we obtain the following pointwise formulation of a nonlocal Hopf lemma for solutions with boundary blow-up. 

\begin{lemma}
\label{lemma:Hopf-large}
Let $L \in \mathcal{L}_s^{\text{hom}}(\lambda,\Lambda)$. Let $\Omega \subset \R^n$ be an open, bounded domain with $\partial \Omega \in C^{1,\gamma}$ for some $\gamma > 0$. Let $v \in L^1_{2s}(\R^n)$ with $v/d^{s-1} \in C(\overline{\Omega})$ satisfy in the viscosity sense:
\begin{align*}
\begin{cases}
L v &\ge 0 ~~ \text{ in } \Omega,\\
v &\ge 0 ~~ \text{ in } \R^n \setminus \Omega,\\
v/d^{s-1} &= g ~~ \text{ on } \partial \Omega
\end{cases}
\end{align*}
for some $g \in C(\partial\Omega)$. Let $x_0 \in \partial \Omega$ be such that $\min_{ \overline{\partial \Omega} } g = g(x_0) \le 0$.  Then, either $v \equiv 0$, or we have that in the viscosity sense
\begin{align*} 
\partial_{\nu} (v/b)(x_0) > 0
\end{align*}
for any $b$ as in \eqref{eq:b} with $b/d^{s-1} = 1$ on $\partial \Omega \cap (\{g < 0 \} \cup \{ x_0 \})$.
\end{lemma}

Note that in particular, \autoref{lemma:Hopf-large} implies that for the regularized distance $d$,
\begin{align*}
\partial_{\nu} (v/d^{s-1})(x_0) = \partial_{\nu} (v/b)(x_0) + g(x_0) \partial_{\nu} (b / d^{s-1})(x_0) > g(x_0) \partial_{\nu} (b / d^{s-1})(x_0).
\end{align*}
We stress that the sign of the right-hand side depends on the choice of the regularized distance $d$.

\begin{proof}
Note that since $g(x_0) \le 0$ we have by the construction of $b$ in \eqref{eq:b}
\begin{align*}
L(v - g(x_0)b ) \ge - g(x_0) &\ge 0 ~~ \text{ in } \Omega,\\
v - g(x_0) b &\ge 0 ~~ \text{ in } \R^n \setminus \Omega,\\
(v - g(x_0) b)/d^{s-1} = g - g(x_0) &\ge 0 ~~ \text{ on } \partial \Omega \cap \{ g < 0 \},\\
(v - g(x_0) b)/d^{s-1} \ge g &\ge 0 ~~ \text{ on } \partial \Omega \cap \{ g \ge 0 \}.
\end{align*}

Thus, an application of \autoref{lemma:Hopf-viscosity} to $v - g(x_0) b$ yields that either $v - g(x_0) b \equiv 0$ in $\Omega$, or
\begin{align}
\label{eq:Hopf-large-help}
v - g(x_0) b \ge c d^s ~~ \text{ near } x_0.
\end{align}
Note that we cannot have $v - g(x_0) b \equiv 0$, unless $g(x_0) = 0$ (in which case $v \equiv v - g(x_0) b \equiv 0$), since then
\begin{align*}
L v = g(x_0) L b \le g(x_0) < 0 ~~ \text{ in } \Omega,
\end{align*}
a contradiction.
Thus, unless $v \equiv 0$, we have \eqref{eq:Hopf-large-help}, and we compute, using that $b \ge 0$ and $(b/d^{s-1})(x_0) = 1$,
\begin{align*}
\partial_{\nu}(v/b)(x_0) = \lim_{x \to x_0} \frac{\frac{v(x)}{b(x)} - g(x_0)}{d(x)} = \lim_{x \to x_0} \frac{v(x) - g(x_0) b(x)}{b(x)d(x)} \ge c \lim_{x \to x_0} \frac{d^{s-1}(x)}{b(x)} = c > 0.
\end{align*}
Note that if the limit in the previous estimate does not exist, we need to interpret the boundary condition in the viscosity sense, i.e., take any smooth $\psi$ with $\psi(x_0) = (v/d^{s-1})(x_0) = g(x_0)$ and $\psi \ge v/d^{s-1}$. Then, the limit $\partial_{\nu}\psi (x_0)$ exists, and an analogous computation as above yields $\partial_{\nu}\psi(x_0) \ge c > 0$, i.e. $\partial_{\nu}(v/b)(x_0) > 0$ in the viscosity sense.
\end{proof}

Finally, we are in a position to prove the main result of this section, a maximum principle for nonlocal equations with local Neumann conditions.

\begin{lemma}
\label{lemma:Neumann-max-princ}
Let $L \in \mathcal{L}_s^{\text{hom}}(\lambda,\Lambda)$. Let $\Omega \subset \R^n$ be an open, bounded domain with $\partial \Omega \in C^{2,\gamma}$ for some $\gamma > 0$ and $K \in C^{5+2\gamma}(\mathbb{S}^{n-1})$. Let $\Gamma \subset \partial \Omega$, $v \in L^1_{2s}(\R^n)$ with $v/d^{s-1} \in C(\overline{\Omega})$ satisfy in the viscosity sense:
\begin{align*}
\begin{cases}
L v &\ge f  ~~ \text{ in } \Omega,\\
v &\ge 0 ~~ \text{ in } \R^n \setminus \Omega,\\
\partial_{\nu} (v/b) &\le g ~~ \text{ on } \partial \Omega \setminus \Gamma,\\
v/b & \ge 0 ~~ \text{ on } \partial \Omega \cap \Gamma
\end{cases}
\end{align*}
for some $f \in C(\Omega)$ with $d^{s+1-\eps}f \in L^{\infty}( \Omega )$ for some $\eps \in (0,s]$, and $g \in C(\partial \Omega)$. 
Here, $b$ is as in \eqref{eq:b} with $b/d^{s-1} = 1$ on $\partial \Omega \setminus \Gamma'$ for some $\Gamma' \Subset \Gamma$.
Then, there exists $c > 0$, depending only on $n,s,\lambda,\Lambda,\gamma,\eps$, and the $C^{2,\gamma}$ radius of $\Omega$ and $\diam(\Omega)$, such that 
\begin{align*}
v/d^{s-1} \ge - c \Vert d^{s-\eps} f \Vert_{L^{\infty}( \Omega )} - c  \Vert g \Vert_{L^{\infty}( \partial \Omega \setminus \Gamma )} ~~ \text{ in } \Omega.
\end{align*}
\end{lemma}

\begin{proof}
The case $f \ge 0$ and $g \le 0$ follows from the Hopf lemma (see \autoref{lemma:Hopf-large}). 
In fact, since $v/b \in C(\partial\Omega)$, there exists $x_0 \in \partial \Omega$ with $\min_{\partial \Omega} (v/b) = (v/b)(x_0)$. If $(v/b)(x_0) \ge 0$, then we have that $v/d^{s-1} \ge 0$ on $\partial \Omega$. Otherwise, $(v/b)(x_0) < 0$, and then by assumption it must be $x_0 \in \partial \Omega \setminus \Gamma$. However, in this case \autoref{lemma:Hopf-large} implies that either $v \equiv 0$, (in which case we are done), or $\partial_{\nu}(v/b)(x_0) > 0$, which contradicts $g(x_0) \le 0$. Thus, we must have $v/d^{s-1} \ge 0$ on $\partial \Omega$.
However, by the weak maximum principle (see \autoref{lemma:weak-max-princ-large}), this implies $v \ge 0$, as desired.

Now, we explain how to get the result with general $f,g$.
To do so, let $\tilde{\psi}_1$ be the solution to
\begin{align*}
\begin{cases}
L \tilde{\psi}_1 &= 0 ~~ \text{ in } \Omega,\\
\tilde{\psi}_1 &= 0 ~~ \text{ in } \R^n \setminus \Omega,\\
\tilde{\psi}_1/d^{s-1} &= h ~~ \text{ on } \partial \Omega
\end{cases}
\end{align*}
for some smooth function $h$ which satisfies $0 \le h \le 1$, and is such that $h = 1$ on $\partial \Omega \setminus \Gamma$, and $h = 0$ in $\partial \Omega \cap \Gamma'$.

From \autoref{lemma:Hopf-large}, we deduce that $\partial_{\nu}(\tilde{\psi}_1/b) < 0$ on $\partial \Omega \setminus \Gamma$. Since $\partial \Omega \in C^{2,\gamma}$, by \autoref{thm:dirichlet} we have that $\partial_{\nu}(\tilde{\psi}_1/b) \in C^{\gamma}(\partial \Omega)$, and therefore, there is $c_0 > 0$ such that
\begin{align*}
\partial_{\nu}(\tilde{\psi}_1/b) \le -c_0 < 0 ~~ \text{ on } \partial \Omega \setminus \Gamma.
\end{align*}

Moreover, let us denote by $\tilde{\psi}_2$ the function $\tilde{\psi}$ from the second claim of \autoref{lemma:supersol}, which satisfies for some $c_2 > 0$
\begin{align*}
\begin{cases}
L \tilde{\psi}_2 &\ge d^{\eps - s} ~~ \text{ in } \Omega,\\
\tilde{\psi}_2 &= 0 ~~~~~~ \text{ in } \R^n \setminus \Omega,\\
\tilde{\psi}_2/d^{s-1} &= 0 ~~~~~~ \text{ on } \partial \Omega,\\
\partial_{\nu}(\tilde{\psi}_2/b) &\le c_2 ~~~~~ \text{ on } \partial \Omega.
\end{cases}
\end{align*}
Hence, if we take $M = c_0^{-1}(c_2+1) > 0$ and denote $\tilde{\psi} := M\tilde{\psi}_1 + \tilde{\psi_2}$, we obtain
\begin{align*}
\begin{cases}
L \tilde{\psi} &\ge d^{\eps-s} ~~ \text{ in } \Omega,\\
\tilde{\psi} &= 0 ~~~~~~ \text{ in } \R^n \setminus \Omega,\\
\tilde{\psi}/d^{s-1} &= Mh ~~~~ \text{ on } \partial \Omega,\\
\partial_{\nu}(\tilde{\psi}/b) &\le -1  ~~~~ \text{ on } \partial \Omega \setminus \Gamma.
\end{cases}
\end{align*}

We apply the previous argument with $v$ replaced by 
\begin{align*}
w = v + \big(\Vert d^{s-\eps} f \Vert_{L^{\infty}( \Omega ) } + \Vert g \Vert_{L^{\infty}( \partial \Omega \setminus \Gamma)} \big) \tilde{\psi}, 
\end{align*}

Then, we have that in the viscosity sense:
\begin{align*}
\begin{cases}
L w \ge f + d^{\eps- s} \big(\Vert  d^{s-\eps} f  \Vert_{L^{\infty}(\Omega) } + \Vert g \Vert_{L^{\infty}(\partial \Omega \setminus \Gamma )} \big) &\ge 0 ~~ \text{ in } \Omega,\\
w \ge v &\ge 0 ~~ \text{ in } \R^n \setminus \Omega,\\
\partial_{\nu} (w/b) \le g - \big( \Vert d^{s-\eps} f \Vert_{L^{\infty}(\Omega)}  + \Vert g \Vert_{L^{\infty}(\partial \Omega \setminus \Gamma)} \big) &\le 0 ~~ \text{ on } \partial \Omega \setminus \Gamma,\\
w/b \ge Mh &\ge 0 ~~ \text{ on } \partial \Omega \cap \Gamma.
\end{cases}
\end{align*}

Altogether, by the same argument as at the beginning of the proof, we have $w \ge 0$ in $\Omega$.
Let us now observe that by construction and the same argument as in the proof of \eqref{eq:application-Dirichlet-barrier} we have
\begin{align*}
\tilde{\psi} \le C d^{s-1} ~~ \text{ in } \Omega
\end{align*}
for some $C > 0$. Therefore, we obtain
\begin{align*}
v \ge - \tilde{\psi} \big(\Vert d^{s-\eps} f \Vert_{L^{\infty}( \Omega) } + \Vert g \Vert_{L^{\infty}( \partial \Omega \setminus \Gamma)} \big) \ge - C d^{s-1}\big(\Vert d^{s-\eps} f \Vert_{L^{\infty}(\Omega) } + \Vert g \Vert_{L^{\infty}(\partial \Omega \setminus \Gamma)} \big) ~~ \text{ in } \Omega,
\end{align*}
as desired.
\end{proof}

\begin{proof}[Proof of \autoref{lemma:Neumann-max-princ-intro}]
\autoref{lemma:Neumann-max-princ-intro} is a special case of \autoref{lemma:Neumann-max-princ}.
\end{proof}

\section{H\"older estimates up to the boundary}
\label{sec:Holder-estimate}

The previous maximum principle for nonlocal equations with local Neumann conditions (see \autoref{lemma:Neumann-max-princ}) puts us in a position to establish a Harnack inequality for solutions to \eqref{eq:Neumann-problem-blowup} at the boundary, which will eventually lead to the H\"older regularity estimate in \autoref{thm:bdry-Holder}.\\
To prove it, we adapt some of the ideas in \cite{LiZh23} to the framework of solutions to nonlocal problems which blow up at the boundary.

For $\delta > 0$, let us define $\Omega_{\delta} = \{ x \in \Omega : \dist(x,\partial \Omega) \ge \delta \}$.

\begin{lemma}
\label{lemma:bdry-weak-Harnack}
Let $L \in \mathcal{L}_s^{\text{hom}}(\lambda,\Lambda)$. Let $\Omega \subset \R^n$ be an open, bounded domain with $0 \in \partial \Omega$ and $\partial \Omega \in C^{2,\gamma}$ for some $\gamma > 0$ and $K \in C^{5+2\gamma}(\mathbb{S}^{n-1})$. Let $v \in L^1_{2s}(\R^n)$ with $v/d^{s-1} \in C(\overline{\Omega})$ be a viscosity solution to
\begin{align*}
\begin{cases}
Lv &\ge f ~~ \text{ in } \Omega \cap B_1,\\
v &\ge 0 ~~ \text{ in } \R^n,\\
\partial_{\nu} (v/b_{\Omega}) &\le g ~~ \text{ on } \partial \Omega \cap B_1
\end{cases}
\end{align*}
for some $f \in C(\Omega \cap B_1)$ with $d^{s-\alpha} f \in L^{\infty}(\Omega \cap B_1)$ for some $\alpha \in (0,s]$, and $g \in C(\overline{\partial \Omega \cap B_1})$.
Assume that $0 \in \partial \Omega$. Then,
\begin{align*}
\dashint_{\Omega_{1/2} \cap B_1} (v/b_{\Omega}) \d x \le c \inf_{\Omega \cap B_{\eta^{-1}} } (v/b_{\Omega}) + c \left( \Vert d^{s-\alpha} f_- \Vert_{L^{\infty}(\Omega \cap B_1)} + \Vert g_+ \Vert_{L^{\infty}( \partial \Omega \cap B_1)} \right),
\end{align*}
where $\eta > 1$ and $c > 0$ depend only on $n,s,\lambda,\Lambda,\gamma,\alpha$, and the $C^{2,\gamma}$ radius of $\Omega$. Here, $b_{\Omega}$ is defined as in \eqref{eq:b}.
\end{lemma}

\begin{proof}
The interior weak Harnack inequality for viscosity supersolutions (see \cite[Theorem 3.3.1]{FeRo23}) applied with $v$ implies
\begin{align*}
\dashint_{\Omega_{1/2} \cap B_1} v(x) \d x \le c \inf_{x \in \Omega_{1/2} \cap B_1 } v(x) + c \Vert f \Vert_{L^{\infty}(\Omega_{1/2} \cap B_1)},
\end{align*}
where $c > 0$ depends on $n,s,\lambda,\Lambda,\eta, \alpha$. Moreover, since $b \asymp c > 0$ in $\Omega_{1/2} \cap B_1$, it follows for $u := v/b$ by \autoref{lemma:Hopf-viscosity}:
\begin{align*}
\dashint_{\Omega_{1/2} \cap B_1} u(x) \d x \le c \inf_{x \in \Omega_{1/2} \cap B_1} u(x) + c \Vert d^{s-\alpha} f \Vert_{L^{\infty}(\Omega_{1/2} \cap B_1)}.
\end{align*}
Thus, it remains to show
\begin{align}
\label{eq:large-sol-maximum-principle-appl}
\inf_{x \in \Omega_{1/2} \cap B_1} u(x) \le c \inf_{x \in \Omega \cap B_{\eta^{-1}}} u(x) + c \left(\Vert d^{s-\alpha} f \Vert_{L^{\infty}(\Omega \cap B_1)} + \Vert g \Vert_{L^{\infty}( \partial \Omega \cap B_1)} \right).
\end{align}
Note that since $v \ge 0$, by the weak Harnack inequality, either $v \equiv 0$ in $\Omega_{1/2} \cap B_1$, or $\inf_{\Omega_{1/2} \cap B_1} v > 0$. Therefore, without loss of generality, we can assume that $\inf_{\Omega_{1/2} \cap B_1} v = 1$.

To prove \eqref{eq:large-sol-maximum-principle-appl}, let us take a set $D \subset \R^n$ with $\partial D \in C^{2,\gamma}$ such that
\begin{align*}
\Omega \cap B_{1/2} \subset D \subset \Omega \cap B_1,
\end{align*}
Let $w$ be a function such that
\begin{align*}
\begin{cases}
L w &= 0 ~~ \text{ in } D,\\
w &\le 1 ~~ \text{ in } (\Omega_{1/2} \cap B_1) \setminus D,\\
w &= 0 ~~ \text{ in } \R^n \setminus (D \cup (\Omega_{1/2} \cap B_1)),\\
\partial_{\nu}(w/b_{\Omega}) &\ge 0 ~~ \text{ on } \partial D \cap B_{2\eta^{-1}},\\
w/b_{\Omega} &\le 0 ~~ \text{ on } \partial D \setminus B_{2 \eta^{-1}},\\
w/b_{\Omega} &\ge c_1 ~~ \text{ in } \Omega \cap B_{\eta^{-1}},
\end{cases}
\end{align*}

We construct $w$ as follows. Let $h : \partial D \to \R$ and $e : \R^n \setminus D \to \R$ be smooth functions such that for some $\eta < 1/8$
\begin{align*}
h =
\begin{cases}
0 ~~~ \text{ on } \partial D \setminus B_{2 \eta^{-1}},\\
c_1  ~~ \text{ on } \partial D \cap B_{\eta^{-1}},
\end{cases}
\qquad e = 
\begin{cases}
1 ~~ \text{ in } T,\\
0 ~~ \text{ in } \R^n \setminus (D \cup (\Omega_{1/2} \cap B_1)),
\end{cases}
\end{align*}
where $T \Subset (\Omega_{1/2} \cap B_1) \setminus D$, $0 \le h \le c_1$, and $0 \le e \le 1$. We let $w$ be the solution to
\begin{align*}
\begin{cases}
L w &= 0 ~~ \text{ in } D,\\
w &= e ~~ \text{ in } \R^n \setminus D,\\
w/b_D &= h ~~ \text{ on } \partial D.
\end{cases}
\end{align*}
Then, we can show that $\partial_{\nu}(w/b_{\Omega}) \ge C > 0$ in $\partial D \cap B_{2 \eta^{-1}}$ (for any given $C > 0$) by making $c_1 > 0$ small enough. Indeed, if $w_1$ solves the Dirichlet problem with boundary data zero and exterior data $e$, then by the Hopf lemma (see \autoref{lemma:Hopf-viscosity}), we have since $w_1/d_D^s \in C^{1,\gamma}(\overline{D})$ by the boundary regularity results in \cite{AbRo20}, and since $\partial D \cap B_{2 \eta^{-1}} \Subset \partial \Omega$
\begin{align*}
\partial_{\nu}(w_1/b_{\Omega}) &= \partial_{\nu}(w_1/d_D^{s-1}) (d^{s-1}_D/b_{\Omega}) + (w_1/d_D^{s-1}) \partial_{\nu}(d_D^{s-1}/b_{\Omega}) \\
&= \partial_{\nu}(w_1/d_D^{s})d_D + (w_1/d_D^s) \partial_{\nu}(d_D) = w_1/d_D^s \ge c_0 > 0 ~~ \text{ on } \partial D \cap B_{2 \eta^{-1}}.
\end{align*}
Moreover, if $w_2$ solves the Dirichlet problem with boundary data $h$ and exterior data zero, we get from \autoref{thm:dirichlet} that $|\partial_{\nu}(w_2/b_{\Omega})| \le c_3 c_1$ in $B_{2\eta^{-1}}$ for some $c_3 > 0$. Hence, choosing $c_1 > 0$ small enough, we deduce the claim for $w = (C/c_0)w_1 + w_2$.

Thus, we have by construction, and using that $\inf_{\Omega_{1/2} \cap B_1} v = 1$, and $w \asymp d_{D}^s$ near $\partial D \setminus \partial \Omega$,
\begin{align*}
\begin{cases}
L(v-w) &\ge f ~~ \text{ in } D,\\
v-w &\ge 0 ~~ \text{ in } \R^n \setminus D,\\
\partial_{\nu}((v-w)/b_{\Omega}) &\le g ~~ \text{ on } \partial D \cap B_{2 \eta^{-1}},\\
(v-w)/b_{\Omega} &\ge 0 ~~ \text{ on } \partial D \setminus B_{2 \eta^{-1}},
\end{cases}
\end{align*}
Note that $b_{\Omega}$ satisfies
\begin{align*}
\begin{cases}
L b_{\Omega} &\ge 0 ~~ \text{ in } D,\\
L b_{\Omega} &\not= 0 ~~ \text{ in } D,\\
b_{\Omega} &\ge 0 ~~ \text{ in } \R^n \setminus D,\\
b_{\Omega}/d^{s-1}_D &= 1 ~~ \text{ on } \partial D \cap B_{4 \eta^{-1}},\\
b_{\Omega}/d^{s-1}_D &\ge 0 ~~ \text{ on } \partial D.
\end{cases}
\end{align*}
Since $(\partial D \cap B_{4 \eta^{-1}}) \Supset (\partial D \cap B_{2 \eta^{-1}})$, we can apply the maximum principle for the Neumann problem \autoref{lemma:Neumann-max-princ} with $\Gamma = \partial D \setminus B_{2 \eta^{-1}}$ and $b = b_{\Omega}$, and deduce
\begin{align*}
(v-w)/b_{\Omega} \ge - c \Vert d^{s-\alpha} f_- \Vert_{L^{\infty}(D \cap B_1)} - c \Vert g_+ \Vert_{L^{\infty}( \partial D \cap B_1 )} ~~ \text{ in } D \cap B_1.
\end{align*}
Since, by construction, we also have 
\begin{align*}
w/b_{\Omega} \ge c_1 = c_1 \inf_{\Omega_{1/2} \cap B_1} v \ge c_2 \inf_{\Omega_{1/2} \cap B_1} u ~~ \text{ in } \Omega \cap B_{\eta^{-1}},
\end{align*}
for some $c_2 > 0$, since $b_{\Omega} \asymp c > 0$ in $\Omega_{1/2} \cap B_1$, we deduce
\begin{align*}
v/b_{\Omega} &= (w + v-w)/b_{\Omega} \\
&\ge c_2 \inf_{\Omega_{1/2} \cap B_1} u - c \Vert d^{s-\alpha} f_- \Vert_{L^{\infty}(D \cap B_1)} - c \Vert g_+ \Vert_{L^{\infty}( \partial D \cap B_{1} )} \\
 &\ge c_2 \inf_{\Omega_{1/2} \cap B_1} u - c \Vert d^{s-\alpha} f_- \Vert_{L^{\infty}(\Omega \cap B_1)} - c \Vert g_+ \Vert_{L^{\infty}( \partial \Omega \cap B_{1} )} ~~ \text{ in } \Omega \cap B_{\eta^{-1}},
\end{align*}
where we used $D \cap B_1 \subset \Omega \cap B_1$. Hence, we obtain \eqref{eq:large-sol-maximum-principle-appl}, as desired.
\end{proof}

As a corollary of the previous weak Harnack inequality at the boundary, we obtain a growth lemma.

\begin{lemma}
\label{lemma:growth-lemma}
Let $L \in \mathcal{L}_s^{\text{hom}}(\lambda,\Lambda)$. Let $\Omega \subset \R^n$ be an open, bounded domain with $\partial \Omega \in C^{2,\gamma}$ for some $\gamma > 0$ and $K \in C^{5+2\gamma}(\mathbb{S}^{n-1})$. Let $\eta > 1$ be as in \autoref{lemma:bdry-weak-Harnack}. Assume that $x_0 \in \partial \Omega$ and let $0 < R \le 1$.
Let $v \in L^1_{2s}(\R^n)$ with $v/d^{s-1} \in C(\overline{\Omega})$ be a viscosity solution to
\begin{align*}
\left\{\begin{array}{rcl}
Lv &\ge& f ~~~~  \qquad\qquad\qquad \quad  \text{ in } \Omega \cap B_R(x_0),\\
\partial_{\nu} (v/b_{\Omega}) &\le& g ~~~~ \qquad\qquad\qquad \quad  \text{ on } \partial \Omega \cap B_R(x_0),\\
v &\ge& 0 ~~~~ \qquad\qquad\qquad \quad \text{ in } B_{R}(x_0),\\
v &\ge& b_{\Omega}(1 - \eta^{j\beta}) ~~~~~~~~ \qquad  \text{ in } B_{\eta^j R}(x_0) \cap \Omega ~~ \forall j \ge 1,\\
v &\ge& (1 - \eta^{j\beta}) ~~~ \qquad \qquad \text{ in } B_{\eta^j R}(x_0) \setminus \Omega ~~ \forall j \ge 1,\\
|\Omega_{R/4} \cap B_{R/2}(x_0) \cap \{ v/b_{\Omega} \ge \frac{1}{4} \}| &\ge& \frac{1}{2} |\Omega_{R/4} \cap B_{R/2}(x_0)|
\end{array}\right.
\end{align*}
for some $f \in C(\Omega \cap B_R(x_0))$ with $d^{s-\alpha} f \in L^{\infty}(\Omega \cap B_R(x_0))$ for some $\alpha \in (0,s]$, and $g \in C(\overline{\partial \Omega \cap B_R(x_0)})$. Then, there exist $\delta > 0$, and $\beta \in (0,1)$, depending only on $n,s,\lambda,\Lambda,\gamma,\alpha$, and the $C^{2,\gamma}$ radius of $\Omega$, such that
\begin{align*}
\inf_{\Omega \cap B_{\eta^{-1} R}(x_0)} (v/b_{\Omega}) + R^{1+\alpha} \Vert d^{s-\alpha} f_- \Vert_{L^{\infty}(\Omega \cap B_{R}(x_0))} + R \Vert g_+ \Vert_{L^{\infty}( \partial \Omega \cap B_R(x_0) )} \ge \delta.
\end{align*}
\end{lemma}

\begin{proof}
Let us assume without loss of generality that $x_0 = 0$.
The proof follows from an application of the weak Harnack inequality (see \autoref{lemma:bdry-weak-Harnack}) to $v_+$. It is slightly involved due to the appearance of the tail term.\\
Indeed, we have
\begin{align*}
L v_+(x) \ge f(x) - \int_{\R^n \setminus B_{R}} v_-(y)K(x-y) \d y =: \tilde{f}(x),
\end{align*}
where we used that by assumption, $v \ge 0$ in $B_{R}$. Then, we obtain from \autoref{lemma:bdry-weak-Harnack} (after scaling), using the last assumption and setting $u := v/b_{\Omega}$
\begin{align}
\label{eq:weak-Harnack-application}
\inf_{\Omega \cap B_{\eta^{-1} R}} u + R^{1+\alpha} \Vert d^{s-\alpha} \tilde{f}_- \Vert_{L^{\infty}(\Omega \cap B_{R/2})} + R \Vert g_+ \Vert_{L^{\infty}( \partial \Omega \cap B_{R/2} )}  \ge c_0 \dashint_{\Omega_{R/4} \cap B_{R/2}} u \d x \ge \frac{c_0}{8},
\end{align}
where $c_0 > 0$ is the constant from the weak Harnack inequality.\\
Next, we estimate $\Vert d^{s-\alpha} \tilde{f} \Vert_{L^{\infty}(\Omega \cap B_R)}$. To do so, we apply a similar reasoning as in the proof of \autoref{lemma:dist-int-est}. First, we recall that for any $x \in B_{R/2}$, there exists $\kappa > 0$ such that for any $t \in (0,\kappa)$:
\begin{align*}
\mathcal{H}^{n-1}(\{d = t\} \cap B_{\eta^{j}R} \setminus B_{\eta^{j-1} R}) \le C (\eta^{j} R)^{n-1},
\end{align*}
where $C > 0$ depends only on $n$ and the $C^{2,\gamma}$ radius of $\Omega$ (we refer to \cite[Lemma B.2.4]{FeRo23} for a reference of this fact). Next, we observe that by the co-area formula, and since $0 \le b_{\Omega} \le C d^{s-1}$:
\begin{align*}
R^{1+s} &\int_{\Omega \setminus B_{R}} v_-(y)K(x-y) \d y \le c\sum_{j \ge 1} (1 - \eta^{j\beta}) R^{1+s} \int_{\Omega \cap (B_{\eta^{j}R} \setminus B_{\eta^{j-1} R})} d^{s-1}(y) |y|^{-n-2s} \d y \\
&\le c\sum_{j \ge 1} (1 - \eta^{j\beta}) R^{1+s} \left( (\eta^{j}R)^{-n-2s} \int_{(B_{\eta^{j} R} \setminus B_{\eta^{j-1} R}) \cap \{d \le \kappa \} } d^{s-1}(y) |\nabla d(y)| \d y + \kappa^{s-1} (\eta^j R)^{-2s} \right) \\
&\le c\sum_{j \ge 1} (1 - \eta^{j\beta}) R^{1+s} \left( (\eta^{j}R)^{-n-2s} \int_0^{ \min\{ \eta^j R , \kappa \} } \hspace{-0.2cm} t^{s-1} \left( \int_{ (B_{\eta^{j}R} \setminus B_{\eta^{j-1} R}) \cap \{ d = t \} } \hspace{-0.6cm} \d \mathcal{H}^{n-1}(y) \right) \d t + (\eta^j R)^{-2s} \right) \\
&\le c\sum_{j \ge 1} (1 - \eta^{j\beta}) R^{1+s} \left( (\eta^j R)^{-1-s} + (\eta^j R)^{-2s} \right) \le c\sum_{j \ge 1} (1 - \eta^{j\beta}) \eta^{-2sj}
\end{align*}
for some $c > 0$, depending only on $n,s,\lambda,\Lambda,\kappa,C,\eta$, where we also used that $R \le 1$. Similarly,
\begin{align*}
R^{1+s} &\int_{(\R^n \setminus \Omega) \setminus B_{R}} v_-(y)K(x-y) \d y \le c\sum_{j \ge 1} (1 - \eta^{j\beta}) R^{1+s} \int_{\R^n \cap (B_{\eta^{j}R} \setminus B_{\eta^{j-1} R})} |y|^{-n-2s} \d y \\
&\le  c\sum_{j \ge 1} (1 - \eta^{j\beta}) R^{1+s} (\eta^j R)^{-2s} \le c \sum_{j \ge 1} (1 - \eta^{j\beta}) \eta^{-2js},
\end{align*}
where we used that $R^{1-s} \le 1$, and $c > 0$ depends only on $n,s,\Lambda$. Therefore, we obtain
\begin{align*} 
R^{1+s} \int_{\R^n \setminus B_{R}} d^{s-1}(y) v_-(y)K(x-y) \d y \le c\sum_{j \ge 1} (1 - \eta^{j\beta}) \eta^{-2js}.
\end{align*}
Since this quantity vanishes as $\beta > 0$ goes to zero, we can make the whole expression smaller than $c_0/16$, which implies by recalling the definition of $\tilde{f}$
\begin{align*}
R^{1+\alpha} \Vert d^{s-\alpha} \tilde{f}_- \Vert_{L^{\infty}(\Omega \cap B_{R/2})}  &\le R^{1+\alpha} \Vert d^{s-\alpha} f_- \Vert_{L^{\infty}(\Omega \cap B_{R/2})} + R^{1+s} \left\Vert\int_{\Omega \setminus B_{R}} v_-(y)K(\cdot -y) \d y \right\Vert_{L^{\infty}(B_{R/2})} \\
&\le R^{1+\alpha} \Vert d^{s-\alpha} f_- \Vert_{L^{\infty}(\Omega \cap B_R)} + \frac{c_0}{16},
\end{align*}
and therefore by the estimate \eqref{eq:weak-Harnack-application}
\begin{align*}
\inf_{\Omega \cap B_{\eta^{-1} R}} u + R^{1+\alpha} \Vert d^{s-\alpha} f_- \Vert_{L^{\infty}(\Omega \cap B_R)} + R \Vert g_+ \Vert_{L^{\infty}(\partial \Omega \cap B_R )} \ge \frac{c_0}{8} - \frac{c_0}{16} = \frac{c_0}{16},
\end{align*}
as desired.
\end{proof}

We are now in a position to prove the boundary H\"older regularity.

\begin{lemma}
\label{lemma:bdry-osc-decay}
Let $L \in \mathcal{L}_s^{\text{hom}}(\lambda,\Lambda)$. Let $\Omega \subset \R^n$ be an open, bounded domain with $\partial \Omega \in C^{2,\gamma}$ for some $\gamma > 0$ and $K \in C^{5+2\gamma}(\mathbb{S}^{n-1})$. Assume that $x_0 \in \partial \Omega$ and let $0 < R \le 1$. Let $v \in L^1_{2s}(\R^n)$ with $v/d^{s-1} \in C(\overline{\Omega})$ be a viscosity solution to
\begin{align*}
\begin{cases}
L v &= f ~~ \text{ in } \Omega \cap B_R(x_0),\\
v &= 0 ~~ \text{ in } B_R(x_0) \setminus \Omega, \\
\partial_{\nu}(v/b_{\Omega}) &= g ~~ \text{ on } \partial \Omega \cap B_R(x_0)
\end{cases}
\end{align*}
for some $f \in C(\Omega \cap B_R(x_0))$ and $g \in C(\overline{\partial \Omega \cap B_R(x_0)})$. Then, there exist $c > 0$, and $\alpha_0 \in (0,1)$, depending only on $n,s,\lambda,\Lambda,\gamma$, and the $C^{2,\gamma}$ radius of $\Omega$, such that if $d^{s-\alpha} f \in L^{\infty}(\Omega \cap B_R(x_0))$ for some $\alpha \in (0,\alpha_0]$, then it holds:
\begin{align*}
[ v/d^{s-1} ]_{C^{\alpha}(\Omega \cap B_{R/2}(x_0))} \le c R^{-\alpha} \big(\Vert v/ d^{s-1} \Vert_{L^{\infty}(\Omega)} &+ \Vert v \Vert_{L^{\infty}(\R^n \setminus \Omega)} \\
&+ R^{1+\alpha} \Vert d^{s-\alpha} f \Vert_{L^{\infty}(\Omega \cap B_R(x_0) )} + R \Vert g \Vert_{L^{\infty}( \partial \Omega \cap B_R(x_0) )} \big).
\end{align*}
\end{lemma}

\begin{proof}
Let us assume without loss of generality that $x_0 = 0$. We will prove the desired result in two steps. Let us denote by $\eta > 1$ the constant from \autoref{lemma:growth-lemma}.

Step 1: We claim that for any $k \in \N$:
\begin{align*}
\osc_{B_{\eta^{-k} R}} (v/b_{\Omega}) &\le c \eta^{-\alpha k} \Big(\Vert v/d^{s-1} \Vert_{L^{\infty}(\Omega)} + \Vert v \Vert_{L^{\infty}(\R^n \setminus \Omega)} \\
&\qquad\qquad\qquad\qquad\qquad\quad + R^{1+\alpha} \Vert d^{s-\alpha} f \Vert_{L^{\infty}(\Omega \cap B_R)} +  R^{\alpha} + R \Vert g \Vert_{L^{\infty}(\partial\Omega \cap B_R)} \Big).
\end{align*}
for some constant $c  > 0$, depending only on $n,s,\lambda,\Lambda,\gamma$, and the $C^{2,\gamma}$ radius of $\Omega$. To prove it, we set $\alpha_0 := \min \{ \beta , \gamma s, 1-s [- \log_{\eta}(1 - \frac{\delta'}{2})] \}$, and $\delta := 1 - \eta^{-\alpha_0}$, where $\delta',\beta, \eta$ are the constants from \autoref{lemma:growth-lemma}. This yields
\begin{align}
\label{eq:alpha-delta}
(1-\delta) = \eta^{-\alpha_0}, \qquad \alpha_0 \le \min \{ \beta , \gamma s , 1-s \}, \qquad \delta \le \delta'/2.
\end{align}
Let us set $u = v/b_{\Omega}$, take $\alpha \in (0,\alpha_0]$, and
\begin{align*}
M := 4 \delta^{-1} c_1 \left(\Vert v/d^{s-1} \Vert_{L^{\infty}(\Omega)} + \Vert v \Vert_{L^{\infty}(\R^n \setminus \Omega)} + R^{1+\alpha} \Vert d^{s-\alpha} f \Vert_{L^{\infty}(\Omega \cap B_R(x_0))} +  R^{\alpha}  + R \Vert g \Vert_{L^{\infty}(\partial\Omega \cap B_R(x_0))} \right), 
\end{align*}
where $c_1 > 0$ denotes the constant $c_1$ from \autoref{lemma:dist-s-1+eps}.\\
The claim of Step 1 will follow immediately, once we construct an increasing sequence $(m_k)_k$ and a decreasing sequence $(M_k)_k$ such that for any $k \in \N$:
\begin{align}
\label{eq:osc-decay-1}
m_k &\le u \le M_k  ~~ \text{ in } B_{\eta^{-k} R},\\
\label{eq:osc-decay-2}
M_k - m_k &= M \eta^{-\alpha k}.
\end{align}
We prove \eqref{eq:osc-decay-1} and \eqref{eq:osc-decay-2} by induction. Setting $m_0 = - \frac{\delta}{2c_1} M$, $M_0 = \frac{\delta}{2c_1} M$, we obtain the desired results for $k = 0$. Let us now assume that \eqref{eq:osc-decay-1} and \eqref{eq:osc-decay-2} hold true for any $j \le k-1$. \\
We will now prove it for $k$. 
Clearly, one of the following two options always holds true:
\begin{align*}
\left| \Omega_{\eta^{-(k-1)}R/4} \cap B_{\eta^{-(k-1)}R/2} \cap \left\{ u \ge \frac{M_{k-1} + m_{k-1}}{2} \right\}  \right| &\ge \frac{|\Omega_{\eta^{-(k-1)}R/4} \cap B_{\eta^{-(k-1)}R/2} |}{2}, \\
\left| \Omega_{\eta^{-(k-1)}R/4} \cap B_{\eta^{-(k-1)}R/2} \cap \left\{ u \ge \frac{M_{k-1} + m_{k-1}}{2} \right\}  \right| &\le \frac{|\Omega_{\eta^{-(k-1)}R/4} \cap B_{\eta^{-(k-1)}R/2} |}{2} .
\end{align*}
In the first case, and in the second case, we define 
\begin{align*}
w = \frac{v - (b_{\Omega} + \1_{\R^n \setminus \Omega} \1_{\{ m_{k-1} < 0\}}) m_{k-1}}{M_{k-1} - m_{k-1}}, \quad w = \frac{(b_{\Omega} + \1_{\R^n \setminus \Omega} \1_{\{ M_{k-1} > 0 \}}) M_{k-1} - v}{M_{k-1} - m_{k-1}}, ~~ \text{ respectively}.
\end{align*}
Let us assume that we are in the first case. The proof of the second case goes via the same arguments, and we will skip it. Let us verify that $w$ satisfies the assumptions of \autoref{lemma:growth-lemma}. 
First, note that if $u(x) \ge \frac{M_{k-1} + m_{k-1}}{2}$ for some $x \in \Omega$, it follows that
\begin{align*}
\frac{w}{b_{\Omega}}(x) = \frac{u(x) - m_{k-1}}{M_{k-1} - m_{k-1}} \ge \frac{u(x) - m_{k-1}}{M_{k-1} - m_{k-1}} \ge \frac{ \frac{M_{k-1} + m_{k-1}}{2} - m_{k-1} }{M_{k-1} - m_{k-1}} = \frac{1}{2},
\end{align*}
Thus, as an immediate consequence of being in the first case, we get
\begin{align*}
\left| \Omega_{\eta^{-(k-1)}R/4} \cap B_{\eta^{-(k-1)}R/2} \cap \left\{ \frac{w}{b_{\Omega}} \ge \frac{1}{2} \right\}  \right| &\ge \frac{|\Omega_{\eta^{-(k-1)}R/4} \cap B_{\eta^{-(k-1)}R/2} |}{2}.
\end{align*}
Moreover, by \eqref{eq:osc-decay-1} (for $k-1$), we have
\begin{align*}
w &= \frac{v - b_{\Omega} m_{k-1}}{M_{k-1} - m_{k-1}} \ge 0 ~~ \text{ in } B_{\eta^{-(k-1)} R} \cap \Omega.
\end{align*}
Non-negativity of $w$ in $B_{\eta^{-(k-1)} R} \setminus \Omega$ follows by assumption and construction.
Note that we obtain
\begin{align*}
|L (b_{\Omega} + \1_{\R^n \setminus \Omega})| \le c_1 ~~ \text{ in } \Omega,
\end{align*}
and therefore $d^{s - \alpha} L (b_{\Omega} + \1_{\R^n \setminus \Omega}) \in L^{\infty}(\Omega \cap B_R)$. Then, by \eqref{eq:osc-decay-2} (for $k-1$) we have
\begin{align}
\label{eq:bdry-osc-PDE}
L w = \frac{f - L (b_{\Omega} + \1_{\R^n \setminus \Omega} \1_{\{ m_{k-1} < 0 \}} )m_{k-1}}{M_{k-1} - m_{k-1}} \ge \frac{f - c_1 m_{k-1}}{M_{k-1} - m_{k-1}} ~~ \text{ in } \Omega \cap B_{R}.
\end{align}
Moreover, clearly
\begin{align*}
\partial_{\nu}(w/b_{\Omega}) = \frac{g - \partial_{\nu}(b_{\Omega}/b_{\Omega})m_{k-1}}{M_{k-1} - m_{k-1}} = \frac{g}{M_{k-1} - m_{k-1}} ~~ \text{ on } \partial \Omega \cap B_R.
\end{align*}

It remains to verify the fourth and fifth assumption of \autoref{lemma:growth-lemma}.
Let us first consider $j \le k-1$. In that case, for any $x \in B_{\eta^{-(k-1)+j} R} \cap \Omega$ it holds by \eqref{eq:osc-decay-1} and \eqref{eq:osc-decay-2}:
\begin{align*}
\frac{w}{b_{\Omega}}(x) &= \frac{u(x) - m_{k-1}}{M_{k-1} - m_{k-1}} \ge \frac{m_{k-j-1} - m_{k-1}}{M_{k-1} - m_{k-1}} \\
&\ge \frac{M_{k-1} - M_{k-j-1} + m_{k-j-1} - m_{k-1}}{M_{k-1} - m_{k-1}} = 1 - \frac{M_{k-j-1} - m_{k-j-1}}{M_{k-1} - m_{k-1}} = 1 - \eta^{\alpha j}.
\end{align*}
Clearly, for any $x \in B_{\eta^{-(k-1)+j} R} \setminus \Omega$ and in case $m_{k-1} < 0$, by the same arguments as above, using \eqref{eq:osc-decay-1}, we have
\begin{align*}
w(x) = \frac{v(x) - m_{k-1}}{M_{k-1} - m_{k-1}} \ge \frac{m_{k-j-1} - m_{k-1}}{M_{k-1} - m_{k-1}} \ge 1 - \eta^{\alpha j}.
\end{align*}
If however $m_{k-1} \ge 0$, then we can use that $v = 0$ in $B_R \setminus \Omega$.
Moreover, if $j > k-1$ we compute for $x \in B_{\eta^{-(k-1)+j} R} \cap \Omega$:
\begin{align*}
\frac{w}{b_{\Omega}}(x) &= \frac{u(x) - m_{k-1}}{M_{k-1} - m_{k-1}} \ge \frac{m_0 - m_{k-1}}{M_{k-1} - m_{k-1}}\\
&\ge \frac{(M_{k-1} - m_{k-1}) - (M_0 - m_0)}{M_{k-1} - m_{k-1}} = 1 - \eta^{\alpha(k-1)} \ge 1 - \eta^{\alpha j}.
\end{align*}
Finally, for $x \in B_{\eta^{-(k-1)+j} R} \setminus \Omega$, again by the same arguments as above, and using that $v \ge m_0$ by construction, we have
\begin{align*}
w(x) = \frac{v(x) - m_{k-1}\1_{ \{ m_{k-1} < 0 \} }}{M_{k-1} - m_{k-1}} \ge \frac{m_0 - m_{k-1}}{M_{k-1} - m_{k-1}} \ge 1 - \eta^{\alpha j}.
\end{align*}

Consequently, all assumptions of \autoref{lemma:growth-lemma} are satisfied for $w$ with radius $\eta^{-(k-1)} R$. Thus, we deduce from \autoref{lemma:growth-lemma} and the choice of $\delta$:
\begin{align*}
u - m_{k-1} &= (M_{k-1} - m_{k-1})\frac{w}{b_{\Omega}} \\
&\ge 2\delta (M_{k-1} - m_{k-1}) - (\eta^{-(k-1)}R)^{1+\alpha} \left(\Vert d^{s-\alpha} f \Vert_{L^{\infty}(\Omega \cap B_{\eta^{-(k-1)}R} )} + c_1 |m_{k-1}| \right)\\
&\qquad\qquad\qquad\qquad ~~ - (\eta^{-(k-1)}R) \Vert  g \Vert_{L^{\infty}(\partial \Omega \cap B_{\eta^{-(k-1)}R} )} \qquad\qquad\qquad\qquad \text{ in } \Omega \cap B_{\eta^{-k} R}.
\end{align*}
Moreover, by \eqref{eq:alpha-delta}, the choice of $M$, \eqref{eq:osc-decay-2}, and the estimate $|m_{k-1}| \le M_0 = \frac{\delta}{2c_1} M$ we estimate
\begin{align*}
(\eta^{-(k-1)}R)^{1+\alpha} & \left(\Vert d^{s-\alpha} f \Vert_{L^{\infty}(\Omega \cap B_{\eta^{-(k-1)}R} )} + c_1 |m_{k-1}| \right) + (\eta^{-(k-1)}R) \Vert g \Vert_{L^{\infty}(\partial \Omega \cap B_{\eta^{-(k-1)}R} )} \\
&\le \eta^{-\alpha (k-1)} \delta M = \delta(M_{k-1} - m_{k-1}).
\end{align*}
Therefore, we deduce
\begin{align*}
m_k := \delta(M_{k-1} - m_{k-1}) + m_{k-1} \le u \le M_{k-1} =: M_k ~~ \text{ in } \Omega \cap B_{\eta^{-k} R},
\end{align*}
which proves \eqref{eq:osc-decay-1} for $k$. \eqref{eq:osc-decay-2} for $k$ follows from \eqref{eq:alpha-delta}. The proof of Step 1 is complete.

Step 2: Now that we have established the claim of Step 1, let us show how to conclude the proof. Let us take $x,y \in B_{R/2}$. We define $k \in \N$ as
\begin{align*}
\inf \{ k \in \N : |x-y| \ge \eta^{-k} (R/2) \}.
\end{align*}
Then, $|x-y| \le \eta^{-k+1} (R/2)$ and by Step 1, it holds
\begin{align*}
\frac{|u(x) - u(y)|}{|x-y|^{\alpha}} &\le \eta^{k\alpha}(R/2)^{-\alpha} \osc_{B_{\eta^{-k+1}(R/2)}} u \\
&\le c R^{-\alpha}  \left(\Vert u \Vert_{L^{\infty}(\Omega)} + \Vert v \Vert_{L^{\infty}(\R^n \setminus \Omega)} + R^{1+\alpha} \Vert d^{s-\alpha} f \Vert_{L^{\infty}(\Omega \cap B_R)} + R^{1+\alpha} + R \Vert g \Vert_{L^{\infty}(\partial\Omega \cap B_R)} \right).
\end{align*}
Note that we can omit the additional summand $+R^{1+\alpha}$ by an additional scaling and normalization argument, i.e., by assuming that $R = 1$ and $\Vert u \Vert_{L^{\infty}(\Omega)} + \Vert v \Vert_{L^{\infty}(\R^n \setminus \Omega)} + \Vert d^{s-\alpha} f \Vert_{L^{\infty}(\Omega \cap B_1)} + \Vert g \Vert_{L^{\infty}(\partial\Omega \cap B_1)} = 1$, applying the previous estimate, and rescaling to general $R$.
This concludes the proof after using that by \autoref{thm:dirichlet} it holds $b_{\Omega}/d^{s-1} \in C^{\alpha}(\Omega \cap B_{R/2}(x_0))$.
\end{proof}

We are now in a position to deduce the boundary H\"older regularity estimate in $C^{1,\gamma}$ domains.

\begin{proof}[Proof of \autoref{thm:bdry-Holder}]
Note that
\begin{align*}
\partial_{\nu}(v/b_{\Omega}) = \partial_{\nu}(v/d^{s-1}) - \partial_{\nu}(b_{\Omega}/d^{s-1}) (v/d^{s-1}),
\end{align*}
and recall that $|\partial_{\nu}(b_{\Omega}/d^{s-1})| \le C$. Hence, we can apply
\autoref{lemma:bdry-osc-decay} (with $R = 1/2$ and varying $x_0 \in \partial \Omega$). Combining it with the interior regularity results from \cite[Theorem 2.4.3]{FeRo23}, and a covering argument, we deduce the desired result. In order to produce the tail-term in the estimate, we employ a truncation argument in the same way as in the proof of \autoref{cor:bdry-Holder}.
\end{proof}

We end this section with a boundary H\"older regularity estimate for solutions that are defined up to a polynomial and might grow fast at infinity.

\begin{corollary}
\label{cor:bdry-Holder}
Let $L \in \mathcal{L}_s^{\text{hom}}(\lambda,\Lambda)$. Let $k \in \N \cup \{ 0 \}$. Let $\Omega \subset \R^n$ be an open, bounded domain with $\partial \Omega \in C^{2,\gamma}$ for some $\gamma > 0$ and $K \in C^{5+2\gamma}(\mathbb{S}^{n-1})$. Let $f \in C(\Omega \cap B_4)$, $g \in C(\overline{\partial \Omega \cap B_4})$, and $v$ with $v/d^{s-1} \in C(\overline{\Omega})$ be a viscosity solution to
\begin{align*}
\begin{cases}
L v &\overset{k}{=} f ~~ \text{ in } \Omega \cap B_4,\\
v &= 0 ~~ \text{ in } B_4 \setminus \Omega,\\
\partial_{\nu}(v/d^{s-1}) &= g ~~ \text{ on }  \partial \Omega \cap B_4. 
\end{cases}
\end{align*}
Then, there exists $\alpha_0 > 0$ such that if for some $\alpha \in (0,\alpha_0]$ we have $d^{s+1-\alpha}f \in L^{\infty}(\Omega \cap B_4)$, then the following holds true:
If $k = 0$, and $v \in L^1_{2s}(\R^n)$, it holds  $v/d^{s-1} \in C^{\alpha}_{loc}(\overline{\Omega} \cap B_4)$, and
\begin{align*}
\left\Vert \frac{v}{d^{s-1}} \right\Vert_{C^{\alpha}(\overline{\Omega} \cap B_1)} \le c \Big( \left\Vert \frac{v}{d^{s-1}} \right\Vert_{L^{\infty}(\Omega \cap B_4)} &+ \Vert v \Vert_{L^1_{2s}(\R^n \setminus B_4)} + \Vert d^{s+1-\alpha} f \Vert_{L^{\infty}(\Omega \cap B_4)} + \Vert g \Vert_{L^{\infty}( \partial \Omega \cap B_4 )} \Big).
\end{align*}
If $k \in \N$, $K \in C^{k-1+\delta}(\mathbb{S}^{n-1})$, $v \in L^1_{2s+k-1+\delta}(\R^n)$ for some $\delta > 0$, it holds 
$v/d^{s-1} \in C^{\alpha}_{loc}(\overline{\Omega} \cap B_4)$, and
\begin{align*}
\left\Vert \frac{v}{d^{s-1}} \right\Vert_{C^{\alpha}(\overline{\Omega} \cap B_1)} \le c \Big( \left\Vert \frac{v}{d^{s-1}} \right\Vert_{L^{\infty}(\Omega \cap B_4)} &+ \left\Vert v  \right\Vert_{L^{1}_{2s+k-1+\delta}(\R^n \setminus B_4)} + \Vert d^{s+1-\alpha} f \Vert_{L^{\infty}(\Omega \cap B_4)} + \Vert g \Vert_{L^{\infty}( \partial \Omega \cap B_4 )} \Big),
\end{align*}
where $c > 0$ and $\alpha_0$, depend only on $n,s,\lambda,\Lambda,\gamma,k,\delta$, and the $C^{2,\gamma}$ radius of $\Omega$.
\end{corollary}

\begin{proof}
In case $k = 0$, the proof follows by a truncation argument. Indeed, let us define $w = v \1_{B_4}$ and observe that
\begin{align*}
L w = f - L(v \1_{\R^n \setminus B_4}) =: \tilde{f} ~~ \text{ in } \Omega \cap B_2.
\end{align*}
Moreover, we can estimate 
\begin{align*}
\Vert d^{s+1-\alpha} \tilde{f} \Vert_{L^{\infty}(\Omega \cap B_2)} \le c\Vert d^{s-1 + \alpha} f \Vert_{L^{\infty}(\Omega \cap B_2)} + c \Vert v \Vert_{L^1_{2s}(\R^n \setminus B_4)}.
\end{align*}
Thus, the desired result follows immediately by application of \autoref{thm:bdry-Holder} to $w$, using that $v = 0$ in $B_4 \setminus \Omega$, by assumption.

Let now $k \in \N$. Again, we define $w = v\1_{B_4}$, but this time, since the equation only holds up to a polynomial, we obtain for any $R > 4$
\begin{align*}
L w = f_R - L(v \1_{B_R \setminus B_4}) + p_R =: \tilde{f} ~~ \text{ in } \Omega \cap B_3,
\end{align*}
where $f_R \to f$ in $d^{s+1-\alpha}L^{\infty}(\Omega \cap B_3)$, as $R \to \infty$, and $p_R \in \cP_{k-1}$. As in the proof of \autoref{lemma:interior-regularity} (see also \cite[Lemma 3.6]{AbRo20} and \cite[Lemma 4.8]{Kuk21}), taking difference quotients of order $k-1+\delta$ of the equation for $w$, and using crucially that $K \in C^{k-1+\delta}(\mathbb{S}^{n-1})$, we can find a polynomial $p \in \cP_{\lfloor k-1 + \delta \rfloor}$ and $h$ with $d^{s+1-\alpha} h \in L^{\infty}(\Omega \cap B_3)$ such that
\begin{align*}
\begin{cases}
L w &= h + p ~~ \text{ in } \Omega \cap B_3,\\
w &= 0 ~~\qquad \text{ in } \R^n \setminus (\Omega \cap B_4),
\end{cases}
\end{align*}
and moreover, $h$ satisfies the estimate
\begin{align}
\label{eq:h-estimate}
\Vert d_{\Omega}^{s+1-\alpha} h \Vert_{L^{\infty}(\Omega \cap B_3)} \le C \left( \Vert d_{\Omega}^{s+1-\alpha} f \Vert_{L^{\infty}(\Omega \cap B_3)} + \left\Vert v |\cdot|^{-n-2s-(k-1+\delta)} \right\Vert_{L^{1}(\R^n \setminus B_4 )} \right).
\end{align}
Next, let us take a bounded domain $D \subset \R^n$ with $\partial D \in C^{1,\gamma}$ such that $\Omega \cap B_2 \subset D \subset \Omega \cap B_3$. Moreover, we find $w_1, w_2$ such that $w_1/d_D^{s-1} , w_1/d_D^{s-1} \in C(\overline{D})$ and $w = w_1 + w_2$ satisfying
\begin{align*}
\begin{cases}
L w_1 &= h ~~ \text{ in } D,\\
w_1 &= w ~~ \text{ in } \R^n \setminus D,\\
w_1/d_{D}^{s-1} &= v/d_{D}^{s-1} ~~ \text{ on } \partial D,
\end{cases}
\qquad \text{ and } \qquad
\begin{cases}
L w_2 &= p ~~ \text{ in } D,\\
w_2 &= 0~~ \text{ in } \R^n \setminus D,\\
w_2/d_{D}^{s-1} &= 0 ~~ \text{ on } \partial D.
\end{cases}
\end{align*}
Note that the existence of $w_2 \in L^{\infty}(\R^n)$ follows from \cite[Theorem 3.2.27]{FeRo23}, and we obtain $w_2/d^s_D \in C^{\gamma}(\overline{D})$ from \cite[Theorem 2.7.1]{FeRo23}, which yields $w_2/d_D^{s-1} = d_D (w_2/d_D^s) \in C^{\gamma}(\overline{D})$ since $\partial D \in C^{1,\gamma}$. Then, we can define $w_1 := w - w_2$.
We claim that
\begin{align}
\label{eq:w-1-claim}
\Vert w_1/d_{D}^{s-1} \Vert_{L^{\infty}(D)} \le c \left( \Vert v/d_{\Omega}^{s-1} \Vert_{L^{\infty}(\Omega \cap B_4)} + \Vert d_{\Omega}^{s+1-\alpha} h \Vert_{L^{\infty}(\Omega \cap B_3)} \right).
\end{align}

To see this, let us recall the function $\psi_1$ (with respect to $D$) from \autoref{lemma:supersol}, and observe that by \autoref{lemma:dist-s-1+eps}, we can take it in such a way that
\begin{align}
\label{eq:psi-1-d-omega-compare}
L (\psi_1 + d_{\Omega}^{s-1}) \ge c_0 d_D^{\alpha -s-1} ~~ \text{ in } D
\end{align}
for some $c_0 > 0$. Moreover, recall $\psi_1 / d_D^{s-1} = 1$ on $\partial D$. Then, let us define 
\begin{align*}
\Psi(x) &= c_1 \psi_1(x) \left( \Vert v/d_{D}^{s-1} \Vert_{L^{\infty}(\partial D)} + \Vert d_D^{s+1-\alpha} h \Vert_{L^{\infty}(D)} + \Vert w/d_{\Omega}^{s-1} \Vert_{L^{\infty}(\Omega \cap B_4)} \right) \\
&\quad + c_1 d_{\Omega}^{s-1}(x) \Vert w/d_{\Omega}^{s-1} \Vert_{L^{\infty}(\Omega \cap B_4)},
\end{align*}
where $c_1 := \max\{ c_0^{-1} , 1\}$, and observe that by \eqref{eq:psi-1-d-omega-compare} we have
\begin{align*}
\begin{cases}
L w_1 &\le L \Psi ~~ \qquad \text{ in } D,\\
w_1 &\le \Psi ~~ \qquad ~~ \text{ in } \R^n \setminus D,\\
w_1/d_D^{s-1} &\le \Psi/d_D^{s-1} ~~~ \text{ on } \partial D,
\end{cases}
\end{align*}
which, recalling that $\psi_1 \le c_1 d_D^{s-1}$ in $D$, and $d_D \le d_{\Omega}$, as well as the definition of $D$, imply that
\begin{align*}
\frac{w_1}{d_D^{s-1}} &\le c_2c_1 \left( \Vert v/d_{D}^{s-1} \Vert_{L^{\infty}(\partial D)} + \Vert d_D^{s+1-\alpha} h \Vert_{L^{\infty}(D)} + \Vert w/d_{\Omega}^{s-1} \Vert_{L^{\infty}(\Omega \cap B_4)} \right) + c_2c_1 \frac{d_{\Omega}^{s-1}}{d_D^{s-1}} \Vert w/d_{\Omega}^{s-1} \Vert_{L^{\infty}(\Omega \cap B_4)} \\
&\le c \left( \Vert v/d_{\Omega}^{s-1} \Vert_{L^{\infty}(\partial \Omega \cap B_3)} + \Vert d_{\Omega}^{s+1-\alpha} h \Vert_{L^{\infty}(\Omega \cap B_3)} +  \Vert v/d_{\Omega}^{s-1} \Vert_{L^{\infty}(\Omega \cap B_4)} \right),
\end{align*}
which yields our claim in \eqref{eq:w-1-claim}.
As a direct consequence of \eqref{eq:w-1-claim}, we deduce
\begin{align}
\label{eq:w-2-claim}
\begin{split}
\Vert w_2/d_{D}^{s-1} \Vert_{L^{\infty}(D)} &\le \Vert w/d_{D}^{s-1} \Vert_{L^{\infty}(D)} + \Vert w_1/d_{D}^{s-1} \Vert_{L^{\infty}(D)} \\
&\le c \left( \Vert v/d_{\Omega}^{s-1} \Vert_{L^{\infty}(\Omega \cap B_4)} + \Vert d_{\Omega}^{s+1-\alpha} h \Vert_{L^{\infty}(\Omega \cap B_3)} \right).
\end{split}
\end{align}

Finally, we claim that 
\begin{align}
\label{eq:poly-bounded-sol-claim}
\Vert p \Vert_{L^{\infty}(D)} \le c \Vert w_2/d_{D}^{s-1} \Vert_{L^{\infty}(D)}.
\end{align}
Note that once we show \eqref{eq:poly-bounded-sol-claim}, then the proof is complete after combination of \eqref{eq:poly-bounded-sol-claim}, \eqref{eq:w-2-claim}, \eqref{eq:h-estimate}, and application of the boundary H\"older regularity estimate (see \autoref{thm:bdry-Holder}) to $w$ in $\Omega$, as in the case $k = 0$. We prove \eqref{eq:poly-bounded-sol-claim} by contradiction. Suppose there are sequences $(L_j)_j$, $(w_j)_j$, $(p_j)_j$ with
\begin{align*}
\Vert p_j \Vert_{L^{\infty}(D)} = 1, \qquad \text{ and } \qquad 
\begin{cases}
L_j w_j &= p_j ~~ \text{ in } D,\\
w_j &= 0 ~~~ \text{ in } \R^n \setminus D,\\
w_j/d_{\Omega}^{s-1} &= 0 ~~~ \text{ on } \partial D,\\
\lim_{j \to \infty} \Vert w_j/d^{s-1}_{D} \Vert_{L^{\infty}(D)} &= 0.
\end{cases}
\end{align*}
Then, up to subsequences, it holds $L_{j_m} \rightharpoonup L$, $w_{j_m}/d_{D}^{s-1} \to u_0$ in $L^{\infty}(D)$ for some $u_0 \in L^{\infty}(D)$, $p_{j_m} \to p_0$ in $L^{\infty}(D)$. While the first convergence statement follows from \cite[Lemma 3.7]{AbRo20}, the second convergence statement follows from \autoref{thm:bdry-Holder} and the Arzel\`a-Ascoli theorem, and the third one is immediate from the boundedness of $(p_{j_m})$ in a finite dimensional space.\\
We can now make use of the stability result in \cite[Proposition 2.2.36]{FeRo23}, and deduce that for $w_0 = d_D^{s-1} u_0$, it holds
\begin{align*}
\Vert p_0 \Vert_{L^{\infty}(D)} = 1, \qquad \text{ and } \qquad 
\begin{cases}
L w_0 &= p_0 ~~ \text{ in } D,\\
w_0 &= 0 ~~~ \text{ in } \R^n \setminus D,\\
w_0/d_{D}^{s-1} &= 0 ~~~ \text{ on } \partial D,\\
\Vert w_0/d^{s-1}_{D} \Vert_{L^{\infty}(D)} &= 0.
\end{cases}
\end{align*}
Clearly, $w_0 = 0$ is not a solution to $Lw_0 = p_0$ in $D$, so we have obtained a contradiction, and conclude the proof of \eqref{eq:poly-bounded-sol-claim}.
\end{proof}

\section{Liouville theorem in the half-space}
\label{sec:Liouville}

The proof of our main result (see \autoref{thm:higher-reg}) is based on a blow-up argument. A crucial ingredient in such proof is a suitable Liouville theorem in the half-space. In this section, we will establish such result for nonlocal problems with local Neumann boundary conditions:

\begin{theorem}
\label{thm:Liouville}
Let $L \in \mathcal{L}_s^{\text{hom}}(\lambda,\Lambda)$. Let $k \in \N$, $\gamma \in (0,1)$ with $\gamma \not = s$, and $K \in C^{k-1+\gamma-s +\delta}(\mathbb{S}^{n-1})$ for some $\delta > 0$. Let $u \in C(\R^n)$ be a viscosity solution to 
\begin{align*}
\begin{cases}
L((x_n)_+^{s-1} u) &\overset{k-1+\lceil \gamma -s \rceil}{=} 0 \qquad~~\quad~~ \text{ in } \{ x_n > 0 \},\\
\partial_n u &\quad~~= p ~~~~~~ \qquad\qquad \text{ in } \{ x_n = 0 \},\\
|u(x)| &\quad~~\le C(1 + |x|)^{k+\gamma} ~~~ \forall x \in \{ x_n > 0 \}
\end{cases}
\end{align*}
for some $C > 0$, $p \in \cP_{k-1}$. Then, there exist $a_{\beta} \in \R$ for any $\beta \in (\N \cup \{0\})^n$ with $|\beta| \le k$ such that 
\begin{align*}
u(x) = \sum_{|\beta| \le k} a_{\beta} x_1^{\beta_1} \dots x_n^{\beta_n} ~~ \forall x \in \{ x_n > 0 \}.
\end{align*}
\end{theorem}

In order to prove \autoref{thm:Liouville}, we first establish the following one-dimensional version, which can be proved by combination of the arguments in \cite[Lemma 6.2]{RoSe16} and \cite[Lemma 3.3]{AbRo20}.

\begin{lemma}
\label{lemma:1d-Liouville}
Let $k \in \N$, $\gamma \in (0,1)$ with $\gamma \not=s$, and $u \in C(\R)$ satisfying
\begin{align*}
\begin{cases}
(-\Delta)^s((x_+)^{s-1}u) &\overset{k-1 + \lceil \gamma -s \rceil}{=} 0 ~~  \quad~~\qquad \text{ in } (0,\infty),\\
|u(x)| &\quad~~\le C (1 + |x|)^{k+\gamma} ~~~ \forall x > 0
\end{cases}
\end{align*}
for some $C > 0$. Then, the exist $a_0,a_1,\dots,a_k \in \R$ such that
\begin{align*}
u(x) = \sum_{j = 0}^{k} a_j x^j ~~ \forall x > 0.
\end{align*}
\end{lemma}

\begin{proof}
In case $k = 1$ and $\gamma < s$, the proof is an application of \cite[Lemma 6.2]{RoSe16} with $u(x) := (x_+)^{s-1}u(x)$, and $\delta = s > 0$, $\beta = s + \gamma \in (0,2s)$.\\
In case $k > 1$ or $\gamma > s$, we have $k-1 + \lceil \gamma -s \rceil \ge 1$, let us define $v(x) = (x_+)^{s-1}u(x)$, let $V : \R \times [0,\infty) \to \R$ be the harmonic extension of $v$ in the sense of \cite[Lemma 3.3]{AbRo20}, and finally define $\tilde{V}(x,y) = \int_{-\infty}^x V(z,y) \d z$. Note that $\tilde{V}$ satisfies (see \cite[Lemma 3.3]{AbRo20})
\begin{align*}
\begin{cases}
\dvg(y^{1-2s}\nabla \tilde{V}(x,y)) &= 0 ~~ \qquad\qquad\qquad\qquad\qquad\qquad\qquad\qquad\qquad\qquad\qquad \text{ in } \R \times (0,\infty),\\
\tilde{V}(x,y) &= v(x) ~~ \qquad\qquad\qquad\qquad\qquad\qquad\qquad\qquad\qquad\qquad ~~ \text{ on } \R \times \{ 0 \},\\
|\tilde{V}(x,y)| &\le C (1 + |x|^{2(k-1 + \lceil \gamma -s \rceil )+1+\gamma+s} + |y|^{(k-1 + \lceil \gamma -s \rceil) + 1 + \gamma + s})  ~~ \text{ in } \R \times (0,\infty).
\end{cases}
\end{align*}
Next, by \cite[Lemma 6.2]{RoSe16} (see also \cite[Theorem 1.10.16]{FeRo23}), we have the representation formula
\begin{align*}
\tilde{V}(x,y) = \tilde{V}(r \cos \theta , r \sin \theta) = \sum_{j = 0}^{\infty} a_j \Theta_j(\theta) r^{j+s}, ~~ \forall  x \in \R,~  y \in [0,\infty), 
\end{align*}
where $a_j \in \R$, and $(\Theta_j)_j$ is a complete orthogonal system in the subspace of even functions in $L^2((0,\pi),(\sin \theta)^{1-2s} \d \theta)$. By the Parseval identity, the bounds on $|\tilde{V}|$ imply 
\begin{align*}
\sum_{j = 0}^{\infty} a_j^2 R^{2+2j} = \int_{\partial B_R \cap \{ y > 0 \}} \tilde{V}(x,y)^2 y^{1-2s} \d \sigma \le C R^{4(k-1+\lceil \gamma -s \rceil)+2+2\gamma+2} = C R^{4(k+\lceil \gamma -s \rceil)+2\gamma}.
\end{align*}
Therefore, it must be $a_j = 0$ for any $j > j_0$, where $j_0 = \min\{ j \in \N : 2 + 2j > 4(k+ \lceil \gamma - s \rceil) +  2\gamma \}$, which implies
\begin{align*}
\tilde{V}(x,y) = \sum_{j = 0}^{j_0} a_j \Theta_j(\theta) r^{j+s} ~~ \forall x \in \R,~  y \in [0,\infty).
\end{align*}
Upon recalling the definition of $V$ and $\tilde{V}$, this implies 
\begin{align*}
v(x) = (x_+)^s \sum_{j = 0}^{j_0 - 1} b_j x^j
\end{align*}
for some $b_j \in \R$, and since $|v(x)| \le C(1 + |x|)^{k+\gamma-1+s}$ and $\gamma \in (0,1)$ by assumption, it must be $b_j = 0$ for any $j \ge k$. Recalling that by definition $u(x) = v(x)(x_+)^{1-s}$, we deduce that $u$ must be a polynomial of degree at most $k$ in $\{ x > 0\}$, as desired.
\end{proof}

Moreover, we will need the following lemma (see also \cite[Proposition 4.3]{Kuk21}):

\begin{lemma}
\label{lemma:large-solution-linear-fct}
Let $L \in \mathcal{L}_s^{\text{hom}}(\lambda,\Lambda)$. Let $k \in \N$, $K \in C^{k-2+\delta}(\mathbb{S}^{n-1})$ for some $\delta > 0$. Let $f \in \cP_k$. Then,
\begin{align*}
L((x_n)_+^{s-1} f) \overset{k-1}{=} 0 ~~ \text{ in } \{ x_n > 0 \}.
\end{align*}
\end{lemma}

\begin{proof}
Let us first give a simple proof in case $k =1$. Then, it suffices to prove that for any $i \in \{ 1, \dots, n \}$
\begin{align*}
L((x_n)_+^{s-1} x_i) = 0 ~~ \text{ in } \{ x_n > 0 \}.
\end{align*}
First, by integrating $x \mapsto (x_n)_+^{s-1}$ in $x_i$, and using that $L((x_n)_+^{s-1}) = 0$, we deduce
\begin{align*}
L((x_n)_+^{s-1} x_i) \equiv c ~~ \text{ in } \{ x_n > 0 \}
\end{align*}
for some constant $c \in \R$.
Then, since $x \mapsto (x_n)_+^{s-1} x_i$ is homogeneous of degree $s$, we deduce that for any $\lambda > 0$ and $x \in \{ x_n > 0 \}$:
\begin{align*}
c = L((x_n)_+^{s-1} x_i)(\lambda x) = \lambda^{-2s} L((\lambda x_n)_+^{s-1} \lambda x_i)(x) = \lambda^{-s} L((x_n)_+^{s-1} x_i)(x) = \lambda^{-s} c.
\end{align*}
This implies that $c = 0$, as desired.\\
For $k \ge 2$, we prove the result by induction. Assume that we know already
\begin{align}
\label{eq:up-to-polynomial-induction}
L((x_n)_+^{s-1}p) \overset{k-2}{=} 0 ~~ \text{ in } \{ x_n > 0\}
\end{align}
for every $p \in \cP_{k-1}$. Now, let $q \in \cP_{k}$. Then, by integrating \eqref{eq:up-to-polynomial-induction} with $p := \partial_i q$ for $i \in \{ 1, \dots, n \}$, by \autoref{lemma:derivative-up-to-a-polynomial} we find that there exists a constant $c \in \R$ such that
\begin{align*}
L((x_n)_+^{s-1}q) \overset{k-1}{=} c ~~ \text{ in } \{ x_n > 0\}.
\end{align*}
Since $c \overset{k-1}{=} 0$ for any $k \ge 2$, we conclude the proof.
\end{proof}

Finally, we state a H\"older regularity estimate in the half-space, which follows from \autoref{cor:bdry-Holder}.

\begin{corollary}
\label{cor:bdry-Holder-half-space}
Let $L \in \mathcal{L}_s^{\text{hom}}(\lambda,\Lambda)$. Let $k \in \N \cup \{ 0 \}$, $\gamma > 0$, and $K \in C^{k-1+\delta}(\mathbb{S}^{n-1})$ for some $\delta > 0$. Let $f \in C(\{ x_n > 0 \} \cap B_2)$, $g \in C(\overline{\{ x_n = 0 \} \cap B_2})$, and $u \in C( \{ x_n \ge 0\} )$ be a viscosity solution to
\begin{align*}
\begin{cases}
L((x_n)_+^{s-1} u) &\overset{k}{=} f ~~ \text{ in } \{ x_n > 0 \} \cap B_2,\\
\partial_n u &= g ~~ \text{ on }  \{ x_n = 0 \} \cap B_2.
\end{cases}
\end{align*}
Then, there exists $\alpha_0 > 0$ such that if  $(x_n)_+^{s+1-\alpha}f \in L^{\infty}(\{ x_n > 0 \} \cap B_2)$ for some $\alpha \in (0,\alpha_0]$, then the following holds true: If $k = 0$ and $(x_n)_+^{s-1} u \in L^1_{2s}(\R^n)$ it holds $u \in C^{\alpha}_{loc}(\{ x_n \ge 0 \}\cap B_2)$, and
\begin{align*}
\left\Vert u \right\Vert_{C^{\alpha}( \{ x_n \ge 0 \} \cap B_1)} \le c \big( \left\Vert u \right\Vert_{L^{\infty}(\{ x_n > 0 \} \cap B_4 )} &+ \Vert (x_n)_+^{s-1} u \Vert_{L^1_{2s}(\R^n \setminus B_4)} \\
&+ \Vert (x_n)_+^{s+1-\alpha} f \Vert_{L^{\infty}(\{ x_n > 0 \} \cap B_2 )} + \Vert g \Vert_{L^{\infty}( \{ x_n = 0 \} \cap B_2 )} \big),
\end{align*}
and if $k \in \N$ and $(x_n)_+^{s-1} u \in L^1_{2s+(k-1+\delta)}(\R^n)$ it holds $u \in C^{\alpha}_{loc}(\{ x_n \ge 0 \}\cap B_2)$, and
\begin{align*}
\left\Vert u \right\Vert_{C^{\alpha}( \{ x_n \ge 0 \} \cap B_1)} \le c \big( \left\Vert u \right\Vert_{L^{\infty}(\{ x_n > 0 \} \cap B_4 )} &+ \big\Vert [(x_n)_+^{s-1} u] |\cdot|^{-n-2s-(k-1+\delta)} \big\Vert_{L^1(\{ x_n > 0 \} \setminus B_4)} \\
&+ \Vert (x_n)_+^{s+1-\alpha} f \Vert_{L^{\infty}(\{ x_n > 0 \} \cap B_2 )}  + \Vert g \Vert_{L^{\infty}( \{ x_n = 0 \} \cap B_2 )} \big),
\end{align*}
where $c > 0$ and $\alpha_0$ depend only on $n,s,\lambda,\Lambda,\gamma,k,\delta$.
\end{corollary}

\begin{proof}
The result follows directly from \autoref{cor:bdry-Holder} applied to some domain $\Omega \subset \R^n$ with $\partial \Omega \in C^{1,\gamma}$, which satisfies $\{ x_n > 0 \} \cap B_2 \subset \Omega \subset \{ x_n > 0 \} \cap B_4$.
\end{proof}

With the help of the one-dimensional Liouville theorem in the half-space and the H\"older regularity estimate up to the boundary (see \autoref{cor:bdry-Holder-half-space}), the proof of \autoref{thm:Liouville} follows by a standard procedure, which is explained in detail for instance in \cite[Proof of Theorem 3.10]{AbRo20}.

\begin{proof}[Proof of \autoref{thm:Liouville}]
First, we observe that by scaling \autoref{cor:bdry-Holder-half-space}, we obtain that for any $R \ge 1$:
\begin{align}
\label{eq:Liouville-Holder}
\begin{split}
[ u ]_{C^{\alpha}(B_{R})} \le c R^{-\alpha} \big[ & \Vert u \Vert_{L^{\infty}(B_{4R})} + R^{1+s} \Vert(x_n)_+^{s-1} u |\cdot|^{-n-2s} \Vert_{L^1(\R^n \setminus B_{4R})} \1_{\{k=1 \text{ and } \gamma < s \}} \\
&+ R^{s+k-1+ \lceil \gamma -s \rceil +\eta} \left\Vert [(x_n)_+^{s-1} u] |\cdot|^{-n-2s-(k-2+\lceil \gamma -s \rceil+\eta)} \right\Vert_{L^1(\{ x_n > 0\} \setminus B_{4R})} \1_{\{ k \ge 2 \text{ or } \gamma > s \} } \\
& + R \Vert p \Vert_{L^{\infty}( \{ x_n = 0\} \cap B_{4R})} \big] \le cR^{k + \gamma - \alpha},
\end{split}
\end{align}
where we take $\eta = 1+\gamma - s - \lceil \gamma -s \rceil + \delta$ and used in the last estimate the growth condition on $u$, the fact that $\Vert p \Vert_{L^{\infty}( \{ x_n = 0\} \cap B_{R})} \le c R^{k-1}$, and the following computation  using polar coordinates with $y_n = r \cos \theta$ for some $\theta \in [0,2 \pi)$ (similar to the proof of \autoref{lemma:growth-lemma}), which is slightly different in case ($k = 1$ and $\gamma < s$) and ($k \ge 2$ or $\gamma > s$). In case $k=1$ and $\gamma < s$, we obtain:
\begin{align}
\label{eq:tail-polar-coord}
\begin{split}
R^{1+s} \Vert(x_n)_+^{s-1} u |\cdot|^{-n-2s} \Vert_{L^1(\R^n \setminus B_{4R})} &\le C R^{1+s} \int_{\R^n \setminus B_{4R}} (y_n)_+^{s-1} |y|^{-n-2s+1+\gamma} \d y \\
&\le c R^{1+s} \int_0^{2\pi} \cos(\theta)_+^{s-1} \left( \int_{4R}^{\infty} r^{s-1} r^{-1-2s+1+\gamma}  \d r \right) \d \theta \\
& \le c R^{1+s} R^{\gamma - s } \left( \int_0^{2\pi} \cos(\theta)_+^{s-1} \d \theta \right) \le c R^{1 + \gamma}.
\end{split}
\end{align}
In case ($k \ge 2$ or $\gamma > s$), we obtain, using that $\eta > 1 +\gamma -s - \lceil \gamma -s \rceil$:
\begin{align*}
R^{s+k-1+\lceil \gamma -s \rceil+\eta} \big\Vert [(x_n)_+^{s-1} u] & |\cdot|^{-n-2s-(k-2+\lceil \gamma -s \rceil+\eta)} \big\Vert_{L^1(\{ x_n > 0\} \setminus B_4)} \\
&\le C R^{s+k-1+\lceil \gamma -s \rceil+\eta} \int_{\R^n \setminus B_{4R}} (y_n)_+^{s-1} |y|^{-n-2s-(k-2+\lceil \gamma -s \rceil+\eta)+k+\gamma} \d y \\
&\le c R^{s+k-1+\lceil \gamma -s \rceil+\eta} \int_0^{2\pi} \cos(\theta)_+^{s-1} \left( \int_{4R}^{\infty} r^{s-1} r^{-1-2s+2-\lceil \gamma -s \rceil+\gamma-\eta}  \d r \right) \d \theta \\
& \le c R^{k+s} R^{\gamma - s } \left( \int_0^{2\pi} \cos(\theta)_+^{s-1} \d \theta \right) \le c R^{k + \gamma}.
\end{align*}

Next, let us take any $\tau \in \mathbb{S}^{n-1}$ such that $\tau_n = 0$ and $0 < h < R/2$. We consider the difference quotients
\begin{align*}
w_{1,\tau}(x) = \frac{u(x+ h \tau) - u(x) }{h^{\alpha}}, \qquad p_{1,\tau}(x)= \frac{p(x+ h \tau) - p(x) }{h^{\alpha}}
\end{align*}
and deduce from \eqref{eq:Liouville-Holder} (after applying the estimate to smaller balls of radius comparable to $R$ inside $B_R$) that 
\begin{align*}
\Vert w_{1,\tau} \Vert_{L^{\infty}(B_R)} \le c R^{k+\gamma-\alpha} ~~ \forall R \ge 1.
\end{align*}
Clearly, since $\tau_n = 0$, $w_{1,\tau}$ satisfies in the viscosity sense
\begin{align}
\label{eq:diff-quot-PDE}
\begin{cases}
L((x_n)_+^{s-1} w_{1,\tau}) &\overset{k-1+\lceil \gamma -s \rceil}{=} 0 ~~ \text{ in } \{ x_n > 0 \},\\
\partial_n w_{1,\tau} &\quad~~= p_{1,\tau} ~~~~ \text{ on } \{ x_n =  0 \}.
\end{cases}
\end{align}
Here, we are using that sums of viscosity solutions are again viscosity solutions by \autoref{lemma:sum-viscosity}. Using \eqref{eq:diff-quot-PDE} and also that $|p_{1,\tau}(x)| \le c |x|^{k-1-\alpha}$ since $p \in \cP_{k-1}$, we can apply the previous arguments to $w_{1,\tau}$. Eventually, this implies that $w_{2, \tau}(x) = \frac{w_{1,\tau}(x+h \tau) - w_{1,\tau}(x)}{h^{\alpha}}$ satisfies $\Vert w_{2, \tau} \Vert_{L^{\infty}(B_R)} \le c R^{k+\gamma - 2\alpha}$.
This way, we obtain higher order difference quotients $w_{j,\tau}$, $j \in \N$, and they satisfy
\begin{align*}
\begin{cases}
L((x_n)_+^{s-1} w_{j,\tau}) &\overset{k-1+\lceil \gamma -s \rceil}{=} 0 ~~ \quad~~ \text{ in } \{ x_n > 0 \},\\
\partial_n w_{j,\tau} &\quad~~ = p_{j,\tau} ~~\qquad~ \text{ on } \{ x_n = 0 \},\\
\Vert w_{j,\tau} \Vert_{L^{\infty}(B_R)} &\quad~~\le c R^{k+\gamma-j\alpha} ~~ \forall R \ge 1,\\
\Vert p_{j,\tau} \Vert_{L^{\infty}(B_R)} &\quad~~\le c R^{k-1-j\alpha} ~~ \forall R \ge 1. 
\end{cases}
\end{align*}
Then, taking $j_0 \in \N$ as the smallest number such that $j_0 \alpha > k + \gamma$, and upon taking the limit $R \to \infty$, we deduce that
\begin{align*}
\lim_{R \to \infty} \Vert w_{j_0,\tau} \Vert_{L^{\infty}(B_R)} \le c \lim_{R \to \infty} R^{k+ \gamma - j_0 \alpha} = 0,
\end{align*}
i.e., $w_{j_0 ,\tau} \equiv 0$ in $\R^n$. Thus, $w_{j_0 - 1,\tau}$ is a function that is constant in the $\tau$-direction. Clearly, we can also take difference quotients of $w_{j_0 - 1,\tau}$ in other directions $\tau' \in \mathbb{S}^{n-1}$ with $\tau'_n = 0$, and the same arguments as before apply. Therefore, $w_{j_0 - 1 , \tau}(x) = w_{j_0-1,\tau}(x_n)$ is one-dimensional for any $\tau \in \mathbb{S}^{n-1}$ with $\tau_n = 0$.\\
Unraveling the higher order difference quotients, we get that $w_{j_0-2,\tau}(x) = (V_1(x_n), x') + V_2(x_n)$ for some one-dimensional functions $V_1 : \R \to \R^{n-1}$ and $V_2 : \R \to \R$, and continuing this argument $j_0 - 1$ times, we deduce that $u$ must be a polynomial in $x'$ with coefficients that are one-dimensional functions from $\R \to \R$ in $x_n$.\\ 
Then, by the growth condition on $u$, for any multi-index $\beta \in (\N \cup \{ 0 \})^{n-1}$ with $|\beta| \le k$, we obtain functions $A_{\beta}$ in $x_n$ such that
\begin{align*}
u(x) = \sum_{|\beta| \le k} (x')^{\beta} A_{\beta}(x_n).
\end{align*}
In particular, this implies that $D^{\beta}_{x'} u(x) = c(\beta) A_{\beta}(x_n)$ for any $|\beta| = k$ and some constant $c(\beta) > 0$, where $D^{\beta}_{x'}$ denotes an incremental quotient approximating the partial derivative $\partial^{\beta}_{x'}$ in the $x'$-variables. Therefore, discretely differentiating the equation for $u$, we deduce
\begin{align*}
c(\beta) L((x_n)_+^{s-1}A_{\beta}) (x) = L((x_n)_+^{s-1}D_{x'}^{\beta} u) (x) = L(D^{\beta}_{x'} [(x_n)_+^{s-1} u]) (x) \overset{k-1+\lceil \gamma -s \rceil}{=} 0 ~~ \text{ in } \{ x_n > 0 \}.
\end{align*}
By the growth condition on $u$, it must be $|A_{\beta}(x_n)| \le c (1 + |x_n|)^{k-|\beta|+\gamma} = c(1+ |x|)^{\gamma}$, and since $A_{\beta}$ was also one-dimensional, i.e., $L A_{\beta} = (-\Delta)^s_{\R} A_{\beta}$, we can apply \autoref{lemma:1d-Liouville} to $A_{\beta}$, which yields $A_{\beta}(x_n) = p_{\beta}(x_n)$ for some polynomial $p_{\beta} \in \cP_{k-|\beta|} = \cP_{0}$. 
Next, we recall from \autoref{lemma:large-solution-linear-fct} 
\begin{align*}
L((x_n)_+^{s-1}(x')^{\beta}p_{\beta}(x_n)) \overset{k-1+\lceil \gamma -s \rceil}{=} 0 ~~ \text{ in } \{ x_n > 0\}.
\end{align*}
Thus, repeating the arguments from above, we deduce that, for every $\beta$ with $|\beta| \le k$ it holds
\begin{align*}
\begin{cases}
L((x_n)_+^{s-1}A_{\beta}) &\overset{k-1+\lceil \gamma -s \rceil}{=} 0 ~~ \qquad\qquad~~~ \text{ in } \{ x_n > 0 \},\\
|A_{\beta}(x)| &\quad ~~\le C(1 + |x|)^{k - |\beta| + \gamma} ~~ \forall x \in \{ x_n > 0 \},
\end{cases}
\end{align*}
and hence $A_{\beta}(x_n) = p_{\beta}(x_n)$ for some polynomial $p_{\beta} \in \cP_{k-|\beta|}$. This implies $u(x) = p(x)$ for some polynomial $p$, and by the growth condition on $u$, it must be $p \in \cP_{k}$, as desired.
\end{proof}

\section{Higher order boundary regularity}
\label{sec:higher}

The goal of this section is to prove the desired higher order boundary regularity for nonlocal equations with local Neumann conditions (see \autoref{thm:higher-reg}). The proof goes by a blow-up argument and heavily uses the Liouville theorem in the half-space (see \autoref{thm:Liouville}), as well as the boundary H\"older estimate (see \autoref{thm:bdry-Holder}).

\begin{lemma}
\label{lemma:expansion}
Let $L \in \mathcal{L}_s^{\text{hom}}(\lambda,\Lambda)$. Let $k \in \N$ and $\Omega \subset \R^n$ be an open, bounded domain with $\partial \Omega \in C^{k+1,\gamma}$ for some $\gamma \in (0,1)$ with $\gamma \not= s$, and $K \in C^{2k+2\gamma + 3}(\mathbb{S}^{n-1})$. Let $v \in L^{1}_{2s}(\R^n)$ with $v/d^{s-1} \in C(\overline{\Omega})$ be a viscosity solution to
\begin{align*}
\begin{cases}
L v &= f ~~ \text{ in } \Omega \cap B_{1},\\
v &= 0 ~~ \text{ in } B_1 \setminus \Omega,\\
\partial_{\nu}(v/d^{s-1}) &= g ~~ \text{ on } \partial \Omega \cap B_1.
\end{cases}
\end{align*}

\begin{itemize}
\item[(i)] If $k = 1$ and $\gamma < s$, $f \in C(\Omega \cap B_1)$ with $d^{s-\gamma} f \in L^{\infty}(\Omega \cap B_1)$, $g \in C^{\gamma}(\partial \Omega \cap B_1)$, then for any $x_0 \in \partial \Omega \cap B_{1/2}$ and $x \in \Omega \cap B_{1/2}$ it holds
\begin{align*}
\qquad \qquad  \Bigg| & \frac{v}{d^{s-1}}(x) - \Big(\frac{v}{d^{s-1}}(x_0) - A(x_0) \cdot (x - x_0)\Big)\Bigg| \\
&\le c\left(\left\Vert \frac{v}{d^{s-1}} \right\Vert_{L^{\infty}(\Omega)} + \Vert v \Vert_{L^{\infty}(\R^n \setminus \Omega)} + \Vert d^{s-\gamma} f \Vert_{L^{\infty}(\Omega \cap B_1)} + \Vert g \Vert_{C^{\gamma}(\partial\Omega \cap B_1)} \right)|x-x_0|^{1+\gamma}
\end{align*}
for some $c > 0$, which only depends on $n,s,\lambda,\Lambda,\gamma$, and the $C^{2,\gamma}$ radius of $\Omega$.
If in addition, it holds $g \equiv 0$, then $A(x_0) \cdot \nu_{x_0} = 0$.

\item[(ii)] If $k \ge 2$ or $\gamma >s$, $f \in C^{(k-1)-s+\gamma}(\Omega \cap B_1)$, $g \in C^{k-1 + \gamma}(\partial \Omega \cap B_1)$, then for any $x_0 \in \partial \Omega \cap B_{1/2}$, there is $Q(\cdot ; x_0) \in \cP_{k}$ such that for any $x \in \Omega \cap B_{1/2}$ it holds
\begin{align*}
\qquad \qquad \Bigg| & \frac{v}{d^{s-1}}(x) - Q(x; x_0) \Bigg| \\
&\le c\left(\left\Vert \frac{v}{d^{s-1}} \right\Vert_{L^{\infty}(\Omega)} + \Vert v \Vert_{L^{\infty}(\R^n \setminus \Omega)} + \Vert f \Vert_{C^{(k-1)-s+\gamma}(\Omega \cap B_1)} + \Vert g \Vert_{C^{k-1+\gamma}(\partial\Omega \cap B_1)} \right)|x-x_0|^{k+\gamma}
\end{align*}
for some $c > 0$, which only depends on $n,s,\lambda,\Lambda,\gamma,k$, and the $C^{k+1,\gamma}$ radius of $\Omega$.
\end{itemize}
\end{lemma}

\begin{proof}
Let us assume without loss of generality that $x_0 = 0 \in \partial \Omega$ with $\partial_{\nu_0} = e_n$. We set $u:= v/d^{s-1}$. \\
We will prove the desired result by a blow-up argument. To do so, we assume by contradiction that for any $j \in \N$ there exist $C^{k+1,\gamma}$ domains $\Omega_j \subset \R^n$, $f_j \in C^{k-1}(\Omega_j \cap B_1)$, $g_j \in C^{k-1+\gamma}(\partial \Omega_j \cap B_1)$, $r_j > 0$, operators $L_j$ with ellipticity constants $\lambda,\Lambda$, and $v_j \in C(\Omega_j) \cap L^{1}_{2s}(\R^n)$ viscosity solutions to 
\begin{align*}
\begin{cases}
L_j v_j &= f_j ~~ \text{ in } \Omega_j \cap B_1,\\
v_j &= 0 ~~~ \text{ in } B_1 \setminus \Omega_j,\\
\partial_{\nu} (v_j/d_j^{s-1}) &= g_j ~~ \text{ on } \partial \Omega_j \cap B_1,
\end{cases}
\end{align*}
such that 
\begin{align*}
|\diam \Omega_j| &+ \Vert u_j \Vert_{L^{\infty}(\Omega_j)} + \Vert v_j \Vert_{L^{\infty}(\R^n \setminus \Omega_j)} +\1_{\{k = 1 \text{ and } \gamma < s \}} \Vert d_j^{s-\gamma} f_j \Vert_{L^{\infty}(\Omega_j \cap B_1)} \\
&+ \1_{\{ k \ge 2 \text{ or } \gamma > s \}}\Vert f_j \Vert_{C^{(k-1)-s+\gamma}(\Omega_j \cap B_1)} + \Vert g_j \Vert_{C^{k-1+\gamma}(\Omega_j \cap B_1)} + \Vert d_{j} \Vert_{C^{k+1,\gamma}(\Omega_{j} \cap B_1)} \le C
\end{align*}
for some $C > 0$, denoted $d_{\Omega_j} = d_j$, and used that $d_j \in C^{k+1,\gamma}$ by \cite[Definition 2.7.5]{FeRo23}. Finally, we assume by contradiction
\begin{align*}
\sup_{j \in \N} \sup_{r > 0} r^{-k-\gamma} \Vert u_j - Q \Vert_{L^{\infty}( \Omega_j \cap B_r )} = \infty ~~ \forall Q \in \mathcal{P}_k.
\end{align*}
Observe that up to a rotation, $r_m^{-1} \Omega_{j_m} \cap B_{r_m^{-1}} \to \{ x_n > 0 \}$. Moreover, we will write $\tilde{d}_{j_m}\1_{r_m^{-1}\Omega_{j_m}} := \tilde{d}_{j_m} =: r_m^{-1} d_{j_m}(r_m \cdot)$ for the (regularized) distance with respect to $r_m^{-1} \Omega_{j_m}$.\\
We consider the $L^2(\Omega_j \cap B_r)$-projections of $u_j$ over $\mathcal{P}_k$, and denote them by $Q_{j,r} \in \mathcal{P}_k$. They satisfy the following properties:
\begin{align*}
\Vert u_j - Q_{j,r} \Vert_{L^2( \Omega_j \cap B_r )} &\le \Vert u_j - Q \Vert_{L^2(\Omega_j \cap B_r )} ~~ \forall Q \in \mathcal{P}_k,\\
\int_{\Omega_j \cap B_r} \left( u_j(x) - Q_{j,r}(x) \right) Q(x) \d x &= 0 \qquad\qquad\qquad\qquad~ \forall Q \in \mathcal{P}_k.
\end{align*}
Next, we introduce 
\begin{align}
\label{eq:theta-def}
\theta(r) := \sup_{j \in \N} \sup_{\rho \ge r} \rho^{-k-\gamma} \Vert u_j - Q_{j,\rho} \Vert_{L^{\infty}(\Omega_j \cap B_r )}.
\end{align}
Observe that $\theta(r) \nearrow \infty$, as $r \searrow 0$. This follows from \cite[Lemma 4.3]{AbRo20} applied with $s = 0$ (note that the proof remains exactly the same in this case).\\
As a consequence, there exist sequences $(r_m)_m$ and $(j_m)_m$ such that
\begin{align}
\label{eq:vm-nondeg}
\frac{\Vert u_{j_m} - Q_{j_m,r_m} \Vert_{L^{\infty}(\Omega_{j_m} \cap B_{r_m} )} }{ r_m^{k+\gamma} \theta(r_m)} \ge \frac{1}{2} ~~ \forall m \in \N.
\end{align}
Let us define for any $m \in \N$,
\begin{align}
\label{eq:wm-def}
w_m(x) = \frac{ u_{j_m}(r_m x) - Q_{j_m,r_m}(r_m x) }{r_m^{k+\gamma} \theta(r_m)},
\end{align}
and observe that by construction, we have
\begin{align}
\label{eq:orthogonality}
\Vert w_m \Vert_{L^{\infty}(r_m^{-1} \Omega_{j_m} \cap B_1 )} \ge \frac{1}{2}, \qquad \int_{r_m^{-1} \Omega_{j_m} \cap B_1 } w_m(x) Q(r_m x) \d x = 0  ~~ \forall m \in \N, ~~ \forall Q \in \mathcal{P}_k.
\end{align}
Next, we claim that
\begin{align}
\label{eq:vm-growth}
\Vert w_m \Vert_{L^{\infty}(r_m^{-1} \Omega_{j_m} \cap B_R )} \le c R^{k+\gamma} ~~ \forall R \ge 1, ~~ \forall m \in \N.
\end{align}
To see this, we estimate for any $R \ge 1$, using the definitions of $\theta(R r_m)$ and $w_m$ (see \eqref{eq:theta-def} and \eqref{eq:wm-def}):
\begin{align}
\label{eq:wm-estimate}
\begin{split}
\Vert w_m \Vert_{L^{\infty}(r_m^{-1}\Omega_{j_m} \cap B_{R} )} &\le \frac{ \Vert u_{j_m} - Q_{j_m,R r_m} \Vert_{L^{\infty}(\Omega_{j_m} \cap B_{R r_m} )} }{r_m^{k+\gamma} \theta(r_m)}  + \frac{ \Vert Q_{j_m,R r_m} - Q_{j_m,r_m} \Vert_{L^{\infty}(\Omega_{j_m} \cap B_{R r_m} )} }{r_m^{k+\gamma} \theta(r_m)} \\
&\le \frac{ (R r_m)^{k+\gamma} \theta(R r_m) }{r_m^{k+\gamma} \theta(r_m)} + \frac{ \Vert Q_{j_m,R r_m} - Q_{j_m,r_m} \Vert_{L^{\infty}(\Omega_{j_m} \cap B_{R r_m} )} }{r_m^{k+\gamma} \theta(r_m)}.
\end{split}
\end{align}
Moreover, it follows that for any $j \in \N$, $r > 0$, and $R \ge 1$:
\begin{align}
\label{eq:Q-difference}
\Vert Q_{j,Rr} - Q_{j,r} \Vert_{L^{\infty}(\Omega_{j} \cap  B_{R r} )} \le c \theta(r) (Rr)^{k+\gamma}.
\end{align}
Indeed, if we write
\begin{align*}
Q_{j,r}(x) = \sum_{|\beta| \le k} a_{j,r}^{(\beta)} x_1^{\beta_1} \cdots x_n^{\beta_n}, ~~ \beta \in \N^n, ~a_{j,r}^{(\beta)} \in \R,
\end{align*}
then by \cite[Lemma A.10]{AbRo20} we have for any $|\alpha| \le k$
\begin{align*}
r^{|\beta|} & |a_{j,r}^{(\beta)} - a_{j,2r}^{(\beta)}| \\
&\le  c \Vert Q_{j,r} - Q_{j,2r} \Vert_{L^{\infty}(\Omega_{j} \cap  B_{r} )} \\
&\le c \Vert u_j - Q_{j,r} \Vert_{L^{\infty}(\Omega_{j} \cap  B_{r} )} + c \Vert u_j - Q_{j,2r} \Vert_{L^{\infty}(\Omega_{j} \cap  B_{2r} )} \\
&\le c \theta(r) r^{k+\gamma} + c \theta(2r) (2r)^{k+\gamma} \le c \theta(r) (2r)^{k+\gamma}.
\end{align*}
By iteration of this inequality, we obtain for any $l \in \N$
\begin{align*}
|a_{j,r}^{(\beta)} - a_{j,2^l r}^{(\beta)}| &\le \sum_{i = 0}^{l-1} |a_{j,2^i r}^{(\beta)} - a_{j,2^{i+1} r}^{(\beta)}| \le c \sum_{i = 0}^{l-1} \theta(2^i r) (2^i r)^{k+\gamma - |\beta|} \\
&\le c \theta(r) r^{k+\gamma - |\beta|} \sum_{i = 0}^{l-1} \frac{\theta(2^i r)}{ \theta(r)} 2^{i(k+\gamma - |\beta|)} \le c \theta(r) (2^l r)^{k + \gamma - |\beta|}.
\end{align*}
This yields for any $R > 1$
\begin{align*}
|a_{j,r}^{(\beta)} - a_{j,R r}^{(\beta)}| \le c \theta(r) (Rr)^{k+\gamma - |\beta|},
\end{align*}
which implies \eqref{eq:Q-difference}.

Thus, combining \eqref{eq:wm-estimate} and \eqref{eq:Q-difference},
\begin{align*}
\Vert w_m \Vert_{L^{\infty}(\Omega_{j_m} \cap B_R )} &\le \frac{ (R r_m)^{k+\gamma} \theta(R r_m) }{r_m^{k+\gamma} \theta(r_m)} + c\frac{(R r_m)^{k+\gamma} \theta(r_m) }{r_m^{k+\gamma} \theta(r_m)} \le c R^{k+\gamma},
\end{align*}
where we used in the last step that $t \mapsto \theta(r)$ is monotone decreasing, proving \eqref{eq:vm-growth}.

Next, using \eqref{eq:vm-growth}, we will estimate the $L^1_{2s+(k+\lceil \gamma - s \rceil -1)}$ norm of $w_m$. 
We have the following estimate:
\begin{align*}
\int_{(\Omega_{j_m} \setminus B_{R r_m}) \cap \{ d_{j_m} \ge \kappa\}} & d_{j_m}^{s-1}(y) |y|^{-n-2s- \lceil \gamma - s \rceil +1+\gamma} \d y \le \kappa^{s-1} \int_{(\Omega_{j_m} \setminus B_{R r_m})} |y|^{-n-s + \gamma - \lceil \gamma - s \rceil} |y|^{1-s} \d y \\
&\le c \kappa^{s-1} \diam(\Omega_{j_m})^{1-s} \int_{\R^n \setminus B_{R r_m}} |y|^{-n-s + \gamma - \lceil \gamma - s \rceil} \le c (R r_m)^{\gamma -s - \lceil \gamma - s \rceil},
\end{align*}
where we used that always $\gamma < s + \lceil \gamma - s \rceil < 0$. Moreover, by a similar computation as in \autoref{lemma:dist-int-est} (with $\gamma := s-1 < 2s + \lceil \gamma - s \rceil -1-\gamma =: \beta$), we have
\begin{align*}
\int_{(\Omega_{j_m} \setminus B_{Rr_m}) \cap \{ d_{j_m} < \kappa \}} d_{j_m}^{s-1}(y) |y|^{-n-2s-\lceil \gamma - s \rceil +1+\gamma} \d y \le c (R r_m)^{\gamma -s-\lceil \gamma - s \rceil }. 
\end{align*}
Thus, altogether, using \eqref{eq:vm-growth} and $\gamma \in (0,1)$ we obtain:
\begin{align}
\label{eq:wm-weighted-l1}
\begin{split}
\Vert \tilde{d}_{j_m}^{s-1} & w_m |\cdot|^{-n-2s-(k+\lceil \gamma - s \rceil -1)} \Vert_{L^1(\R^n \setminus B_R)} \\
&\le c \int_{r_m^{-1} \Omega_{j_m} \setminus B_R} \tilde{d}_{j_m}^{s-1}(y) |y|^{-n-2s-\lceil \gamma - s \rceil +1+\gamma} \d y \\
&\le c r_m^{1-s} \int_{r_m^{-1} \Omega_{j_m} \setminus B_R} d_{j_m}^{s-1}(r_m y) |y|^{-n-2s-\lceil \gamma - s \rceil +1+\gamma} \d y \\
&= c r_m^{s-\gamma+\lceil \gamma - s \rceil } \int_{\Omega_{j_m} \setminus B_{R r_m}} d_{j_m}^{s-1}(y) |y|^{-n-2s-\lceil \gamma - s \rceil +1+\gamma} \d y \\
& \le c r_m^{s-\gamma+\lceil \gamma - s \rceil } \int_{(\Omega_{j_m} \setminus B_R) \cap \{ d_{j_m} \ge \kappa \}} d_{j_m}^{s-1}(y) |y|^{-n-2s-\lceil \gamma - s \rceil +1+\gamma} \d y \\
&\quad + c r_m^{s-\gamma+\lceil \gamma - s \rceil } \int_{(\Omega_{j_m} \setminus B_R) \cap \{ d_{j_m} < \kappa \}} d_{j_m}^{s-1}(y) |y|^{-n-2s-\lceil \gamma - s \rceil +1+\gamma} \d y \\
&\le c r_m^{s-\gamma+\lceil \gamma - s \rceil } (Rr_m)^{\gamma - s-\lceil \gamma - s \rceil } \le R^{\gamma -s-\lceil \gamma - s \rceil } \to 0 ~~ \text{ as } R \to \infty,
\end{split}
\end{align}

Now, we investigate the equation that is satisfied by $w_m$.
We claim that 
\begin{align}
\label{eq:polynomial-coefficients-decay}
\frac{|a_{j_m,r_m}^{(\beta)}|}{\theta(r_m)} \to 0, ~~ \text{ as } m \to \infty ~~ \forall  |\beta| \le k.
\end{align}

Indeed, from the considerations above, we deduce that for any $m,l \in \N$
\begin{align*}
\frac{|a_{j_m,r_m}^{(\beta)} - a_{j_m,2^{l}r_m}^{(\beta)}| }{\theta(r_m)} \le c\sum_{i = 1}^{l} \frac{\theta(2^{l-i}r_m)}{\theta(r_m)} (2^{l-i}r_m)^{k+\gamma - |\beta|}.
\end{align*}
Hence, choosing $l \in \N$ such that $2^l r_m \in [1,2)$, we deduce
\begin{align*}
\frac{|a_{j_m,r_m}^{(\beta)}|}{\theta(r_m)} &\le \frac{|a_{j_m,2^l r_m}^{(\beta)}|}{\theta(r_m)}  + \frac{|a_{j_m,r_m}^{(\beta)} - a_{j_m,2^l r_m}^{(\beta)}|}{\theta(r_m)} \\
&\le c \theta(r_m)^{-1} \left( |a_{j_m,2^l r_m}^{(\beta)}| +  \sum_{i = 1}^l \theta(2^{-i})  (2^{-i})^{k+\gamma - |\beta|} \right) \to 0 ~~ \text{ as } m \to \infty,
\end{align*}
which implies \eqref{eq:polynomial-coefficients-decay}.

Let us now distinguish between the cases ($k = 1$ and $\gamma < s$) and ($k \ge 2$ or $\gamma > s$). In case $k = 1$ and $\gamma < s$, we find that it holds in the viscosity sense
\begin{align}
\label{eq:L-wm-computation}
\begin{split}
\tilde{d}_{j_m}^{s-\gamma} L_{j_m} (\tilde{d}_{j_m}^{s-1} w_m) &= r^{\gamma -s}d_{j_m}^{s-\gamma}(r_m \cdot) r_m^{1-s} L_{j_m} (d_{j_m}^{s-1}(r_m \cdot) w_m)\\
&= d_{j_m}^{s-\gamma}(r_m \cdot) \frac{ L_{j_m}(d_{j_m}^{s-1} u_{j_m}(r_m \cdot)) - L_{j_m}(d_{j_m}^{s-1} Q_{j_m,r_m} (r_m \cdot)) }{r_m^{2s} \theta(r_m)} \\
&= d_{j_m}^{s-\gamma}(r_m \cdot)  \frac{ f_{j_m}(r_m \cdot) - L_{j_m} (d_{j_m}^{s-1} Q_{j_m,r_m})(r_m \cdot) }{\theta(r_m)} ~~ \text{ in } r_m^{-1} \Omega_{j_m} \cap B_{r_m^{-1}}.
\end{split}
\end{align}
Moreover, by \autoref{cor:dist-s-1-smooth-domains}(i), it holds
\begin{align}
\label{eq:application-L-poly-1}
\begin{split}
\Vert d_{j_m}^{s - \gamma} & L_{j_m} (d_{j_m}^{s-1} Q_{j_m,r_m})(r_m \cdot) \Vert_{L^{\infty} (r_m^{-1}\Omega_{j_m} \cap B_{r_m^{-1}})} \\
&= \Vert d_{j_m}^{s-\gamma} L_{j_m} (d_{j_m}^{s-1} Q_{j_m,r_m}) \Vert_{L^{\infty} ( \Omega_{j_m} \cap B_{1})} \le c \sum_{|\beta| \le 1} |a_{j_m,r_m}^{(\beta)}|.
\end{split}
\end{align}
Therefore, recalling $\Vert d_{j_m}^{s-\gamma} f_{j_m} \Vert_{L^{\infty}(\Omega_{j_m} \cap B_1 )} \le C$, and combining \eqref{eq:L-wm-computation}, \eqref{eq:application-L-poly-1}, and \eqref{eq:polynomial-coefficients-decay}, we obtain
\begin{align}
\label{eq:blow-up-f-convergence-0}
\Vert \tilde{d}_{j_m}^{s - \gamma} L_{j_m} (\tilde{d}_{j_m}^{s-1} w_m) \Vert_{L^{\infty}(r_m^{-1} \Omega_{j_m} \cap B_{r_m^{-1}} )} &\le c\frac{ \Vert d_{j_m}^{s-\gamma} f_{j_m} \Vert_{L^{\infty}(\Omega_{j_m} \cap B_1 ) } +  \sum_{|\beta| \le 1} |a_{j_m,r_m}^{(\beta)}| }{\theta(r_m)} \to 0 ~~ \text{ as } m \to \infty.
\end{align}
In case $k \ge 2$ or $\gamma > s$, we first deduce by an argument analogous to \eqref{eq:L-wm-computation} 
\begin{align}
\label{eq:L-wm-computation-2}
L_{j_m} (\tilde{d}_{j_m}^{s-1} w_m) = \frac{ f_{j_m}(r_m \cdot) - L_{j_m} (d_{j_m}^{s-1} Q_{j_m,r_m})(r_m \cdot) }{ r_m^{(k-1) - s + \gamma} \theta(r_m)} ~~ \text{ in } r_m^{-1} \Omega_{j_m} \cap B_{r_m^{-1}}.
\end{align}
Next, using again \autoref{cor:dist-s-1-smooth-domains}(ii), we obtain
\begin{align}
\label{eq:application-L-poly-2}
\begin{split}
r^{-(k-1)+s-\gamma} [ L_{j_m} (d_{j_m}^{s-1} & Q_{j_m,r_m})(r_m \cdot) ]_{C^{k-1-s+\gamma} (r_m^{-1}\Omega_{j_m} \cap B_{r_m^{-1}})} \\
&= [L_{j_m} (d_{j_m}^{s-1} Q_{j_m,r_m}) ]_{C^{k-1-s+\gamma} ( \Omega_{j_m} \cap B_{1})} \le c \sum_{|\beta| \le k} |a_{j_m,r_m}^{(\beta)}|,
\end{split}
\end{align}
in analogy to \eqref{eq:application-L-poly-1}. Finally, recalling 
\begin{align*}
r^{-(k-1)+s-\gamma} [ f_j(r_m \cdot) ]_{C^{(k-1)-s+\gamma}(r_m^{-1}\Omega_{j_m} \cap B_{r_m^{-1}})} = [ f_j ]_{C^{(k-1)-s+\gamma}(\Omega_{j_m} \cap B_{1})} \le C,
\end{align*}
and combining  \eqref{eq:L-wm-computation-2}, \eqref{eq:application-L-poly-2}, and \eqref{eq:polynomial-coefficients-decay}, we obtain
\begin{align*}
[L_{j_m} (\tilde{d}_{j_m}^{s-1} w_m) ]_{C^{k-1-s+\gamma}(r_m^{-1} \Omega_{j_m} \cap B_{r_m^{-1}} )} &\le c\frac{ [ f_j ]_{C^{(k-1)-s+\gamma}(\Omega_{j_m} \cap B_{1})} +  \sum_{|\beta| \le 1} |a_{j_m,r_m}^{(\beta)}| }{\theta(r_m)} \to 0 ~~ \text{ as } m \to \infty.
\end{align*}
Thus, there exists a polynomial $p_m \in \cP_{k-2+\lceil \gamma - s \rceil}$ such that
\begin{align}
\label{eq:blow-up-f-convergence}
|L_{j_m} (\tilde{d}_{j_m}^{s-1} w_m) - p_m| \to 0, ~~ \text{ as } m \to \infty ~~ \text{ in } L^{\infty}_{loc}(\{ x_n > 0 \}).
\end{align}

Next, considering again all values for $\gamma,k$ at the same time, we treat the Neumann boundary condition:
\begin{align*}
\partial_{\nu} w_m &= \frac{\partial_{\nu} u_{j_m}(r_m \cdot) - \partial_{\nu} (Q_{j_m,r_m})(r_m \cdot)}{r_m^{(k-1) + \gamma} \theta(r_m)} = \frac{g_{j_m}(r_m \cdot) -   \partial_{\nu} (Q_{j_m,r_m})(r_m \cdot)}{r_m^{(k-1) + \gamma} \theta(r_m)} ~~ \text{ on } \partial r_m^{-1}\Omega_{j_m} \cap B_{r_m^{-1}}.
\end{align*}
We obtain
\begin{align*}
r_m^{-(k-1)-\gamma} [\partial_{\nu}Q_{j_m,r_m}(r_m \cdot)]_{C^{(k-1)+\gamma}(\partial r_m^{-1} \Omega_{j_m} \cap B_{r_m^{-1}})} = [\partial_{\nu} Q_{j_m,r_m}]_{C^{(k-1)+\gamma}(\partial \Omega_{j_m} \cap B_{1})} \le c \sum_{|\beta| \le k} |a_{j_m,r_m}^{(\beta)}|,
\end{align*}
and using also that  $g_{j_m} \in C^{k-1 + \gamma}(\Omega_{j_m} \cap B_1)$ by the boundary condition, we deduce 
\begin{align*}
[\partial_{\nu} w_m]_{C^{(k-1)+\gamma}(\partial r_m^{-1} \Omega_{j_m} \cap B_{r_m^{-1}})} \le c \frac{[g_{j_m}]_{C^{(k-1)+\gamma}(\partial \Omega_{j_m} \cap B_{1})} + \sum_{|\beta| \le k} |a_{j_m,r_m}^{(\beta)}| }{\theta(r_m)} \le c \theta(r_m)^{-1} \to 0,
\end{align*}
as $m \to \infty$. Consequently, for any $m \in \N$ there exists a polynomial $q_m \in \cP_{k-1}$ such that 
\begin{align}
\label{eq:blow-up-g-convergence}
|\partial_{\nu} w_m(x) - q_m | \le c \frac{|x|^{\gamma}}{\theta(r_m)} \to 0 ~~ \forall x \in \partial r_m^{-1}\Omega_{j_m} \cap B_{r_m^{-1}}.
\end{align}
Finally, we are in a position to apply the stability theorem (see \autoref{lemma:stability}) to $w_m$. The convergence results in \eqref{eq:blow-up-f-convergence-0}, \eqref{eq:blow-up-f-convergence} and \eqref{eq:blow-up-g-convergence} establish the required convergence of the source terms and the Neumann boundary data.\\ Moreover, the operators $L_{j_m}$ converge to an operator $L$ with the same ellipticity constants.
By the boundary H\"older regularity estimate for solutions to the nonlocal Neumann problem (see \autoref{cor:bdry-Holder} applied with $k:=k+\lceil \gamma - s \rceil$, $\delta := 0$, $\Omega:= r_m^{-1}\Omega_{j_m}$, $v := \tilde{d}_{j_m}^{s-1} w_m$, $f:= L_{j_m}(\tilde{d}_{j_m}^{s-1} w_m)$, and $g := \partial_{\nu} w_m$), together with the Arzel\`a-Ascoli theorem, the sequence $(w_m)_m$ converges in $L^{\infty}_{loc}(\{ x_n \ge 0 \} )$ to some $w \in C(\{ x_n \ge 0\} )$. Note that all the quantities on the right hand side of the estimate in \autoref{cor:bdry-Holder} will be bounded uniformly in $k$, due to \eqref{eq:wm-weighted-l1}, \eqref{eq:blow-up-f-convergence-0}, \eqref{eq:blow-up-f-convergence}, and \eqref{eq:blow-up-g-convergence}. Thus, in particular $\tilde{d}_{j_m}^{s-1}(r_m \cdot) w_m \to (x_n)_+^{s-1}w$ locally uniformly in $\{ x_n > 0 \}$. Finally, in order to apply the stability result in \autoref{lemma:stability}, it remains to establish $\tilde{d}_{j_m}^{s-1}(r_m \cdot)w_m \to (x_n)_+^{s-1}w$ in $L^1_{2s + (k+\lceil \gamma - s \rceil-1)}(\R^n)$. To see this, we also observe that by \eqref{eq:vm-growth},
\begin{align}
\label{eq:growth-limit}
|w(x)| &\le C (1 + |x|)^{k + \gamma} ~~ \forall x \in \{ x_n > 0 \}.
\end{align} 
Therefore, using also \eqref{eq:wm-weighted-l1} and a computation based on polar coordinates (along the lines of \eqref{eq:tail-polar-coord}) we obtain since $\gamma < 1$:
\begin{align*}
\int_{\R^n \setminus B_R} & |\tilde{d}_{j_m}^{s-1}(y) w_m(y) - (y_n)_+^{s-1} w(y)| |y|^{-n-2s -(k+\lceil \gamma - s \rceil-1)} \d y \\
&\le C \int_{\R^n \setminus B_R} (y_n)_+^{s-1} |y|^{-n-2s-\lceil \gamma - s \rceil+1+\gamma} \d y + C \int_{(r_m^{-1} \Omega_{j_m}) \setminus B_R} \tilde{d}_{j_m}^{s-1}(y) |y|^{-n-2s-\lceil \gamma - s \rceil+1+\gamma} \d y\\
&\le C R^{\gamma - s-\lceil \gamma - s \rceil} \to 0 ~~ \text{ as } R \to \infty.
\end{align*}

This implies $\tilde{d}_{j_m}^{s-1} w_m \to (x_n)_+^{s-1}w$ in $L^1_{2s+(k+\lceil \gamma - s \rceil-1)}(\R^n)$, by combining it with the locally uniform convergence in $L^{\infty}_{loc}(\{ x_n \ge 0 \})$.

Thus, by stability of viscosity solutions (see \autoref{lemma:stability}), we deduce that in the viscosity sense 
\begin{align*}
\begin{cases}
L((x_n)_+^{s-1} w) &\overset{k-1+\lceil \gamma -s \rceil}{=} 0 ~~ \qquad \text{ in } \{ x_n > 0\},\\
\partial_n w &\quad ~~= p ~~\qquad\quad ~~ \text{ on } \{ x_n = 0 \},
\end{cases}
\end{align*}
where $p \in \cP_{k-1}$ is a polynomial, and moreover, by \eqref{eq:orthogonality}, it must be
\begin{align}
\label{eq:v-nondeg}
\Vert w \Vert_{L^{\infty}(B_1 \cap \{ x_n > 0 \} )} \ge \frac{1}{2}.
\end{align}
An application of the Liouville theorem (see \autoref{thm:Liouville}, using \eqref{eq:growth-limit}) yields now that $w \in \cP_{k}$. Thus, we can choose $Q(x) = w(r_m^{-1}x)$ in \eqref{eq:orthogonality}. This implies that
\begin{align*}
0 = \lim_{m \to \infty} \int_{B_1 \cap r_m^{-1} \Omega_{j_m} } w_m(x) Q(r_m x) \d x = \lim_{m \to \infty} \int_{B_1 \cap r_m^{-1} \Omega_{j_m} } w_m(x)w(x) \d x = \int_{B_1 \cap \{ x_n > 0 \} } w^2(x) \d x,
\end{align*}
where we used in the last step $w_m \to w$ and $r_m^{-1} \Omega_{j_m} \to \{ x_n > 0\}$. This yields $w = 0$, which however contradicts \eqref{eq:v-nondeg}, and thus, we conclude the proof of (ii). \\
Finally, note that if $k = 1$ and $\gamma < s$, then by the Liouville theorem (see \autoref{thm:Liouville}), there exist $a \in \R^n$ and $b \in \R$ such that
\begin{align*}
w(x) = (a,x) + b.
\end{align*}
Moreover, if $g_j \equiv 0$, then also $g \equiv 0$. Thus, it must be $\partial_n w = 0$ in $\{ x_n = 0\}$, which implies $a_n = 0$.
\end{proof}

We are now in a position to prove our main result.

\begin{proof}[Proof of \autoref{thm:higher-reg}]
We define $u:=v/d^{s-1}$. Let us assume that
\begin{align*}
\Vert u \Vert_{L^{\infty}(\R^n)} + \Vert d^{s-\gamma} f \Vert_{L^{\infty}(\Omega \cap B_2)} \1_{\{ 1+\gamma < 2s\}} + \Vert f \Vert_{C^{k-2s+\gamma}(\Omega \cap B_2)} \1_{\{ 1+\gamma > 2s\}} + \Vert g \Vert_{C^{k-1+\gamma}(\partial \Omega \cap B_2)} \le 1.
\end{align*}
First, we claim that
for any $x_0 \in \Omega \cap B_{1/2}$ with $z \in \partial \Omega \cap B_{1/2}$ such that $|x_0 - z| = d(x_0) =: r \le 1$, there exists a polynomial $Q \in \cP_k$ of degree $k$ such that
\begin{align}
\label{eq:expansion-consequence}
[u - Q ]_{C^{k+\gamma}(B_{r/2}(x_0))} \le c
\end{align}
for some constant $c > 0$, depending only on $n,s,\lambda,\Lambda,\gamma,\Omega,k$, where we assume without loss of generality that $\nu_z = e_n$.
Note that this estimate already yields the desired result since it implies
\begin{align*}
[u]_{C^{k+\gamma}(B_{r/2}(x_0))} \le [u - Q]_{C^{k+\gamma}(B_{r/2}(x_0))} + [Q ]_{C^{k+\gamma}(B_{r/2}(x_0))} \le c.
\end{align*}
From here, a covering argument (see \cite[Lemma A.1.4]{FeRo23}) together with H\"older interpolation (recall that $\Vert u \Vert_{L^{\infty}(\R^n)} \le 1$) yields the desired regularity estimate in $\overline{\Omega} \cap B_1$. Note that improving the global $L^{\infty}$ norm to the $L^1_{2s}(\R^n)$ norm, or the $L^1_{k+\gamma}(\R^n \setminus B_2)$ norm, respectively, in the estimate goes by the exact same arguments as in the proofs of \autoref{lemma:interior-regularity} and \autoref{cor:bdry-Holder}.

To see \eqref{eq:expansion-consequence}, let us take $z \in \partial \Omega \cap B_{1/2}$ such that $|x_0 - z| = d(x_0) = r$, and apply \autoref{lemma:expansion} to see that there exists a polynomial $Q \in \cP_k$ such that the function
\begin{align*}
u_r(x) := \frac{u(x_0 + r x) - Q(x_0 + r x)}{r^{k+\gamma}} ~~ \text{ satisfies } ~~ \Vert u_r \Vert_{L^{\infty}(B_R)} \le C R^{k+\gamma} ~~ \forall R \in [1,r^{-1}].
\end{align*}
Moreover, since $\Vert u \Vert_{L^{\infty}(\R^n)} \le 1$, and $Q \in \mathcal{P}_k$, we deduce
\begin{align*}
\Vert u_r \Vert_{L^{\infty}(B_R)} \le C r^{-k-\gamma}(1 + (rR)^k) \le C R^{k+\gamma} ~~ \forall R \ge r^{-1}.
\end{align*}
Together, this implies
\begin{align*}
\int_{\R^n \setminus B_{3/4}} \hspace{-0.4cm} \frac{d^{s-1}(x_0 + r x) |u_r(x)|}{|x|^{n+k+\gamma}} \d x &\le c \int_{\R^n \setminus B_{3/4}} \hspace{-0.4cm} \frac{d^{s-1}(x_0 + r x)}{|x|^{n}} \d x \\
&= c \int_{\R^n \setminus B_{3r/4}} \frac{d^{s-1}(x_0 + x)}{|x|^{n}} \d x \le c (1 + r^{s-1}),
\end{align*}
where we used \autoref{lemma:dist-int-est} with $\gamma := s-1 < 0 =: \beta$. Moreover, we have by the definition of $r$
\begin{align*}
\Vert d^{s-1}(x_0 + r \cdot) u_r \Vert_{L^{\infty}(B_{3/4})} \le c r^{s-1}.
\end{align*}

Now, we apply the interior regularity theory for nonlocal problems (see \autoref{lemma:interior-regularity}). To do so, we distinguish between the cases (i) $k = 1$ and $ k + \gamma = 1 + \gamma \le 2s$ and (ii) $k+\gamma > 2s$. In case (i), we apply \autoref{lemma:interior-regularity}(i) with $\beta = 1+\gamma$, observe that automatically $\gamma < s$, and obtain
\begin{align*}
[d^{s-1} (u - Q)]_{C^{1+\gamma}(B_{r/2}(x_0))} &= [d^{s-1}(x_0 + r \cdot) u_r]_{C^{1+\gamma}(B_{1/2})} \\
&\le c \left\Vert d^{s-1}(x_0 + r \cdot) u_r \right\Vert_{L^{\infty}(B_{3/4})} + c \left\Vert \frac{d^{s-1}(x_0 + r \cdot) u_r}{|x|^{n + 1 + \gamma}} \right\Vert_{L^{1}(\R^n \setminus B_{3/4})} \\
&\quad + c \Vert L(d^{s-1}(x_0 + r \cdot) u_r)\Vert_{L^{\infty}(B_{3/4})} \\
&\le c r^{s-1} + c r^{2s-(1+\gamma)} \Vert L v \Vert_{L^{\infty}(B_{3r/4}(x_0))} \\
&\quad + c r^{2s-(1+\gamma)} \Vert L(d^{s-1} Q)\Vert_{L^{\infty}(B_{3r/4}(x_0))} \\
&\le c r^{s-1} + c r^{s-1}\Vert d^{s-\gamma} f \Vert_{L^{\infty}(B_r(x_0))} + c r^{s - 1} \le c r^{s-1},
\end{align*} 
where we used \autoref{cor:dist-s-1-smooth-domains}(i) and that $d \ge r/4$ in $B_{3r/4}(x_0)$ by construction, and $r \le 1$.\\
In case (ii), we apply \autoref{lemma:interior-regularity}(ii) with $\alpha := k+\gamma -2s > 0$ and obtain
\begin{align*}
[d^{s-1} (u - Q)]_{C^{k+\gamma}(B_{r/2}(x_0))} &= [d^{s-1}(x_0 + r \cdot) u_r]_{C^{k+\gamma}(B_{1/2})} \\
&\le c \left\Vert d^{s-1}(x_0 + r \cdot) u_r \right\Vert_{L^{\infty}(B_{3/4})} + c \left\Vert \frac{d^{s-1}(x_0 + r \cdot) u_r}{|x|^{n + k +\gamma}} \right\Vert_{L^{1}(\R^n \setminus B_{3/4})} \\
&\quad + c [L(d^{s-1}(x_0 + r \cdot) u_r)]_{C^{k+\gamma -2s}(B_{3/4})} \\
&\le c r^{s-1} + c [L v]_{C^{k+\gamma -2 s}(B_{3r/4}(x_0))} + c [L(d^{s-1} Q)]_{C^{k+\gamma -2 s}(B_{3r/4}(x_0))} \\
&\le c r^{s-1} + c r^{s-1} \Vert f \Vert_{C^{k+\gamma -2s}(B_r(x_0))} + c r^{s-1} \le c r^{s-1},
\end{align*}
where we used \autoref{cor:dist-s-1-smooth-domains}(iii) in the second to last step. 

Moreover, using again the $L^{\infty}$ estimate for $u_r$ with $R = 1$, we have
\begin{align*}
\Vert d^{s-1}(u - Q) \Vert_{L^{\infty}(B_{r/2}(x_0))} \le c r^{s-1} \Vert u - Q \Vert_{L^{\infty}(B_{r/2}(x_0))} \le c r^{s + (k-1) + \gamma} \Vert u_r \Vert_{L^{\infty}(B_1)} \le c r^{s+ (k-1) + \gamma},
\end{align*}
and hence by H\"older interpolation, we obtain that for any $\delta \in (0,k+\gamma)$ it holds
\begin{align*}
[d^{s-1} (u - Q)]_{C^{\delta}(B_{r/2}(x_0))} \le c r^{s-1 + k+\gamma - \delta}.
\end{align*}
Therefore, altogether by the product rule
\begin{align}
\label{eq:product-rule-comp}
\begin{split}
[(u - & Q)]_{C^{k+\gamma}(B_{r/2}(x_0))} = [D^k(d^{1-s} d^{s-1}(u - Q)) ]_{C^{\gamma}(B_{r/2}(x_0))} \\
&\le \sum_{|\beta| = k} \sum_{\alpha \le \beta} [(\partial^{\alpha} d^{1-s})(\partial^{\beta - \alpha} d^{s-1} (u - Q) ) ]_{C^{\gamma}(B_{r/2}(x_0))} \\
&\le \sum_{|\beta| = k} \sum_{\alpha \le \beta} \Big( \Vert\partial^{\alpha} d^{1-s}\Vert_{L^{\infty}(B_{r/2}(x_0))}[\partial^{\beta - \alpha} d^{s-1} (u - Q) ]_{C^{\gamma}(B_{r/2}(x_0))} \\
&\qquad\qquad\qquad + \Vert\partial^{\beta - \alpha} d^{s-1} (u - Q) \Vert_{L^{\infty}(B_{r/2}(x_0))} [\partial^{\alpha} d^{1-s}]_{C^{\gamma}(B_{r/2}(x_0))} \Big) \\
&\le c \sum_{|\beta| = k} \sum_{\alpha \le \beta} \Big( r^{1-s-|\alpha|} r^{s-1+k+\gamma-(k-|\alpha|+\gamma)} + r^{s-1+k+\gamma-(k-|\alpha|)} r^{1-s-|\alpha|-\gamma} \Big) \\
&\le c  \sum_{|\beta| = k} \sum_{\alpha \le \beta} \le c,
\end{split}
\end{align}

where we used that $r \le 1$, and the following observation based on the fact that $d \in C^{k+1,\gamma}(\overline{\Omega})$ together with corresponding estimates $|D^j d| \le c_j d^{1-j}$, resp. $|D^j d^{1-s}| \le c_j d^{1-s-j}$ in $\Omega$ for every $j \le k$ (see \cite[Lemma B.0.1]{FeRo23}):
\begin{align}
\label{eq:dist-smooth}
[\partial^{\alpha} d^{1-s}]_{C^{\gamma}(B_{r/2}(x_0))} &\le \Vert D^{|\alpha|+1} d^{1-s} \Vert_{L^{\infty}(B_{r/2}(x_0))} \hspace{-0.1cm} \sup_{x,y \in B_{r/2}(x_0)} \hspace{-0.2cm} |x-y|^{1-\gamma} \le c r^{1-s-|\alpha|-\gamma} ~~ \forall |\alpha| \le k.
\end{align}
This proves our claim \eqref{eq:expansion-consequence}. Note that we can replace the $L^{\infty}$ norm of $u$ in $\R^n \setminus B_2$ by the $L^1_{2s}(\R^n)$ norm via a truncation argument, as in the proof of \autoref{cor:bdry-Holder}. We conclude the proof.
\end{proof}

Finally, we explain how to prove \autoref{thm:higher-reg-half-space}.

\begin{proof}[Proof of \autoref{thm:higher-reg-half-space}]
The result follows immediately from \autoref{thm:higher-reg}, however it remains to prove that the result only requires $K \in C^{k-2s+\gamma}(\mathbb{S}^{n-1})$ if $\Omega = \{ x_n = 0 \}$. First of all, we recall the H\"older estimate (see \autoref{cor:bdry-Holder-half-space}), which holds true without any regularity assumption on $K$ if $k = 0$. Note that for the Liouville theorem (see \autoref{thm:Liouville}), we only require $K \in C^{k-1-s+\gamma+\delta}(\mathbb{S}^{n-1})$ for an arbitrarily small $\delta > 0$ and $k-1-s+\gamma <k-2s+\gamma$.
In \autoref{lemma:expansion}, additional regularity for $K$ is assumed in order to apply \autoref{cor:dist-s-1-smooth-domains}. However, if $\Omega = \{ x_n > 0 \}$, we have
\begin{align*}
L(d^{s-1} Q ) = L((x_n)_+^{s-1} Q) \overset{k-1}{=} 0 ~~ \text{ in } \{ x_n > 0 \}
\end{align*}
for any $Q \in \cP_k$. Hence, in case $k = 1$ and $\gamma < s$, the proof goes through exactly as before, without any restrictions on $K$. If $k \ge 2$ or $\gamma > s$, \eqref{eq:L-wm-computation-2} needs to be interpreted as an equation up to a polynomial, but the rest of the proof remains the same. Moreover, we apply the H\"older estimate (see \autoref{cor:bdry-Holder-half-space}), which would force us to impose $K \in C^{k-1}(\mathbb{S}^{n-1})$ in case $k \ge 2$ or $\gamma > s$. Therefore, in this case, we need to proceed a little different. Indeed, we replace the computation in \eqref{eq:wm-weighted-l1} by the following estimate, based on polar coordinates (see also \eqref{eq:tail-polar-coord}) for $\eta = 1+\gamma-s - \lceil \gamma - s \rceil +\delta$ for some very small $\delta > 0$:
\begin{align*}
\Vert (x_n)_+^{s-1} w_m & |\cdot|^{-n-2s-(k-2 +\lceil \gamma - s \rceil + \eta)} \Vert_{L^1(\R^n \setminus B_R)} \\
& \le c \int_{\R^n \setminus B_R} (x_n)_+^{s-1} |x|^{-n-2s-\lceil \gamma - s \rceil+2+\gamma+\eta} \d x \\
&\le c \int_{0}^{2\pi} \hspace{-0.2cm} \cos(\theta)_+^{s-1} \left( \int_R^{\infty} \hspace{-0.2cm} r^{s-1} r^{-1-2s-\lceil \gamma - s \rceil+2+\gamma+\eta} d r \right) \d \theta \\
&= c \int_{0}^{2\pi} \hspace{-0.2cm} \cos(\theta)_+^{s-1} \hspace{-0.1cm} \left( \int_R^{\infty} \hspace{-0.2cm} r^{-1-\delta} d r \right) \d \theta \le c R^{-\delta} \to 0,
\end{align*}
as $R \to \infty$. Then, we can apply \autoref{cor:bdry-Holder-half-space} with $k:= k-1$ and $\delta := \eta$ and only need to assume that $K \in C^{k-2+\eta}(\mathbb{S}^{n-1})$, which is fine by the same reasoning as for the Liouville theorem above. Moreover, the stability theorem (see \autoref{lemma:stability}) can still be applied since $k-2+\eta \le k-1$ if we choose $\delta < s-\eta$.\\
Finally, the proof of \autoref{thm:higher-reg} relies on an application of the interior regularity result (see \autoref{lemma:interior-regularity}). In case $k = 1$ and $1+\gamma \le 2s$, we apply \autoref{lemma:interior-regularity}(i), so in this case, no regularity assumption on $K$ is required, at all. In case $1+\gamma > 2s$, we apply \autoref{lemma:interior-regularity}(ii) with $\alpha := k+\gamma - 2s$ (and interpret the equation up to a polynomial of degree $k-1$, which is possible due to \autoref{remark:int-up-to-poly}), so in this case, we need to assume only that $K \in C^{k-2s+\gamma}(\mathbb{S}^{n-1})$, as desired.
\end{proof}

\section{Nonlocal equations with local Dirichlet boundary conditions}
\label{sec:dirichlet}

Finally, we give the proof of the boundary regularity for nonlocal equations with Dirichlet boundary conditions (see \autoref{thm:dirichlet}).

\begin{proof}[Proof of \autoref{thm:dirichlet}]
Let us first extend $h$ in such a way that $h \in C^{k+\gamma}(\R^n)$. Then, we define $w := v - d^{s-1} h$ and observe that $w$ solves
\begin{align*}
\begin{cases}
L w &= \tilde{f} ~~ \text{ in } \Omega,\\
w &= 0 ~~ \text{ in } \R^n \setminus \Omega,\\
w/d^{s-1} &= 0 ~~ \text{ on }  \partial \Omega, 
\end{cases}
\end{align*}
where $\tilde{f} := f - L(d^{s-1}h)$. Moreover, for $x_0 \in \Omega$ an application of \autoref{cor:dist-s-1-smooth-domains} yields $|\tilde{f}(x_0)| \le c_1 d^{\gamma-s}(x_0)$ in case $k+\gamma < 1+s$, as well as $[\tilde{f}]_{C^{k-1-s+\gamma}(\overline{\Omega})} \le c_2$ in case $k+\gamma > 1+s$, and also $[\tilde{f}]_{C^{k-2s+\gamma}(B_{d(x_0)/2}(x_0))} \le c_3 d^{s-1}(x_0)$ in case $k+\gamma > 2s$. Note that since $w/d^{s-1} = 0$ on $\partial \Omega$, by the maximum principle (see \autoref{lemma:weak-max-princ-large}) and a barrier argument (see for instance the proof of \cite[Lemma 2.3.9]{FeRo23}, using the barrier from \cite[Lemma 2.3.10]{FeRo23} in case $k+\gamma > 1+s$ and the barrier $\tilde{\psi}$ from the second claim in \autoref{lemma:supersol} in case $k+\gamma < 1+s$) it holds $w \in L^{\infty}(\Omega)$ and
\begin{align}
\label{eq:application-Dirichlet-barrier}
\left\Vert w\right\Vert_{L^{\infty}(\Omega)} \le  C \Vert d^{s-\gamma} \tilde{f} \Vert_{L^{\infty}(\Omega)}.
\end{align}
Thus $w$ is a solution in the setting of \cite{RoSe17,AbRo20}. 
We assume without loss of generality
\begin{align*}
\left\Vert w \right\Vert_{L^{\infty}(\Omega)} +  \Vert d^{s-\gamma} \tilde{f} \Vert_{L^{\infty}(\overline{\Omega})} \1_{ \{ k+\gamma < 1 + s \}}  + \Vert \tilde{f} \Vert_{C^{k-1-s+\gamma}(\Omega)} \1_{\{ k+\gamma > 1+s \}} \le 1.
\end{align*}
Then, \cite{RoSe17,AbRo20} imply that for any $z \in \partial \Omega$ there exists a polynomial $Q_z \in \cP_{k-1}$ such that 
\begin{align*}
|w(x) - Q_z(x) d^s| \le c |x-z|^{k-1+\gamma+s} \le c |x-z|^{k+\gamma} d^{s-1}(x) ~~ \forall x \in B_1(z).
\end{align*}
By adjusting the proof of \cite[Proposition 3.2]{RoSe17} in case $k+\gamma < 1 + s$, or the second part of the proof of \cite[Proposition 4.1]{AbRo20} in case $k+\gamma  > 1 + s$, respectively, according to the slight modification of the upper bound in the previous estimate, we get that for any $x_0 \in \Omega \cap B_1(z)$, denoting $r:= d(x_0)$,
\begin{align}
\label{eq:Dir-expansion}
\Vert w - Q_z d^s \Vert_{L^{\infty}(B_{r/2}(x_0))} \le c r^{k+\gamma+s-1},  \qquad [w - Q_z d^s]_{C^{k+\gamma}(B_{r/2}(x_0))} \le c r^{s-1}.
\end{align}
Indeed, while the first estimate is immediate from the expansion, the second result follows by denoting
\begin{align*}
v_r(x) = r^{-k-\gamma}(u(x_0 + rx) - Q_z(x_0+rx)d^s(x_0+rx)),
\end{align*}
and observing that by the previous estimate and the properties of $\tilde{f}$ it holds
\begin{align*}
\Vert v_r \Vert_{L^{\infty}(B_R)} \le c (1 + r^{s-1}) ~~ \forall R > 0, \qquad 
[\tilde{f}]_{C^{k+\gamma-2s}(B_{r/2}(x_0))} \1_{\{ k + \gamma > 2s \}} \le c r^{s-1}.
\end{align*}
Plugging these findings into the remainder of \cite[Proof of Theorem 1.2]{RoSe17}, \cite[Proof of Theorem 1.4]{AbRo20}, we obtain \eqref{eq:Dir-expansion}. From there we can show, using H\"older interpolation, and also $d \in C^{k+1+\gamma}(\overline{\Omega})$ that for any $\delta \in (0,k+\gamma]$ it holds:
\begin{align*}
[ w - Q_z d^s ]_{C^{\delta}(B_{r/2}(x_0))} \le c r^{k+\gamma+s-1-\delta}, \qquad \Vert d^{1-s} \Vert_{L^{\infty}(B_{r/2}(x_0))} \le c r^{1-s}, \qquad [ d^{1-s}]_{C^{\delta}(B_{r/2}(x_0))} \le c r^{1-s-\delta}.
\end{align*}
Thus, proceeding in a similar way as in the proof of \autoref{thm:higher-reg}, and using \eqref{eq:Dir-expansion} as well as the previous estimate, we obtain
\begin{align*}
\left[\frac{w}{d^{s-1}} - Q_z d \right]_{C^{k+\gamma}(B_{r/2}(x_0))} &= \left[D^k(d^{1-s} (w - Q_z d^s))\right]_{C^{\gamma}(B_{r/2}(x_0))} \\
&\le \sum_{|\beta| = k} \sum_{\alpha \le \beta} \left[(\partial^{\alpha} d^{1-s} ) (\partial^{\beta - \alpha} (w - Q_z d^s) ) \right]_{C^{\gamma}(B_{r/2}(x_0))} \\
&\le \sum_{|\beta| = k} \sum_{\alpha \le \beta} \Big( \left\Vert \partial^{\alpha} d^{1-s} \right\Vert_{L^{\infty}(B_{r/2}(x_0))}  \left[\partial^{\beta - \alpha} (w - Q_z d^s)\right]_{C^{\gamma}(B_{r/2}(x_0))} \\
&\qquad\qquad\qquad + \left[ \partial^{\alpha} d^{1-s} \right]_{C^{\gamma}(B_{r/2}(x_0))}  \left\Vert \partial^{\beta - \alpha} (w - Q_z d^s) \right\Vert_{L^{\infty}(B_{r/2}(x_0))} \Big) \\
&\le c\sum_{|\beta| = k} \sum_{\alpha \le \beta} \Big( r^{1-s-|\alpha|} r^{s-1+|\alpha|} + r^{1-s-|\alpha|-\gamma} r^{\gamma+s-1 + |\alpha|} \Big) \le c.
\end{align*}
From here, by a covering argument, and using the continuity of the extension operator,
\begin{align*}
\left\Vert \frac{v}{d^{s-1}} \right\Vert_{C^{k,\gamma}(\overline{\Omega})} & \le c ( \left\Vert v - d^{s-1}h \right\Vert_{L^{\infty}(\R^n)} + \Vert f \Vert_{C^{k-1-s+\gamma}(\Omega)} + \Vert h \Vert_{C^{k+\gamma}(\partial\Omega)} ) \\
&\le c ( \Vert f \Vert_{C^{k-1-s+\gamma}(\Omega)} + \Vert h \Vert_{C^{k+\gamma}(\partial\Omega)} ).
\end{align*}
\end{proof}

\end{document}